\documentclass[11pt,a4paper]{article}
\usepackage[T1]{fontenc}
\usepackage{graphicx}
\usepackage{indentfirst}
\usepackage{bbm}
\usepackage{amsmath,amssymb,amsthm}
\usepackage{latexsym}
\usepackage{natbib}
\usepackage[margin=1in,a4paper]{geometry}
\usepackage{booktabs, multicol, multirow}
\usepackage{colortbl}
\usepackage{color}
\usepackage{tikz-cd}
\usepackage[colorlinks,citecolor=blue,urlcolor=blue]{hyperref}
\usepackage{mathtools}	
\usepackage{leftidx}
\mathtoolsset{showonlyrefs=true}
\usepackage{subcaption}
\usepackage{enumerate}

\numberwithin{equation}{section}
\newtheorem{proposition}{Proposition}[section]
\newtheorem{corollary}{Corollary}[section]
\newtheorem{lemma}{Lemma}[section]
\newtheorem{assumption}{Assumption}
\newtheorem{theorem}{Theorem}[section]
\theoremstyle{remark}
\newtheorem{remark}{Remark}[section]
\theoremstyle{remark}

\allowdisplaybreaks[1]

\title{Markov Chain Approximation of One-Dimensional Sticky Diffusions}
\author{Christian Meier\thanks{Department of Systems Engineering and Engineering Management, The Chinese University of Hong Kong, Hong Kong SAR. Email: meier@se.cuhk.edu.hk.}\and Lingfei Li\thanks{Corresponding author. Department of Systems Engineering and Engineering Management, The Chinese University of Hong Kong, Hong Kong SAR. Email: lfli@se.cuhk.edu.hk.} \and Gongqiu Zhang\thanks{School of Science and Engineering, The Chinese University of Hong Kong (Shenzhen), China. Email: zhanggongqiu@cuhk.edu.cn.}}

\begin{document}
	
\maketitle
\begin{abstract}
We develop continuous time Markov chain (CTMC) approximation of one-dimensional diffusions with a lower sticky boundary. Approximate solutions to the action of the Feynman-Kac operator associated with a sticky diffusion and first passage probabilities are obtained using matrix exponentials. We show how to compute matrix exponentials efficiently and prove that a carefully designed scheme achieves second order convergence. We also propose a scheme based on CTMC approximation for the simulation of sticky diffusions, for which the Euler scheme may completely fail. The efficiency of our method and its advantages over alternative approaches are illustrated in the context of bond pricing in a sticky short rate model for low interest environment.


\hspace{2.0cm}Sturm-Liouville problem.

\end{abstract}

	\section{Introduction}
	\label{sec:introduction}
	
	We consider a one-dimensional (1D) regular diffusion process $X$ on an interval $\mathbb{S}$ with endpoints $l$ and $r$, where $l$ is a sticky lower boundary and $r\leq\infty$.
	The point $l$ is said to be sticky, if the occupation time of the process at $l$, i.e.,
	\begin{equation*}
	O^l_t\left(X\right)=\int_0^tI\left(X_s=l\right)d\langle X\rangle_s,
	\end{equation*}
	is positive for all $t>0$ if $X_0=l$ and $\langle X\rangle$ is the quadratic variation of $X$ (see e.g., \cite{karlin1981}, Section 15.8, or \cite{revuz2005}, Chapter VI). This type of boundary behavior is first discovered by Feller in 1952 and a historical account is given in \cite{peskir2015}. The focus of this paper will be on the case where the lower boundary is sticky and the method we will develop can be easily adapted to treat models with a sticky upper boundary.
	
	Diffusion processes with sticky boundaries have found applications in various areas. For applications in physics and biology, see e.g.,  \cite{gawedzki2004}, \cite{zaslavsky2005}, \cite{kalda2007}, \cite{bonilla2002}, \cite{bonilla2007} and \cite{parashar2008}. Sticky diffusions also arise as the limiting process of time-changed random walks (\cite{amir1991}) and a class of storage models (\cite{harrison1981}, \cite{yamada1994}). Recently, \cite{nie2017} proposes to model short rates (i.e., instantaneous interest rates) by an Ornstein-Uhlenbeck (OU) diffusion with zero as the sticky lower bound (see also \cite{nie2019}). They show that this model is able to produce yield curves of various shapes observed in the market, in particular, the $S$ shape that occurred after the 2008 financial crisis. Some financial applications of a type of sticky diffusions are discussed in \cite{jiang2019}.
	
	There also exist a number of mathematical studies on sticky diffusions. The sticky Brownian motion is studied in \cite{warren1997}, \cite{bass2014} and \cite{engelbert2014} regarding the existence and uniqueness of weak and strong solutions. \cite{nie2019} establish the existence of a unique weak solution to the sticky OU SDE. Some characterizations of more general sticky diffusions can be found in \cite{ikeda1989}.
	
	This paper is mainly concerned with the computation of general 1D sticky diffusions. For the sticky Brownian motion, its transition density is obtained in closed-form in \cite{davies1994}. We are interested in calculating the action of the Feynman-Kac operator on a payoff function, which is required in many applications, as well as first passage probabilities. For the sticky OU model, \cite{nie2019} obtain semi-analytical solutions expressed as infinite series using the eigenfunction expansion method. The drawback of this approach is that approximating the infinite series can require many terms if the time horizon is not long enough, making it computationally inefficient. Furthermore, this method cannot be applied to general sticky diffusions as their eigenvalues and eigenfunctions are generally unavailable.
	
	The main contribution of this paper is a novel, general and efficient computational method based on CTMC approximation. We offer two constructions of CTMC approximation and show that a carefully designed scheme achieves second order convergence. We compare it with a well-known finite difference scheme that numerically solves the parabolic PDE associated with the Feynman-Kac operator and our method can be much faster for similar levels of accuracy. Another advantage of our method is that it is probabilistic whereas the finite difference method is not. Its probabilistic nature allows us to utilize it to also solve the simulation problem of 1D sticky diffusions, for which the widely used Euler discretization may fail completely due to its inability to simulate the sticky behavior at the boundary. Our solution is simply to simulate from the CTMC that approximates the sticky process. Although this method is biased, the bias can be properly contained and numerical experiment shows that it yields accurate results for the financial application in our analysis.
	
	CTMC approximation has been developed for many types of diffusions and Markov jump processes before for solving various types of problems, but not for sticky diffusions. See \cite{mijatovic2013}, \cite{eriksson2015}, \cite{cai2015}, \cite{song2019}, \cite{cui2017,cui2018,cui2018a}, \cite{li2017b}, \cite{li2018}, \cite{zhang2019,zhang2019c} among others for related literature on CTMC approximation for financial applications. Many CTMC-based algorithms entail computing matrix exponentials. Although there exist well-known stable algorithms such as the scaling and squaring algorithm of \cite{higham2005}, they may not be efficient enough. In this paper, we find that an extrapolation-based numerical ODE method in \cite{feng2008a} works much better than standard algorithms for computing matrix exponentials in our problem. We recommend using it for matrix exponential computation especially when the time horizon is large.
	
	It should be noted that convergence rate analysis of our method is not trivial at all. Utilizing the eigenfunction expansions of the exact and approximate solutions, we reduce the problem to the analysis of approximation errors for the eigenvalues and eigenfunctions of a Sturm-Liouville problem with its eigenparameter appearing in the boundary condition. To the best of our knowledge, our estimates are new for the numerical analysis of this important type of Sturm-Liouville problems, and hence we also contribute to this literature.

	The rest of the paper is organized as follows.
	Section \ref{sec:sticky_diffusion_processes} provides various characterizations of 1D sticky diffusions. We first give a general set of conditions that implies the existence of a unique solution to the sticky SDE. Then, we define the Feynman-Kac semigroup of the sticky process, derive its infinitesimal generator and obtain its eigenfunction expansion.
	In Section \ref{sec:ctmc_approximation}, we construct two CTMC approximation schemes for the sticky diffusion and show how to calculate the quantities of interest under the CTMC. We show that the generator of the CTMC also admits an eigenfunction expansion, which will be employed for convergence rate analysis. Section \ref{sec:convergence_analysis} derives the convergence rate while Section \ref{sec:sticky_ou_process_applications} contains various numerical results for the bond pricing problem under the sticky OU model of \cite{nie2019}. Section \ref{sec:summary} concludes and discusses future research. All proofs are collected in the appendix.
	\section{Characterizations of 1D Sticky Diffusions}
	\label{sec:sticky_diffusion_processes}
	
	The sticky diffusion under analysis will be defined as a weak solution to the following system of stochastic differential equations
	\begin{align}
	dX_t&=\mu\left(X_t\right)I\left(X_t>l\right)dt+\sigma\left(X_t\right)I\left(X_t>l\right)dB_t+\frac{1}{2}dL^{l}_t\left(X\right), \label{eq:general_diffusion_sde} \\
	I\left(X_t=l\right)dt&=\frac{1}{2\rho} dL^{l}_t\left(X\right), \label{eq:general_local_time_sde}
	\end{align}
	with $\rho\in(0,\infty)$ representing the stickiness of $X_t$ at $l$, $I$ is the indicator function and $L^l_t$ is the (right) local time process of $X$ at $l$ defined as
	\begin{equation*}
	L^l_t\left(X\right):=\underset{\varepsilon\searrow 0}{\lim}\ \frac{1}{\varepsilon}\int_0^tI\left(X_s\in\left[l,l+\varepsilon\right]\right)d\langle X\rangle_s\ \text{in probability}.
	\end{equation*}
	
	A weak solution to the above system of SDEs is a pair of adapted processes $(X,B)$ defined on a filtered probability space $(\Omega,\mathcal{F},(\mathcal{F}_t)_{t\geq 0},\mathbb{P})$ where $B$ is a standard BM and both equations are satisfied.
	Uniqueness in law holds if for any two solutions $(X,B)$ and $(X^1,B^1)$, $X$ and $X^1$ have the same law.
	We will say that joint uniqueness in law holds if for the two solutions, the pairs $(X,B)$ and $(X^1,B^1)$ have the same law.
	
	We will first provide a set of conditions for the existence of a unique weak solution in law. Notice that two other types of boundary behavior can be recovered from the sticky case as limits. If $\rho=0$, $l$ becomes an absorbing boundary and if $\rho=\infty$, the process is instantaneously reflected at $l$. We adopt the method of time change that constructs the sticky diffusion process as a time-changed reflected diffusion process with the same drift and diffusion coefficient in the interior of the state space.
	
	\subsection{Existence and Uniqueness of Weak Solutions}
	\label{subsec:introduction_existence_sticky_diffusion}
	We make the following assumption.
	\begin{assumption}
		Assume that for the given drift and volatility coefficients, there is a unique weak solution to \eqref{eq:general_diffusion_sde} and \eqref{eq:general_local_time_sde} in the case of $\rho=\infty$ (i.e., $l$ is a reflecting boundary).
		Furthermore, suppose $\sigma^2(x)>0$ for all $x\in\mathbb{S}$. \label{assumption:continuous_derivative_coefficients}
	\end{assumption}
	The well-posedness of reflected SDEs has been well studied and conditions can be found in e.g., \cite{skorokhod1961}, \cite{watanabe1971}, \cite{tanaka1979}, \cite{lions1984} and \cite{slominski1993}.
	Let $X^1$ denote the diffusion instantaneously reflected at $l$. We apply the time change
	\begin{equation*}
	T_t=\left(t+\frac{1}{2\rho}L^l_t\left(X^1\right)\right)^{-1},
	\end{equation*}
	to $X^1$, which slows it down whenever it hits $l$. This introduces stickiness, which can be thought as slow reflection. This construction enables us to show that there exists a unique weak solution to the sticky SDE system as long as it is so for the reflecting case.
	\begin{theorem}\label{theorem:existence_and_uniqueness}
		Under Assumption \ref{assumption:continuous_derivative_coefficients}, Equations \eqref{eq:general_diffusion_sde} and \eqref{eq:general_local_time_sde} have a jointly unique weak solution for any $\rho\in(0,\infty)$.
	\end{theorem}
	The proof of this theorem can be found in the appendix, which is similar to the proof of Theorem 4.1 in \cite{nie2017}.
	However, that the time change is continuous and strictly increasing is shown differently, as the arguments in \cite{nie2017} are only valid for the OU process and cannot be applied to general diffusions.
	
	The following theorem shows that it is possible to add the local time given in \eqref{eq:general_local_time_sde} as an additional term in the drift.
	\begin{theorem}\label{theorem:strong_markov_property}
		If $\rho\in(0,\infty)$, then \eqref{eq:general_diffusion_sde} and \eqref{eq:general_local_time_sde} are equivalent to the following SDE
		\begin{equation}\label{eq:general_diffusion_one_equation}
		dX_t=\mu\left(X_t\right)I\left(X_t>l\right)dt+\rho I\left(X_t=l\right)dt+\sigma\left(X_t\right)I\left(X_t>l\right)dB_t .
		\end{equation}
		Furthermore, $X$ is a strong Markov process.
	\end{theorem}
	The proof of this theorem is omitted as it is similar to that of Theorem 4.1 and Corollary 4.3 in \cite{nie2017}. Theorem \ref{theorem:existence_and_uniqueness} and Theorem \ref{theorem:strong_markov_property} together show that $X$ is a diffusion process. Its stickiness at $l$ is measured by $\rho$ and the smaller the value of $\rho$, the more $X$ sticks to $l$.

	\subsection{Feynman-Kac Semigroup and Infinitesimal Generator}
	The Feynman-Kac operator $\mathcal{P}_t$ associated with diffusion $X$ is defined as
	\begin{equation}\label{eq:FK}
	\mathcal{P}_tf\left(x\right):=\mathbb{E}_x\left(\exp\left(-\int_0^tk\left(X_u\right)du\right)f\left(X_t\right)\right),\ x\in\mathbb{S},
	\end{equation}
	where the function $k(\cdot)$ is assumed to be nonnegative. The semigroup $(\mathcal{P}_t)_{t\ge0}$ is a strongly continuous contraction semigroup on $B_b(\mathbb{S})$, the space of bounded Borel-measurable functions, endowed with the maximum norm.
	
	Assume that $\mu(\cdot)$, $\sigma(\cdot)$ and $k(\cdot)$ are continuous over $\mathbb{S}$. Using the arguments in \cite{nie2017}, one can easily show that the infinitesimal generator $\mathcal{G}$ of $(\mathcal{P}_t)_{t\ge0}$ acts on
	\begin{equation}
	\mathcal{D}:=\left\{f\in C^2(\mathbb{S})\cap C_b(\mathbb{S}): \mathcal{G}f\in C_b\left(\mathbb{S}\right),\mu\left(l\right)f^\prime\left(l\right)+\frac{1}{2}\sigma^2\left(l\right)f^{\prime\prime}\left(l\right)=\rho f^\prime\left(l\right)\right\}
	\end{equation}
	as follows:
	\begin{equation}
	\mathcal{G}f\left(x\right)=\left(\mu\left(x\right)I\left(x>l\right)+\rho I\left(x=l\right)\right)f^\prime\left(x\right)+\frac{1}{2}\sigma^2\left(x\right)I\left(x>l\right)f^{\prime\prime}\left(x\right)-k\left(x\right)f\left(x\right).\label{eq:infinitesimal_generator}
	\end{equation}
	In particular, for $f\in\mathcal{D}$,
	\begin{equation}
	\mathcal{G}f\left(l\right)=\rho f^\prime\left(l\right)-k\left(l\right)f\left(l\right).\label{eq:infinitesimal_generator-lb}
	\end{equation}
	
	\begin{remark}
		The Feynman-Kac semigroup can also be seen as the transition semigroup of a process $\tilde{X}$ which is killed at rate $k$, i.e. $\tilde{X}$ has the same drift and diffusion coefficients as $X$ but is killed at its lifetime $\tilde{\zeta}=\inf\{t\geq 0:\int_0^tk(X_u)du\geq e\}$ where $e\sim\textrm{Exp}(1)$ is exponentially distributed with rate 1.
		At the lifetime $\tilde{\zeta}$ the process is sent to the cemetery state.
		It can be proved that
		\begin{equation}\label{eq:TS-Kill}
		\mathcal{P}_tf\left(x\right)=\mathbb{E}_x\left(f\left(\tilde{X}_t\right)I\left(t>\tilde{\zeta}\right)\right).
		\end{equation}
		See e.g., \cite{linetsky2008}, Section 1.1 for a detailed explanation.
	\end{remark}
	
	We can write
	\begin{equation}
	\mathcal{P}_tf\left(x\right)=\int_{\mathbb{S}}P\left(t,x,dy\right)f\left(y\right),
	\end{equation}
	where the measure $P(t,x,\cdot)$ is generally a sub-probability measure. It has a point mass at $l$ and a density on $\mathbb{S}/\{l\}$. We can further write
	\begin{equation}\label{eq:ptf-p}
	\mathcal{P}_tf\left(x\right)=P\left(t,x,l\right)f\left(l\right)+\int_\mathbb{S}p\left(t,x,y\right)f\left(y\right)dy.
	\end{equation}
	Note that we simply write $P\left(t,x,l\right)$ for $P\left(t,x,\{l\}\right)$.
	
	\subsection{Eigenfunction Expansion Representations}
	Let $S(\cdot)$ and $M(\cdot)$ be the scale and speed measure of $X$. Using the results in \cite{borodin2002}, Section II, we can obtain that the scale measure has a density, i.e., $S(dx)=s(x)dx$ with
	\begin{equation}
	s\left(x\right)=\exp\left(-\int_l^x\frac{2\mu\left(y\right)}{\sigma^2\left(y\right)}dy\right),
	\end{equation}
	and the speed measure is of a mixed type, with
	\begin{equation} M\left(dx\right)=m\left(x\right)dx+\frac{1}{\rho}\delta_l\left(dx\right)=\frac{2}{\sigma^2\left(x\right)s\left(x\right)}dx+\frac{1}{\rho}\delta_l\left(dx\right),
	\end{equation}
	where $\delta_l(\cdot)$ is the Dirac delta measure at $l$. In future discussions, we will use $M(l)$ instead of $M(\{l\})$ to simplify the notation.

	It is possible to extend $(\mathcal{P}_t)_{t\geq 0}$ to a self-adjoint semigroup on the Hilbert space $L^2(\mathbb{S},M)$ of square integrable functions on $\mathbb{S}$ with respect to $M$.
	The inner product in this space is
	\begin{equation}
	\left(f,g\right)=\int_\mathbb{S}f\left(x\right)g\left(x\right)M\left(dx\right)=\frac{1}{\rho}f\left(l\right)g\left(l\right)+\int_{\mathbb{S}^\circ}\frac{2f\left(x\right)g\left(x\right)}{\sigma^2\left(x\right)s\left(x\right)}dx.
	\end{equation}
	
	The application of spectral methods to study diffusions dates back to \cite{mckean1956}. When the spectrum of the generator of the diffusion is simple and purely discrete, a general spectral expansion reduces to an eigenfunction expansion. We make the following assumption for the existence of an eigenfunction expansion.
	\begin{assumption}\label{assumption:rb}
		The right boundary $r$ is finite, and $X$ is sent to the cemetery state $\partial$ once it reaches $r$. Suppose that $\mu$ and $\sigma$ satisfy the requirements in Assumption \ref{assumption:continuous_derivative_coefficients}. Furthermore, suppose that $\mu\in C^3(\mathbb{S})$, $\sigma^2\in C^4(\mathbb{S})$ and $k\in C^2(\mathbb{S})$ with $k(x)\geq 0$ for all $x\in\mathbb{S}$.
	\end{assumption}
	Under Assumption \ref{assumption:rb}, $X$ lives on $\mathbb{S}=[l,r)\cup\{\partial\}$ and we extend the definition of $k$ and the payoff function $f$ to $\partial$ by setting $k(\partial)=f(\partial)=0$ in \eqref{eq:FK}. Let $u(t,x)=\mathcal{P}_tf(x)$. Theorem 3.2 in \cite{linetsky2008} shows that the spectrum of $\mathcal{G}$ is purely discrete and simple, and we have the following eigenfunction expansions:
	\begin{align}
	P\left(t,x,l\right)&=M\left(l\right)\sum\limits_{k=1}^\infty\exp\left(-\lambda_kt\right)\varphi_k\left(x\right)\varphi_k\left(l\right)\qquad\ \textrm{for}\ x\in\mathbb{S},t>0, \label{eq:P-ee}\\
	p\left(t,x,y\right)&=m\left(y\right)\sum\limits_{k=1}^\infty\exp\left(-\lambda_kt\right)\varphi_k\left(x\right)\varphi_k\left(y\right)\qquad\ \textrm{for}\ x,y\in\mathbb{S},y\neq l,t>0,\label{eq:p-ee}\\
	u(t,x)&=\sum\limits_{k=1}^\infty\left(f,\varphi_k\right)\exp\left(-\lambda_kt\right)\varphi_k\left(x\right)\qquad\qquad\ \ \text{for}\ f\in L^2(\mathbb{S},M),\ x\in\mathbb{S},t>0.
	\end{align}
	where $(\lambda_k,\varphi_k)$ is the $k$-th eigenpair which solves the following Sturm-Liouville problem:
	\begin{align*}
	&\mathcal{G}\varphi\left(x\right)=-\lambda\varphi\left(x\right),\qquad\qquad\qquad\ \textrm{for all}\ x\in\mathbb{S}^\circ,\\
	&\rho\varphi^\prime\left(l\right)-\left(k\left(l\right)-\lambda\right)\varphi\left(l\right)=0, \qquad\varphi\left(r\right)=0.
	\end{align*}
	The eigenfunctions satisfy $\int_l^r\varphi_i(x)\varphi_j(x)M(dx)=\delta_{ij}$ (here $\delta_{ij}$ is the Kronecker delta) and hence form a complete orthonormal basis of $L^2(\mathbb{S},M)$.
	
	\begin{remark}
		Notice that Assumption \ref{assumption:rb} is not necessary for the spectrum to be purely discrete. For some sticky diffusions with infinity as the right boundary, for example the sticky OU process, and sticky diffusions with other types of boundary behavior at finite $r$, the spectrum is also purely discrete and hence there exist eigenfunction expansions. We make Assumption \ref{assumption:rb} for two reasons. First, in our computational method, we need to localize the infinite right boundary to a large finite value to construct a CTMC, so it makes sense to assume the sticky diffusion under analysis has a finite right boundary. The error caused by localization is typically very small. Second, assuming that $r$ is a killing boundary leads to Dirichlet condition for the Sturm-Liouville problem at $r$. This is the only type of boundary condition for $r$ that we will analyze in this paper although other types of boundary conditions (e.g., Neumann) can also be handled using our method.
	\end{remark}
	
	The eigenfunction expansion method has been extensively applied in finance for derivatives pricing. See for example \cite{davydov2003}, \cite{li2013b,li2014,li2015a}, \cite{linetsky2008} among many others. The Sturm-Liouville problem associated with the sticky diffusion is different from those studied in the cited papers in that the eigenparameter also appears in the boundary condition, which makes the problem more difficult. The following proposition provides the asymptotic behavior of the eigenvalues and eigenfunctions of this type of Sturm-Liouville problem.
	\begin{proposition}\label{prop:SLeigen}
		Under Assumption \ref{assumption:rb}, there exist constants $C_1,C_2,C_3>0$ such that for $k=1,2,\ldots$,
		\begin{align}
		C_1k^2\leq \lambda_k&\leq C_2k^2, \label{eq:sturm_liouville_theory_asymptotics_eigenvalues}\\
		\left\Vert\varphi_k^{(i)}\right\Vert_\infty&\leq C_3k^i,\qquad\textrm{for}\ i=0,1,\ldots,4. \label{eq:sturm_liouville_theory_asymptotics_eigenfunctions}
		\end{align}
		Here $\varphi_k^{(i)}(x)$ is the $i$-th order derivative of $\varphi_k(x)$ and $\varphi_k^{(0)}(x)=\varphi_k(x)$.
	\end{proposition}
	
	\section{Continuous Time Markov Chain Approximation}
	\label{sec:ctmc_approximation}
	
	This section develops a general method to compute \eqref{eq:FK} using CTMC approximation. Note that the Feynman-Kac semigroup of $X$ is the transition semigroup of $\tilde{X}$, killed at rate $k(\cdot)$. In the following, we construct a CTMC denoted by $Y$ to imitate the transition behavior of $\tilde{X}$. The CTMC lives on a generally non-uniform grid $\mathbb{S}_n=\{x_0,x_1,\ldots,x_n,x_{n+1}\}$, with $x_0=l$, $x_{n+1}=r$ and $\mathbb{S}_n^\circ=\{x_1,\ldots,x_n\}$ is the grid on $\mathbb{S}^\circ=(l,r)$. We also put $\partial\mathbb{S}_n=\{x_0,x_{n+1}\}$, $\mathbb{S}_n^-=\{x_0\}\cup\mathbb{S}_n^\circ$ and $\mathbb{S}_n^+=\mathbb{S}_n^\circ\cup\{x_{n+1}\}$.
	Define
	\begin{equation}
	x^-=\underset{y<x,y\in\mathbb{S}_n}{\arg\min}\ \left\vert y-x\right\vert,\qquad x^+=\underset{y>x,y\in\mathbb{S}_n}{\arg\max}\ \left\vert y-x\right\vert,
	\end{equation}
	for $x\in\mathbb{S}_n$.
	Then $x^-$ is the grid point to the left of $x$ and $x^+$ is the point to the right.
	We define $x^-_0=x_0$ and $x^+_{n+1}=x_{n+1}$.
	Since the grid can be non-uniform, a distinction of the distance of a grid point to the left and right is necessary.
	Let
	\begin{equation}
	\delta^+x=x^+-x,\quad\delta^-x=x-x^-,\quad\delta x=\frac{1}{2}\left(\delta^+x+\delta^-x\right),
	\end{equation}
	which are the right, left and average grid size at a point $x$. The mesh size is denoted by $h_n=\underset{x\in\mathbb{S}^-_n}{\max}\ \delta^+x$.

	\subsection{Transition Rates of the CTMC}
	To obtain the transition rates of the CTMC that resembles the sticky diffusion $\tilde{X}$, we apply finite difference to discretize its generator given by \eqref{eq:infinitesimal_generator} and \eqref{eq:infinitesimal_generator-lb}. Introduce the following difference operators acting on a general function $g$:
	\begin{align*}
	\nabla^+g\left(x\right)&=\frac{g\left(x^+\right)-g\left(x\right)}{x^+-x}, \qquad\qquad\qquad\quad\ \nabla^-g\left(x\right)=\frac{g\left(x\right)-g\left(x^-\right)}{x-x^-}, \\
	\nabla g\left(x\right)&=\frac{\delta^-x}{2\delta x}\nabla^+g\left(x\right)+\frac{\delta^+x}{2\delta x}\nabla^-g\left(x\right), \qquad\Delta g\left(x\right)=\frac{1}{\delta x}\left(\nabla^+g\left(x\right)-\nabla^-g\left(x\right)\right).
	\end{align*}
	For $x\in\mathbb{S}_n^\circ$, we approximate $\mathcal{G}g(x)$ for $g\in\mathcal{D}$ by approximating $g^{\prime}$ by $\nabla g$ and $g^{\prime\prime}$ by $\Delta g$, which gives us
	\begin{equation}\label{eq:discretized_generator_interior}
	G_ng\left(x\right)=\mu\left(x\right)\nabla g\left(x\right)+\frac{1}{2}\sigma^2\left(x\right)\Delta g\left(x\right)-k\left(x\right)g\left(x\right),\qquad\textrm{for}\ x\in\mathbb{S}_n^\circ. \end{equation}
	For $\mathcal{G}g(x_0)$, we apply $\nabla^+g\left(x_0\right)$ to approximate $g^{\prime}(x_0)$ and obtain
	\begin{equation}\label{eq:discretized_generator_boundary-1}
	G_ng\left(x_0\right)=\rho\nabla^+g\left(x_0\right)-k\left(x_0\right)g\left(x_0\right).
	\end{equation}
	We specify $x_{n+1}$ to be a killing boundary for the CTMC, i.e., the chain is sent to the cemetery state $\partial$ once it reaches $x_{n+1}$. So the state space of the chain is $\{x_0,\cdots,x_n,\partial\}$. Let $\mathbb{G}_n$ be an $(n+1)\times(n+1)$ matrix for the transition rates among states $\{x_0,\cdots,x_n\}$. Then, based on the expression of $G_n$,
	\begin{equation*}
	\mathbb{G}_n=\begin{pmatrix} -\frac{\rho}{\delta^+ x_0}-k\left(x_0\right) & \frac{\rho}{\delta^+ x_0} & 0 & \ldots \\  \mathbb{G}_{n,1,0} & -\mathbb{G}_{n,1,0}-\mathbb{G}_{n,1,2}-k\left(x_1\right) & \mathbb{G}_{n,1,2} & \ldots \\ 0 & \mathbb{G}_{n,2,1} & -\mathbb{G}_{n,2,1}-\mathbb{G}_{n,2,3}-k\left(x_2\right) & \ldots \\ \vdots & \vdots & \vdots & \ddots \end{pmatrix},
	\end{equation*}
	which is a tridiagonal matrix and
	\begin{align}
	\mathbb{G}_{n,i,i-1}&=\frac{-\mu\left(x_i\right)\delta^+x_i+\sigma^2\left(x_i\right)}{2\delta^-x_i\delta x_i},\qquad\ i=1,\ldots,n,\\ \mathbb{G}_{n,i,i+1}&=\frac{\mu\left(x_i\right)\delta^-x_i+\sigma^2\left(x_i\right)}{2\delta^+x_i\delta x_i},\qquad\quad i=1,\ldots,n-1,\\
	\mathbb{G}_{n,0,1}&=\frac{\rho}{\delta^+x_0}.
	\end{align}
	Note that, in our calculation of \eqref{eq:FK}, we do not need the transition rates involving $\partial$. We call this construction as Scheme 1.
	
	It turns out that Scheme 1 only converges at first order. Below we provide a better scheme that improves approximation for the sticky behavior. In Scheme 2, we try to match the expectation and variance of the sticky diffusion in short time given that it starts from $x_0$ with those of the CTMC. Let $\hat{O}_t^{x_0}$ approximate the occupation time of $\tilde{X}$ at the left boundary $l=x_0$ up to time $t$, where $t$ is understood to be very small.
	Then, we have the following system of equations after ignoring higher order terms of $t$:
	\begin{align}
	\mathbb{G}_{n,0,1}\delta^+x_0t&=\mu\left(x_0\right)\left(t-\hat{O}_t^{x_0}\right)+\rho\hat{O}_t^{x_0}, \label{eq:intuition_mean} \\
	\mathbb{G}_{n,0,1}\left(\delta^+x_0\right)^2t&=\sigma^2\left(x_0\right)\left(t-\hat{O}_t^{x_0}\right), \label{eq:intuition_variance}
	\end{align}
	together with the condition that $\mathbb{G}_{n,0,0}+\mathbb{G}_{n,0,1}+k(x_0)=0$. Eq. \eqref{eq:intuition_mean} matches the expected change from $x_0$ in $t$. The LHS is clearly the quantity for the CTMC, ignoring higher order terms in $t$. The RHS can be explained as follows. Up to time $t$, the diffusion $\tilde{X}$ spent $\hat{O}^{x_0}_t$ amount of time at the boundary $x_0$ with drift $\rho$ and $t-\hat{O}^{x_0}_t$ amount of time near $x_0$ with drift approximately equal to $\mu(x_0)$. Adding them up gives the RHS of \eqref{eq:intuition_mean}. Eq.\eqref{eq:intuition_variance} can be interpreted in an analogous way by noting that the change of $\tilde{X}$ does not have any variance while it is on the boundary.
	
	Solving the above equations gives the following solution:
	\begin{align*}
	\mathbb{G}_{n,0,1}&= \frac{\rho}{\delta^+x_0+\frac{\rho-\mu\left(x_0\right)}{\sigma^2\left(x_0\right)}\left(\delta^+ x_0\right)^2},\qquad\mathbb{G}_{n,0,0}=-\mathbb{G}_{n,0,1}-k\left(x_0\right), \\
	\hat{O}_t^{x_0}&=\frac{\sigma^2(x_0) - \mu(x_0) \delta^+ x_0}{\sigma^2(x_0) + (\rho - \mu(x_0)) \delta^+ x_0} t.
	\end{align*}
	The above expression shows that $t-\hat{O}_t^{x_0}=O(\delta^+x_0)$.
	Ignoring this difference in \eqref{eq:intuition_mean} yields the formula of $\mathbb{G}_{n,0,1}$ under Scheme 1. Now it is clear that intuitively Scheme 2 is much better than Scheme 1 as it matches the first two moments while Scheme 1 only cares about the first moment. In Section 4, we will prove that Scheme 2 achieves second order convergence.
	
	The operator $G_n$ acting on $g(x_0)$ can be written for both schemes in a unified way as follows:
	\begin{equation}\label{eq:discretized_generator_boundary}
	G_ng\left(x_0\right)=\rho\beta\nabla^+g\left(x_0\right)-k\left(x_0\right)g\left(x_0\right).
	\end{equation}
	where
	\begin{equation}
	\beta=\begin{cases}
	1,&\quad\text{Scheme 1},\\
	\frac{1}{1+\frac{\rho-\mu\left(x_0\right)}{\sigma^2\left(x_0\right)}\delta^+x_0},&\quad\text{Scheme 2}.
	\end{cases}
	\end{equation}
	
	\begin{remark}
		Using Scheme 2 we can obtain CTMC approximation for the absorbing and instantaneous reflecting cases. If we set $\rho=0$, $\mathbb{G}_{n,0,1}=0$ and $\hat{O}_t^{x_0}=t$, which shows the chain is absorbed at $x_0$. If we let $\rho\to\infty$, then $\mathbb{G}_{n,0,1}\to\sigma^2(x_0)/(\delta^+x_0)^2$ and $\hat{O}_t^{x_0}\to 0$. This shows the chain is reflected at $x_0$.
	\end{remark}
	
	\subsection{CTMC Approximation of the Feynman-Kac Semigroup and First Passage Probabilities}\label{sec:CTMC-FK}
	
	To approximate \eqref{eq:FK}, we use $\mathbb{E}_x\left(f\left(Y_t\right)I\left(t>\zeta_Y\right)\right)$, where $\zeta_Y$ is the lifetime of $Y$.
	This expectation is obtained in closed-form in \cite{mijatovic2013}, which is
	\begin{equation}\label{eq:TS-Kill-CTMC}
	u_n(t,x):=\mathbb{E}_x\left(f\left(Y_t\right)I\left(t>\zeta_Y\right)\right)=\exp(\mathbb{G}_nt)f_n(x),
	\end{equation}
	where $f_n=(f(x_0),\ldots,f(x_n))^T$ is an $(n+1)$-dimensional column vector and $\exp(\mathbb{G}_nt)f_n(x)$ is the entry corresponding to $x$ in the vector $\exp(\mathbb{G}_nt)f_n$.
	
	One may also be interested in approximations of $P(t,x,l)$ and $p(t,x,y)$. To this end, denote the transition probability of the CTMC from $x$ to $y$ in time $t$ by $P_n(t,x,y)$, which is given by
	\begin{equation}\label{eq:TP-Kill-CTMC}
	P_n(t,x,y)=\exp(\mathbb{G}_nt)(x,y),
	\end{equation}
	which is the entry of the matrix $\exp(\mathbb{G}_nt)$ with its row corresponding to $x$ and its column to $y$.
	In particular, $P_n(t,x,l)$ approximates $P(t,x,l)$. For $y\in\mathbb{S}_n^\circ$, define
	\begin{equation}
	p_n(t,x,y):=P_n(t,x,y)/\delta y.
	\end{equation}
	Then, we can approximate $p(t,x,y)$ by $p_n(t,x,y)$.
	
	It is also of interests in practice to calculate the first passage probability $P(\tau^X_z>t|X_0=x)$ for $l\leq x<z$ where $\tau^X_z$ is the hitting time of $z$ of the sticky diffusion $X$. \cite{nie2019} obtain a semi-analytical formula for this probability under the sticky OU model. We use CTMC approximation to calculate the probability for general sticky diffusions. Let $Y$ be the CTMC that approximates $X$ (set $k\equiv 0$ in $\mathbb{G}_n$). Then, for $Y$, using \cite{mijatovic2013},
	\begin{equation}
	P(\tau^Y_z>t|Y_0=x)=\exp(\mathbb{H}_nt)\textbf{1}(x),
	\end{equation}
	where $\tau^Y_z=\inf\{t: Y_t>z\}$, $\mathbb{H}_n$ is the submatrix of $\mathbb{G}_n$ with entries corresponding to states smaller than $z$, and $\textbf{1}$ is a vector of ones with dimension compatible with $\mathbb{H}_n$. In particular, setting $x=l$, we obtain an approximation of the first passage probability when the diffusion starts from the sticky boundary.

	To implement the proposed method, we need to compute the matrix exponential for which there exist a slew of algorithms. In this paper we evaluate the performance of three approaches, as listed below.
	\begin{itemize}
		\item The scaling and squaring algorithm of \cite{higham2005} which is implemented in Matlab is used widely in the literature on Markov chain approximation. It is known that this algorithm can handle all kinds of matrices and it is numerically stable. This algorithm requires $O(n^3)$ operations to compute $\exp(\mathbb{G}_nt)f_n$.
		
		\item When $\mathbb{G}_n$ is a tridiagonal matrix, \cite{li2016} developed an algorithm based on the fast matrix eigendecomposition algorithm of \cite{dhillon1997}, known as the MRRR algorithm. It takes only $O(n^2)$ operations to compute $\exp(\mathbb{G}_nt)f_n$.
		
		\item The third approach is to numerically solve the ODE system the matrix exponential satisfies, which is
		\begin{equation}\label{eq:ME-ODE}
		d\mathbb{M}(t)=\mathbb{G}_n\mathbb{M}(t)dt,\ \mathbb{M}(0)=\mathbb{I}.
		\end{equation}
		We propose to apply the extrapolation approach of \cite{feng2008a}, which was not considered in the literature on CTMC approximation before.
	\end{itemize}
	Our experiment in Section \ref{sec:sticky_ou_process_applications} shows that the third approach performs the best and we recommend using it for computing $\exp(\mathbb{G}_nt)$ especially when $t$ is large.

	\subsection{Eigenfunction Expansions for the CTMC}
	\label{subsec:eigenfunction_expansion_ctmc}
	We provide an eigenfunction expansion representation for $p_n(t,x,y)$, which will be utilized later to analyze the convergence rate of CTMC approximation. Let
	\begin{align}
	M_n\left(x_0\right)&=M\left(x_0\right)\exp\left(\frac{\alpha}{\sigma^2\left(x_0\right)}\delta^+x_0\right), \label{eq:approximation_speed_measure_x0} \\
	\frac{1}{s_n\left(x_0\right)}&=\rho M_n(x_0)=\exp\left(\frac{\alpha}{\sigma^2\left(x_0\right)}\delta^+x_0\right), \label{eq:approximation_scale_function_x0}
	\end{align}
	where \begin{equation}
	\alpha=\begin{cases}
	\mu(x_0),&\quad\text{Scheme 1},\\
	\rho,&\quad\text{Scheme 2},
	\end{cases}
	\end{equation}
	and for $x=x_1,\ldots,x_n$,
	\begin{align} m_n\left(x\right)&=M_n\left(x_0\right)\beta\frac{2\rho}{-\mu\left(x_1\right)\delta^+x_1+\sigma^2\left(x_1\right)}\prod_{y=x_1}^{x^-}\frac{\mu\left(y\right)\delta^-y+\sigma^2\left(y\right)}{-\mu\left(y^+\right)\delta^+y^++\sigma^2\left(y^+\right)},
	\label{eq:approximation_speed_density_interior}\\
	\frac{1}{s_n\left(x\right)}&=\frac{1}{2}m_n(x)(\mu(x)\delta^-x+\sigma^2(x)).
	\end{align}
	where $\alpha=\mu(x_0)$ for Scheme 1, $\alpha=\rho$ for Scheme 2 and the product term equals $1$ if $x=x_1$. We also define $M_n(x)=m_n(x)\delta x$. As will be shown later, $m_n(x)\approx m(x)$ for $x\in\mathbb{S}_n^\circ$ and $M_n(x_0)\approx M(x_0)$.

	Recall the expression \eqref{eq:discretized_generator_interior} for $G_n$, the generator of the CTMC. For $x\in\mathbb{S}_n^\circ$, we can rewrite it in the following form
	\begin{equation}
	G_ng\left(x\right)=\frac{1}{m_n\left(x\right)}\frac{\delta^-x}{\delta x}\nabla^-\left(\frac{1}{s_n\left(x\right)}\nabla^+g\left(x\right)\right)-k\left(x\right)g\left(x\right),
	\end{equation}
	which is crucial for the convergence rate analysis.
	
	Consider the eigenvalue problem of $G_n$:
	\begin{align}
	&G_n\varphi(x)=-\lambda\varphi(x),\ x=x_0,\ldots,x_n,\\
	&\varphi(x_{n+1})=0.
	\end{align}
	where $\varphi(x)$ is defined on $\mathbb{S}_n$, i.e., it is a vector. Let $(\lambda^n_k,\varphi^n_k)$ be the $k$-th eigenpair. The eigenvalue problem of the operator $G_n$ is essentially the eigenvalue problem of the matrix $\mathbb{G}_n$, which is tridiagonal with negative diagonal elements and positive off-diagonal ones, and also diagonally dominant. Thus, it has $n+1$ simple and real eigenvalues  with $0\leq\lambda_1^n<\lambda_2^n<\dots<$ (\cite{li2016}, Proposition 3.6) and one expects $\lambda^n_k\approx\lambda_k$.
	For any two functions $g_1,g_2$ defined on $\mathbb{S}_n$, define their inner product $(g_1,g_2)_n:=\sum_{x\in\mathbb{S}_n^-}g_1(x)g_2(x)M_n(x)$. Since eigenfunctions are only unique up to a constant, we normalize them to satisfy $(\varphi_j^n,\varphi_k^n)_n=\delta_{j,k}$ so that $\varphi_k^n(x_i)\approx\varphi_k(x_i)$.
	
	For every $x\in\mathbb{S}_n^-$, we have the following bilinear eigenfunction expansion representations:
	\begin{align}
	P_n(t,x,y)&=M_n(y)\sum_{k=1}^n\exp(-\lambda_k^nt)\varphi_k^n(x)\varphi_k^n(y),\ y\in\mathbb{S}_n^-,\ y\in\mathbb{S}_n^-,\label{eq:P-CTMC-ee}\\
	p_n\left(t,x,y\right)&=m_n\left(y\right)\sum_{k=1}^n\exp\left(-\lambda_k^nt\right)\varphi_k^n\left(x\right)\varphi_k^n\left(y\right),\ y\in\mathbb{S}_n^\circ.\label{eq:p-CTMC-ee}
	\end{align}

	\section{Convergence Rate Analysis}
	\label{sec:convergence_analysis}
	In this section, we derive the convergence rates of two CTMC approximation schemes. The structure of our proof is similar to that in \cite{zhang2019}. However, many details differ due to the change in the boundary behavior of the diffusion. In fact, some arguments in \cite{zhang2019} do not hold for the sticky case.
	
	We will analyze a sequence of grids satisfying the following assumption.
	
	\begin{assumption}\label{assumption:bounded_grids}
		For a sequence of grids $\{\mathbb{S}_n\}$ with $h_n\to 0$, there exists a constant $C>0$ independent of $n$ such that for every grid $\mathbb{S}_n$ we have
		\begin{equation}
		\frac{\max_{x\in\mathbb{S}^-_n}\ \delta^+x}{\min_{x\in\mathbb{S}^-_n}\ \delta^+x}\leq C.
		\end{equation}
	\end{assumption}
	This assumption essentially says the maximum step size and the minimum step size should go to zero at comparable rates, which applies in all the types of grids used in practice.
	
	Our convergence rate analysis hinges on the eigenfunction expansion representations in \eqref{eq:P-ee}, \eqref{eq:p-ee}, \eqref{eq:P-CTMC-ee} and \eqref{eq:p-CTMC-ee}. We will first develop estimates for the errors of eigenvalues and eigenfunctions and then tap the representations to derive the convergence rates for $P_n(t,x,x_0)\to P(t,x,x_0)$ and $p_n(t,x,y)\to p(t,x,y)$. In the following, $\alpha=\mu(x_0)$ for Scheme 1 and $\alpha=\rho$ for Scheme 2.
	\begin{proposition}\label{prop:convergence_speed_measure_density}
		Under Assumption \ref{assumption:rb}, the approximation error for the speed measure and density satisfies
		\begin{align}
		M_n\left(x_0\right)-M\left(x_0\right)&=M\left(x_0\right)\frac{\alpha}{\sigma^2\left(x_0\right)}\delta^+x_0+O\left(h_n^2\right), \label{eq:convergence_speed_measure_x0} \\
		m_n\left(x\right)-m\left(x\right)&=m\left(x\right)\frac{\mu\left(x\right)}{\sigma^2\left(x\right)}\left(\delta^+x-\delta^-x\right)+O\left(h_n^2\right),\qquad\forall x\in\mathbb{S}_n^\circ.	\label{eq:convergence_speed_density_interior}
		\end{align}
		Furthermore, for $x\in\mathbb{S}_n^-$ we have
		\begin{equation}
		\frac{1}{s_n\left(x\right)}-\frac{1}{s\left(x\right)}=O\left(h_n\right).
		\label{eq:convergence_scale_function}
		\end{equation}
	\end{proposition}
	
	\begin{proposition}\label{prop:li_or_paper_prop2}
		Consider $h_n\in(0,\delta)$, where $\delta$ is sufficiently small.
		Then there exists a constant $C>0$ such that for any $k\leq h_n^{-1/4}$,
		\begin{equation}
		\left\vert\lambda_k^n-\lambda_k\right\vert\leq Ck^4h_n^\gamma,\label{eq:or_paper_2018_prop2}
		\end{equation}
		where $C$ is independent of $k$ and $n$, $\gamma=1$ for Scheme 1 and $\gamma=2$ for Scheme 2.
	\end{proposition}

	\begin{proposition}\label{proposition:li_2017_proposition3}
		Suppose Assumptions \ref{assumption:rb} and \ref{assumption:bounded_grids} hold. Consider $h_n\in(0,\delta)$ where $\delta$ is sufficiently small. Then there exists a constant $C>0$ independent of $k$ and $n$ such that for $1\leq k\leq h_n^{-1/5}$ the following holds
		\begin{equation}
		\left\Vert\varphi_k^n-\varphi_k\right\Vert_{n,\infty}\leq Ck^4h_n^\gamma.
		\end{equation}
		Here, $\gamma=1$ for Scheme 1 and $\gamma=2$ for Scheme 2.
	\end{proposition}
	
	\begin{theorem}\label{theorem:convergence_transition_probability}
		Suppose Assumptions \ref{assumption:rb} and \ref{assumption:bounded_grids} hold and consider $h_n\in(0,\delta)$ for $\delta$ sufficiently small.
		For any $t>0$, $x\in\mathbb{S}_n^-$ and $y\in\mathbb{S}_n^\circ$ there holds
		\begin{align*} p_n\left(t,x,y\right)-p\left(t,x,y\right)&=p\left(t,x,y\right)\left(\frac{\mu\left(y\right)}{\sigma^2\left(y\right)}\left(\delta^+y-\delta^-y\right)\right)+C_th_n^\gamma,\\ P_n\left(t,x,x_0\right)-P\left(t,x,x_0\right)&=P\left(t,x,x_0\right)\left(\frac{\alpha}{\sigma^2\left(x_0\right)}\delta^+x_0\right)+C_th_n^\gamma,
		\end{align*}
		where $C_t>0$ is independent of $n$, $x_0$, $x$, $y$ and only depends on $t$. Furthermore, $\gamma=1$ for Scheme 1 and $\gamma=2$ for Scheme 2. \end{theorem}
	
	The value function under the CTMC model is given by
	\begin{equation}\label{eq:ptf-p-CTMC}
	u_n\left(t,x\right)=\sum_{y\in\mathbb{S}_n^-}P_n\left(t,x,y\right)f\left(y\right)=P_n\left(t,x,x_0\right)f\left(x_0\right)+\sum_{y\in\mathbb{S}_n^\circ}p_n\left(t,x,y\right)f\left(y\right)\delta y.
	\end{equation}
	Using \eqref{eq:ptf-p} and \eqref{eq:ptf-p-CTMC} together with the estimates in Theorem \ref{theorem:convergence_transition_probability}, we can estimate the difference between $u_n(t,x)-u(t,x)$, which also depends on the smoothness of the payoff function $f$. In the following discussions, we simply assume that there exists a point $\xi\in (l,r)$ at which $f$ may not be smooth. Specifically, $f$ is $C^2$ on $(l,\xi)\cup (\xi,r)$ and if $\xi$ is a non-smooth point then either $f(\xi-)\neq f(\xi+)$ or $f(\xi-)=f(\xi+)$ but $f'(\xi-)\neq f'(\xi+)$. These types of payoffs are common in financial applications. The result can be easily extended to handle multiple non-smooth points.
	\begin{theorem} \label{theorem:discretization_error_convergence}
		Under Assumptions \ref{assumption:rb} and \ref{assumption:bounded_grids}, there exist constants $C_t,D_t>0$ independent of $n$, such that
		\begin{align*}
		\left\Vert u_n\left(t,\cdot\right)-u\left(t,\cdot\right)\right\Vert_{n,\infty}&\leq\underset{x\in[l,r)}{\sup}\ p\left(t,x,\xi\right)\left\vert f\left(\xi-\right)-f\left(\xi+\right)\right\vert\left\vert\frac{\xi^-+\xi^+}{2}-\xi\right\vert+C_th_n^\gamma, \\
		\left\vert u_n\left(t,x\right)-u\left(t,x\right)\right\vert&\geq p\left(t,x,\xi\right)\left\vert f\left(\xi-\right)-f\left(\xi+\right)\right\vert\left\vert\frac{\xi^-+\xi^+}{2}-\xi\right\vert-D_th_n^\gamma.
		\end{align*}
		Here, $D_t$ is also independent of $x$ and $\gamma=1$ for Scheme 1 and $\gamma=2$ for Scheme 2.
	\end{theorem}
	
	This theorem shows that for Scheme 1, convergence is only first order regardless of how non-smooth $f$ is at $\xi$. For Scheme 2, discontinuity in the first order derivative does not change the convergence order, however discontinuity in the payoff undermines the order. A simple grid design that would restore second order convergence for Scheme 2 is to place $\xi$ in the middle of two grid points. This midpoint rule was proposed and validated in \cite{zhang2019} for diffusions with two killing boundaries.
	
	The convergence rate for the first passage probability $P(\tau^X_z>t|X_0=x)$ can be analyzed in essentially the same way. One can treat $z$ as a killing boundary for the diffusion and assume that $z$ is on the CTMC grid. The estimates in Theorem \ref{theorem:discretization_error_convergence} apply with the payoff $f$ identically $1$, so convergence is first order for Scheme 1 and second order for Scheme 2. It is important to place the passage level $z$ on the grid as explained in \cite{zhang2019}, Section 4.5, otherwise convergence becomes first order even for Scheme 2.
	
	\section{Numerical Results for the Sticky OU Short Rate Model}
	\label{sec:sticky_ou_process_applications}
	
	The usefulness of the CTMC approximation to sticky diffusion processes is demonstrated by studying a sticky Ornstein-Uhlenbeck process,
	which is given by the following SDE
	\begin{equation}
	dX_t=\left(\kappa\left(\theta-X_t\right)I\left(X_t>0\right)+\rho I\left(X_t=0\right)\right)dt+\sigma I\left(X_t>0\right)dB_t,
	\end{equation}
	where $l=0$ is the left boundary. This process is used in \cite{nie2017} (also \cite{nie2019}) to model short rates, which are instantaneous interest rates. The advantages of using this model over standard short-rate models are explained in \cite{nie2017}. In particular, this model is able to produce various shapes of yield curves observed in the market, especially the S-shaped yield curve in low interest environment after the 2008 financial crisis. To illustrate this point, we consider three different sets of values for the parameters of a sticky OU process, which are listed in Table \ref{tab:model_parameter_settings}. Figure \ref{fig:zcp_yields_models} shows three shapes of yield curves produced by the three sets: upward-sloping (Model 1), inverted (Model 2) and S-shaped (Model 3). In general, the magnitude of the stickiness parameter plays an important role in controlling the shape of the curve. Calibration shows that the implied stickiness parameter from real market data is on similar scales with the values in the table (see \cite{nie2017}).
	\renewcommand{\arraystretch}{1.2}
	\begin{center}
		\label{tab:model_parameter_settings}
		\begin{tabular}{ cccccc } \toprule[1.4pt]
			Model & $\kappa$ & $\theta$ & $\sigma$ & $\rho$ & $x$ \\
			\hline
			1 & 0.4500 & 0.1000 & 0.0500 & $4.0\times 10^{-3}$ & 0.0100 \\
			\hline
			2 & 0.7500 & 0.0500 & 0.0150 & $1.0\times 10^{-6}$ & 0.0010 \\
			\hline
			3 & 0.2210 & 0.2000 & 0.0170 & $5.8\times 10^{-5}$ & 0.0000 \\
			\hline
		\end{tabular}
		\captionof{table}{Model parameters for the sticky Ornstein-Uhlenbeck process.}
	\end{center}
	\renewcommand{\arraystretch}{1}\vspace{0.3cm}
	
	\begin{figure}[htbp]
		\centering
		\includegraphics[scale=0.5]{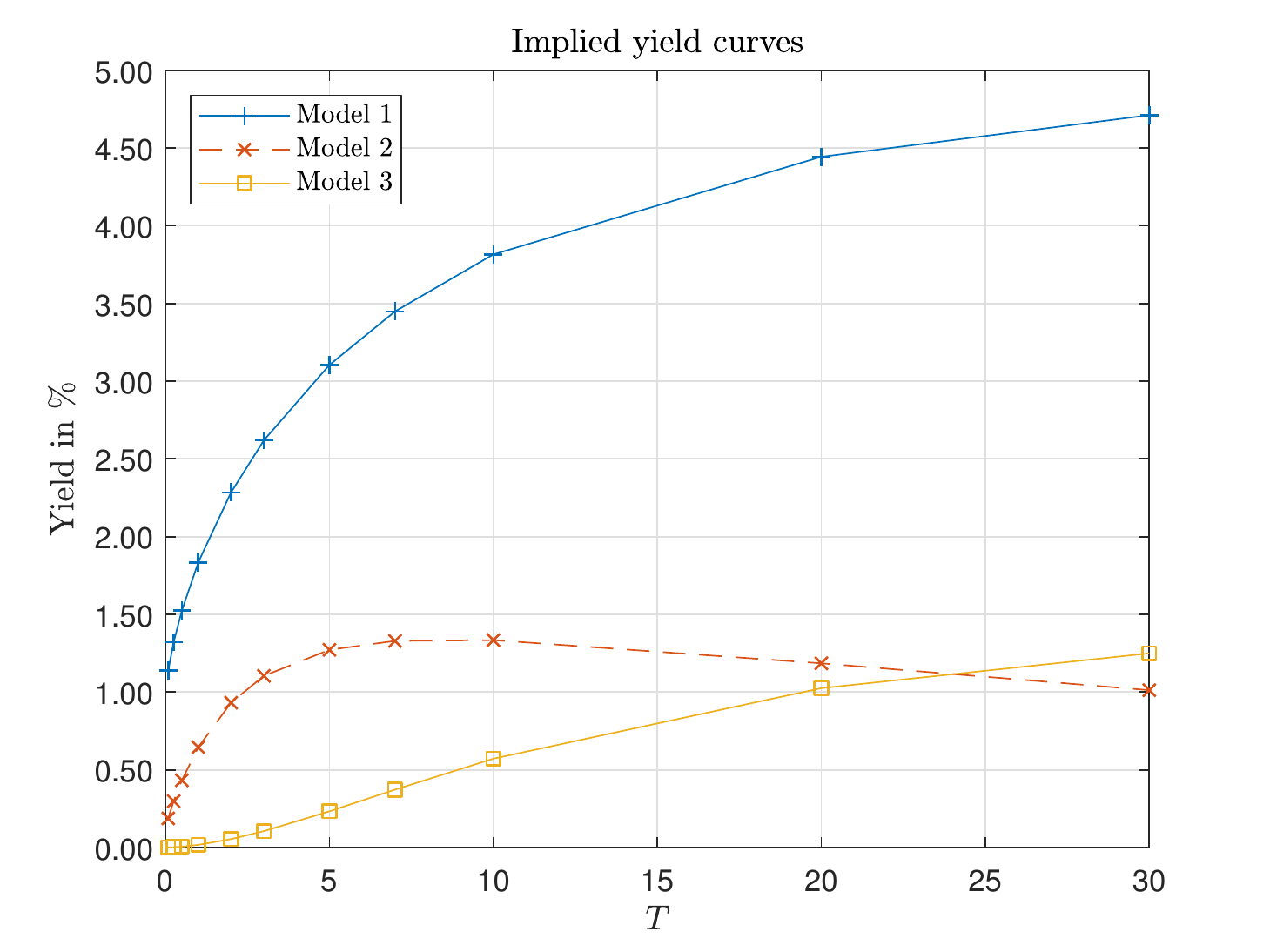}
		\caption{Yield curves for three models.}
		\label{fig:zcp_yields_models}
	\end{figure}
	
	Throughout this section, we consider pricing zero-coupon bonds with unit payoff (i.e.,
	\cite{nie2019} obtained an eigenfunction expansion formula for the bond price under the sticky OU model. We use their method to obtain benchmarks for our method. Table \ref{tab:comparison_ctmc_ee} shows the absolute difference between results of the CTMC method with 20,000 grid points and the eigenfunction expansion method using 50 to 100 terms in the expansion.
	\renewcommand{\arraystretch}{1.2}
	\begin{center}
		\label{tab:comparison_ctmc_ee}
		\begin{tabular}{c|l|l|l|l|l|l|l} \toprule[1.4pt]
			Maturity & \multicolumn{1}{c|}{6M} & \multicolumn{1}{c|}{1Y} & \multicolumn{1}{c|}{2Y} & \multicolumn{1}{c|}{3Y} & \multicolumn{1}{c|}{10Y} & \multicolumn{1}{c|}{20Y} & \multicolumn{1}{c}{30Y} \\ \hline
			Model 1  & 7.98E-10                & 2.90E-09                & 9.88E-09                & 9.24E-09                & 8.13E-08                 & 1.28E-07                 & 1.40E-07                \\ \hline
			Model 2  & 9.74E-08                & 1.48E-07                & 4.71E-07                & 8.71E-07                & 3.50E-06                 & 6.04E-06                 & 7.65E-06                \\ \hline
			Model 3  & 1.14E-05                & 9.86E-06                & 7.34E-06                & 5.43E-06                & 2.03E-06                 & 9.41E-06                 & 1.48E-05                \\ \hline
		\end{tabular}
		\captionof{table}{Differences in prices for the CTMC method and the eigenfunction expansion method of \cite{nie2019}.}
	\end{center}
	\renewcommand{\arraystretch}{1}\vspace{0.3cm}
	It can be seen that these methods yield consistent results for all maturities shown in the table. For small maturities (e.g., 1 or 3 months), the eigenfunction expansion method requires more terms in the expansion to obtain results with good accuracy and hence it can be slow. In contrast, the CTMC method works well for small maturities. In the following, we will use results from the CTMC method with a very fine grid as benchmarks.

	\subsection{Convergence Rates}
	
	We localize the state space $[0,\infty)$ to $[0,1]$. Simple calculations show that for the OU process with parameter values in Table \ref{tab:model_parameter_settings} the probability of breaching this upper boundary is extremely small for all maturities under consideration and hence the localization error is negligible. We then discretize $[0,1]$ with
	$n=100, 200, 400, 800, 1600$ using a uniform grid to detect the convergence rate numerically. Figure \ref{fig:convergence_model123} plots the results for three models with two maturities. It is clear that for Model 1 and 3, Scheme 2 converges at second order while Scheme 1 only attains first order, by checking the slope of the convergence line against that of the small triangle in the lower left corner of each plot. This validates our theoretical estimates in Theorem \ref{theorem:discretization_error_convergence}. For Model 2, the convergence orders of Scheme 1 and 2 are both around two. This may first seem unexpected, but the value of $\rho$ is very small in this model (much smaller than in the other two models), which makes $\mathbb{G}_{n,0,1}$ roughly zero under both schemes. So the two schemes perform similarly. Finally, for all models it is observed that the absolute error is greater when the maturity is longer.
	\begin{figure}[htbp]
		\centering
		\begin{subfigure}{.496\textwidth}
			\centering
			\includegraphics[width=\linewidth]{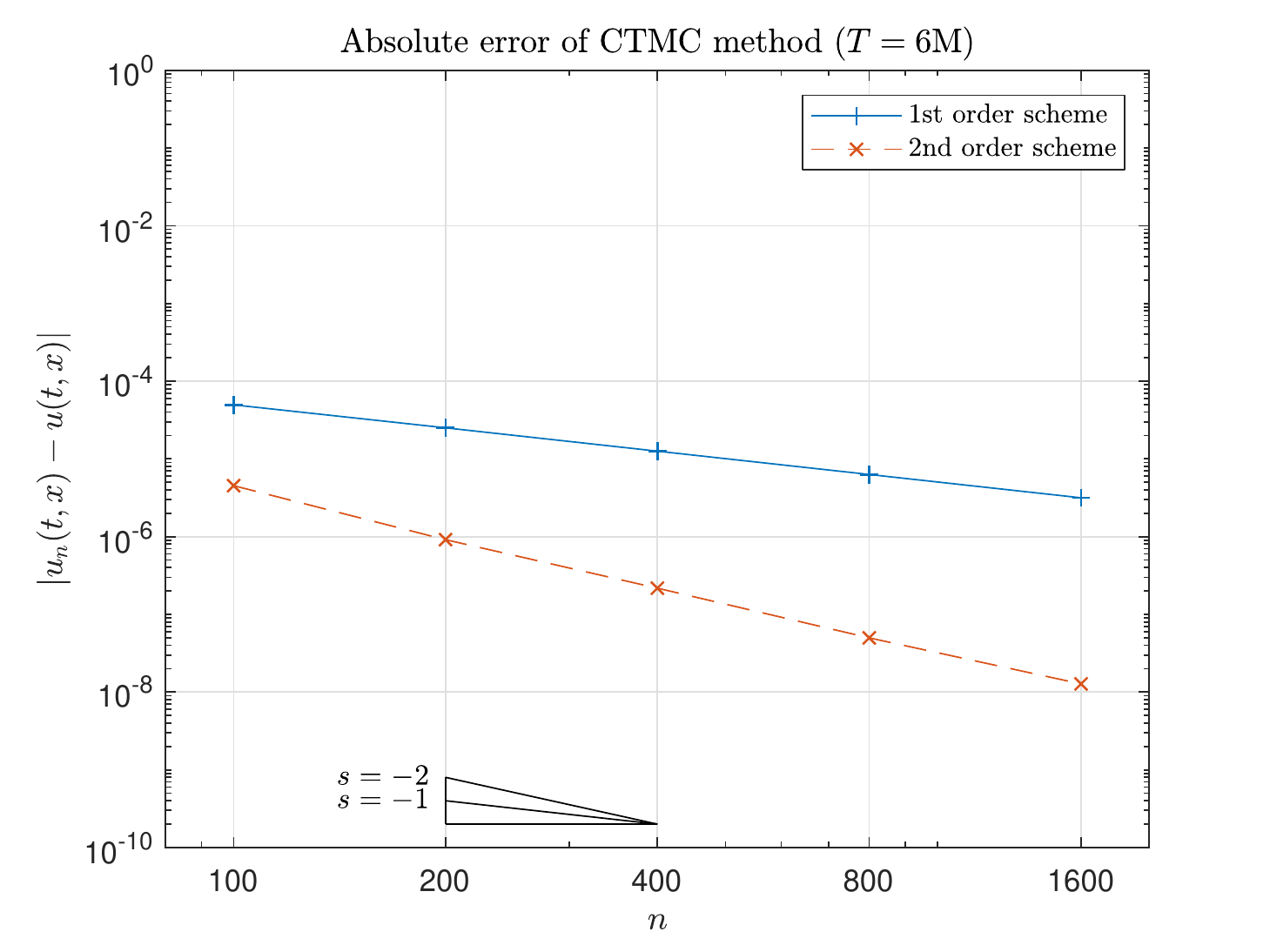}
		\end{subfigure}
		\begin{subfigure}{.496\textwidth}
			\centering
			\includegraphics[width=\linewidth]{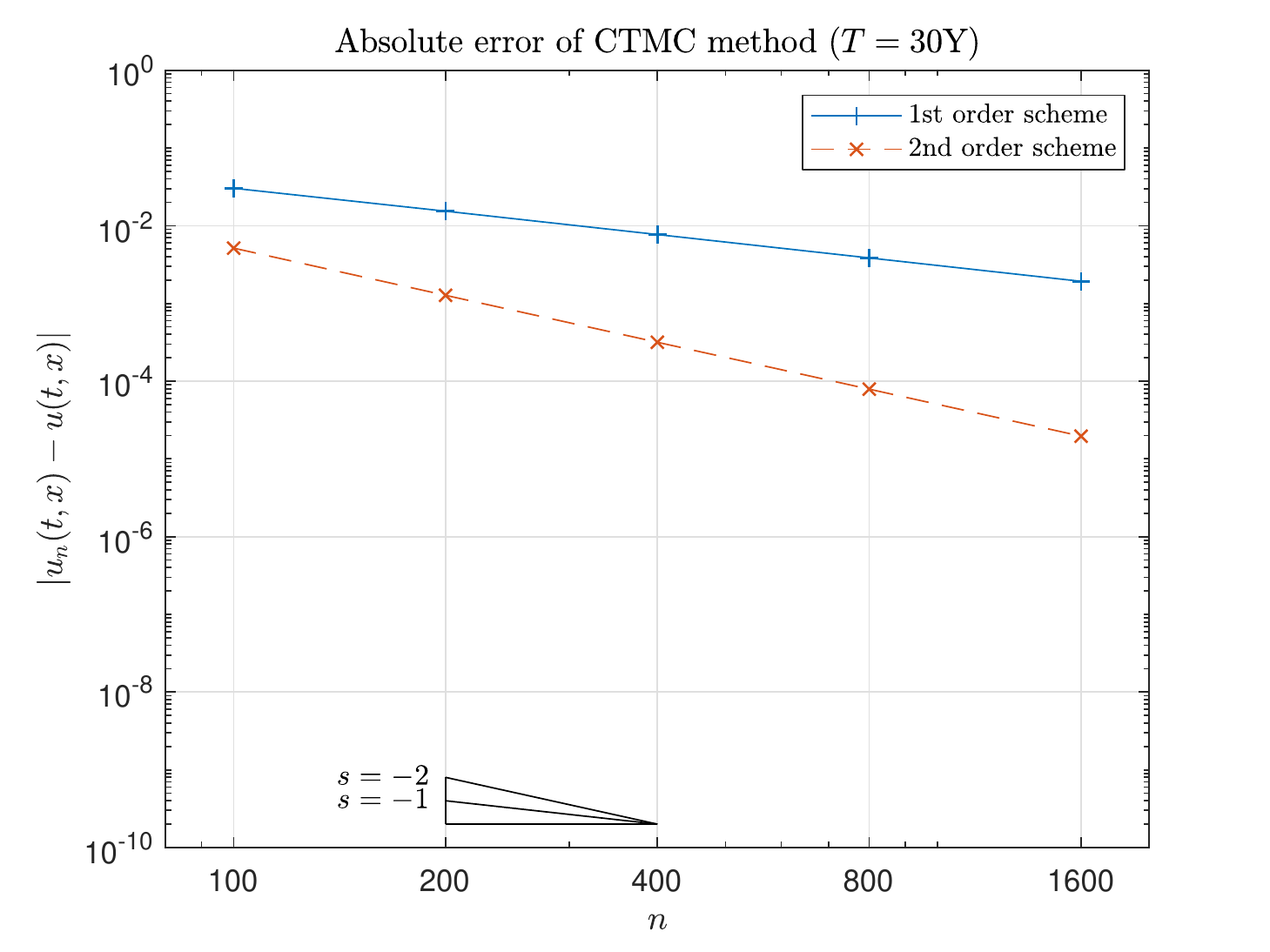}
		\end{subfigure}
		\par\bigskip
		\begin{subfigure}{.496\textwidth}
			\centering
			\includegraphics[width=\linewidth]{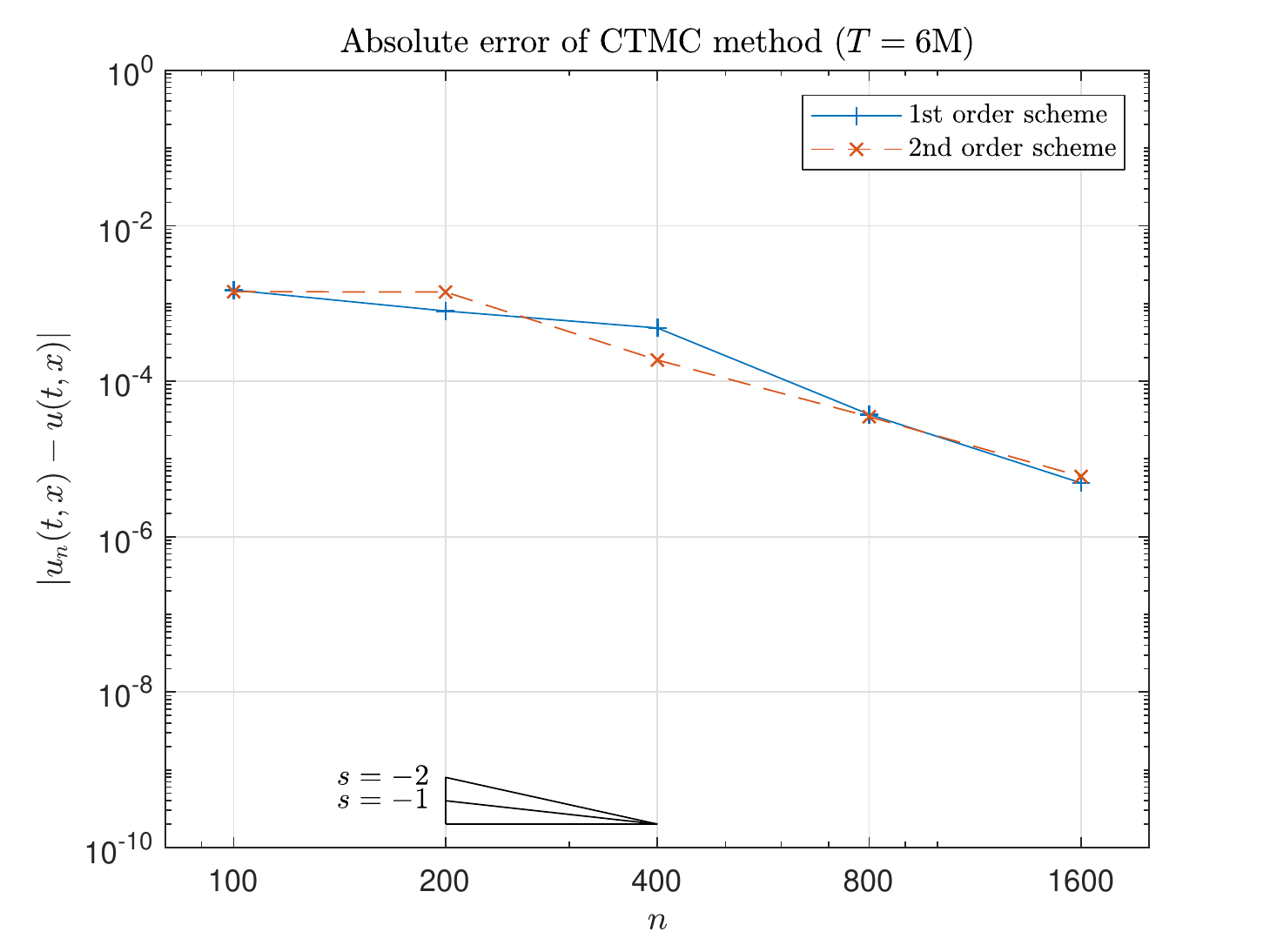}
		\end{subfigure}
		\begin{subfigure}{.496\textwidth}
			\centering
			\includegraphics[width=\linewidth]{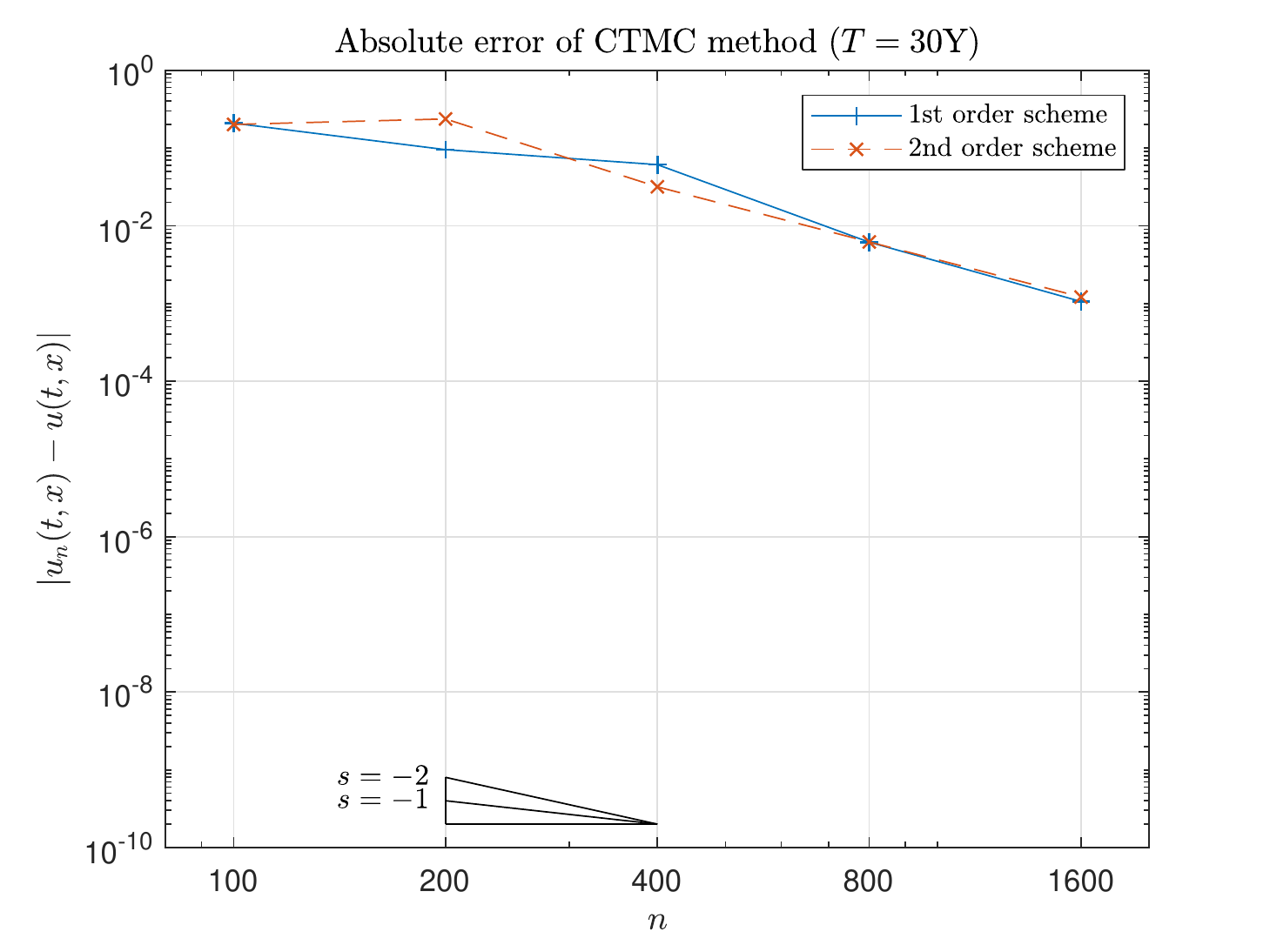}
		\end{subfigure}
		\par\bigskip
		\begin{subfigure}{.496\textwidth}
			\centering
			\includegraphics[width=\linewidth]{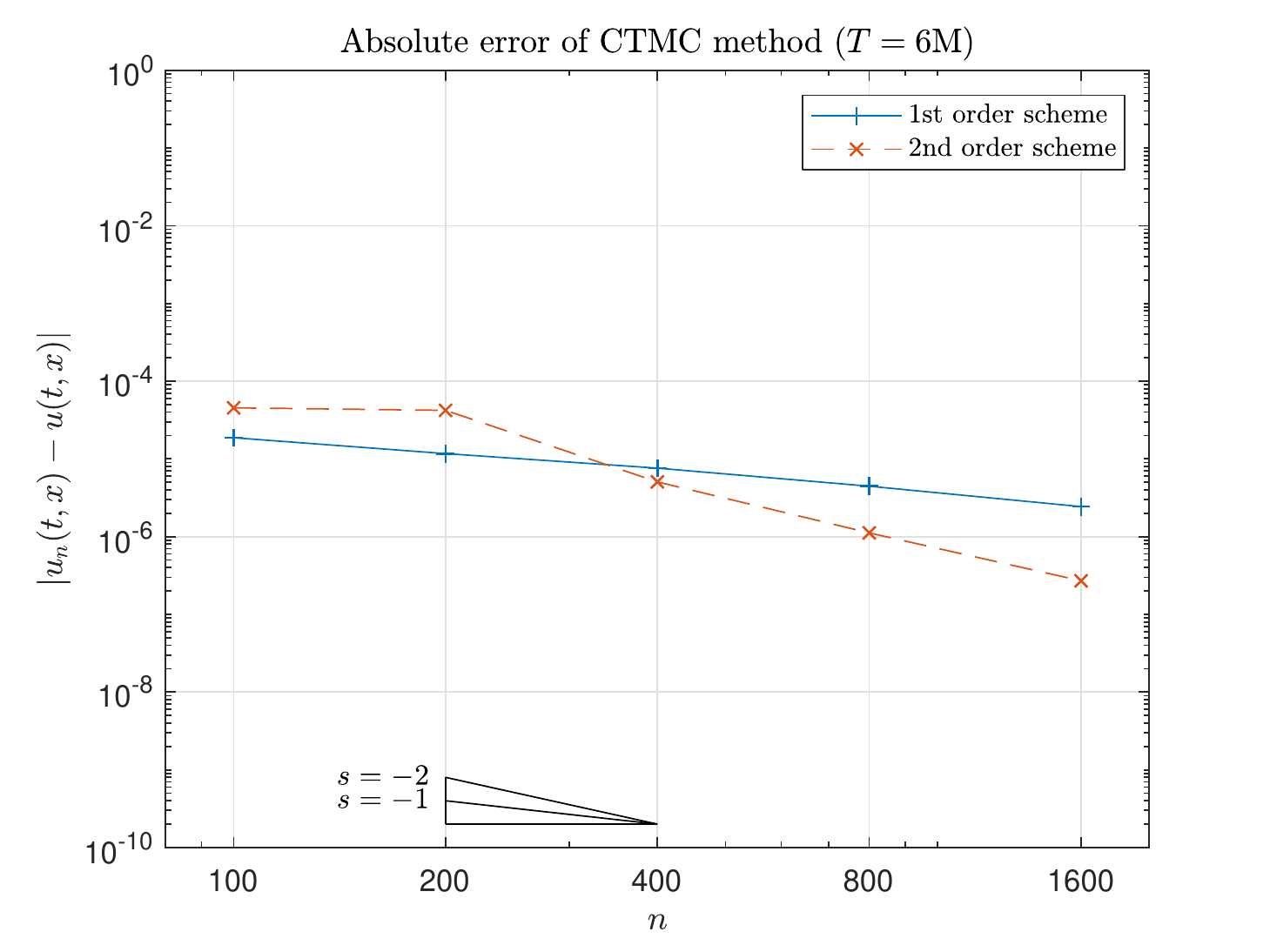}
		\end{subfigure}
		\begin{subfigure}{.496\textwidth}
			\centering
			\includegraphics[width=\linewidth]{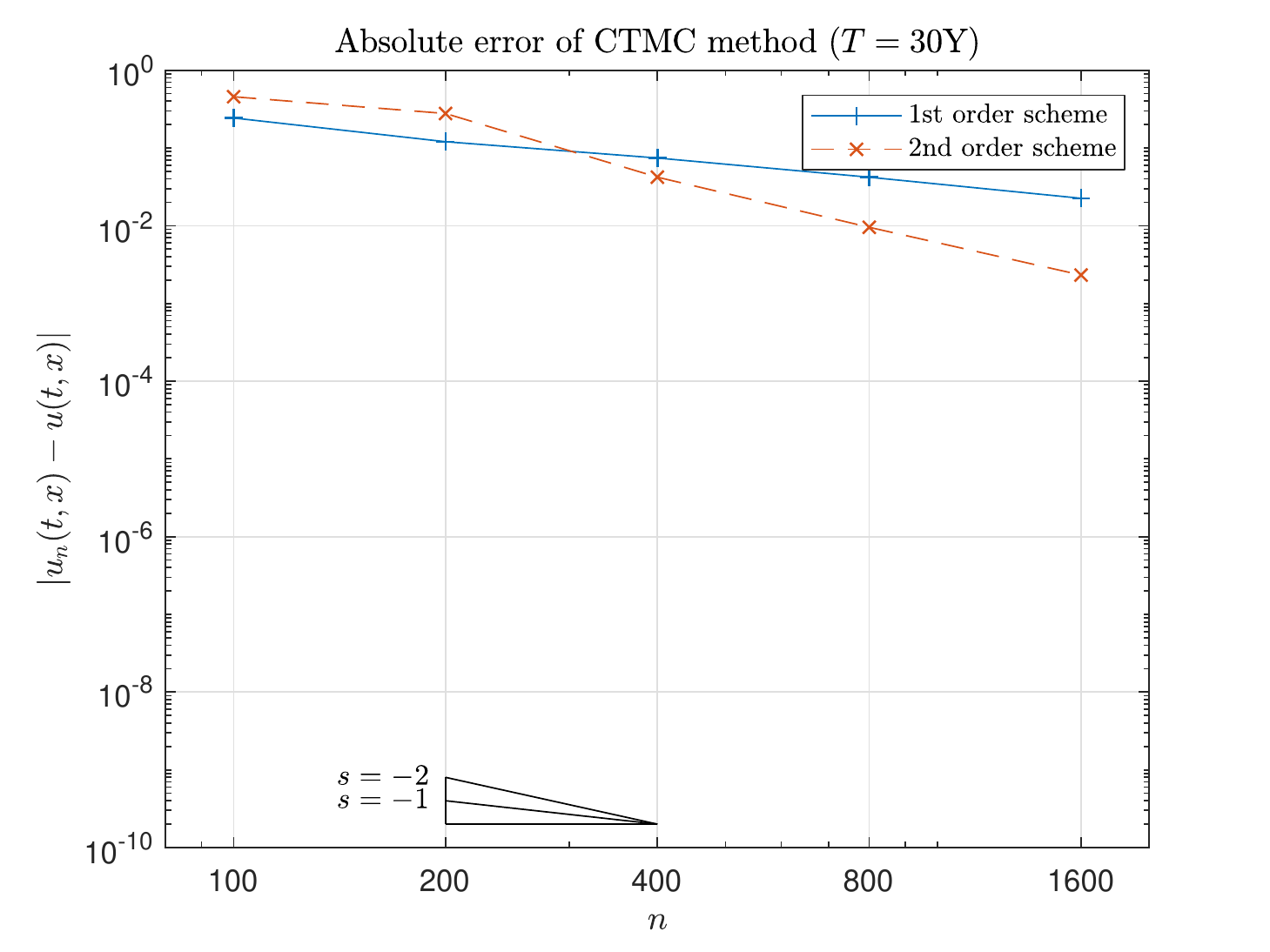}
		\end{subfigure}
		\caption{Absolute error vs. $n$ on log-log scale for Model 1 to 3 (from top to bottom) with 6M and 30Y maturities. }
		\label{fig:convergence_model123}
	\end{figure}

	\subsection{Some Comparisons}
	
	The calculation of the value function under the CTMC model is based on the calculation of the matrix exponential. In this section, we compare three ways of computing the matrix exponential in our problem as listed in Section \ref{sec:CTMC-FK} and we also compare the CTMC approximation algorithm with a standard finite difference scheme.
	
	We now briefly describe the extrapolation approach of \cite{feng2008a} for numerically solving the ODE system satisfied by the matrix exponential.
	This approach uses a basic step size $H$, where $H=0.5$ if maturity $T>0.5$ and $H=T$ otherwise, to divide the interval $[0,T]$ into smaller time periods.
	For each basic interval, $M_i=1,\ldots,s$ time steps are used to evolve the differential equation according to the implicit scheme to calculate the matrix exponential, where $s\geq 1$ denotes the extrapolation stage number.
	Let the approximation of the solution after one basic step be denoted by $A_{i,1}=u_n(H,x; M_i)$ where $M_i$ time points are used for the interval $[0,H]$.
	An extrapolation tableau is constructed by the following equation
	\begin{equation*}
	A_{i,j}=A_{i-1,j-1}+\frac{A_{i,j-1}-A_{i-1,j-1}}{\frac{M_i}{M_{i-j+1}}-1},
	\end{equation*}
	for $i=2,\ldots,s$ and $j=2,\ldots,i$.
	We then use the value $A_{s,s}$ after $s$ extrapolation steps as the starting point of the approximation over the time interval $[H,2H]$.
	This calculation of approximations after basic steps is repeated until one obtains an approximation of $u_n(T,x)$ after $M$ basic step size such that $T=MH$.
	Finally, $A_{s,s}$ after $M$ steps is the approximation of the zero-coupon bond price with maturity $T$ at time 0.
	
	For the finite difference scheme which discretizes both time and space in the PDE satisfied by the value function of the sticky diffusion, we use Crank-Nicolson time stepping, which is a standard choice in the literature for numerical solutions of PDEs. It is also considered in \cite{nie2017}. We use equal time steps with 5 steps in a month. This allows for an adequate number of time steps even for longer maturities.

	The MRRR algorithm, extrapolation approach and the finite difference method are implemented in C++ whereas the scaling and squaring method is already implemented in Matlab so we directly call it in Matlab (we have not found C++ code for this algorithm). Although a Matlab implementation may generally be less efficient than an equivalent C++ implementation, it does not affect the conclusion that the scaling and squaring method is typically the slowest for obtaining similar levels of accuracy, as it requires the highest amount of computational complexity.
	
	The results are obtained using a workstation running CentOS 7 with 3.2 GHz CPUs and memory of 256 GB. All numerical calculations are run 10 times to obtain the average running time.
	Figure \ref{fig:accuracy_computational_time} displays the comparison for all three models with 1Y and 10Y maturities.
	\begin{figure}[htbp]
		\centering
		\begin{subfigure}{.496\textwidth}
			\centering
			\includegraphics[width=\linewidth]{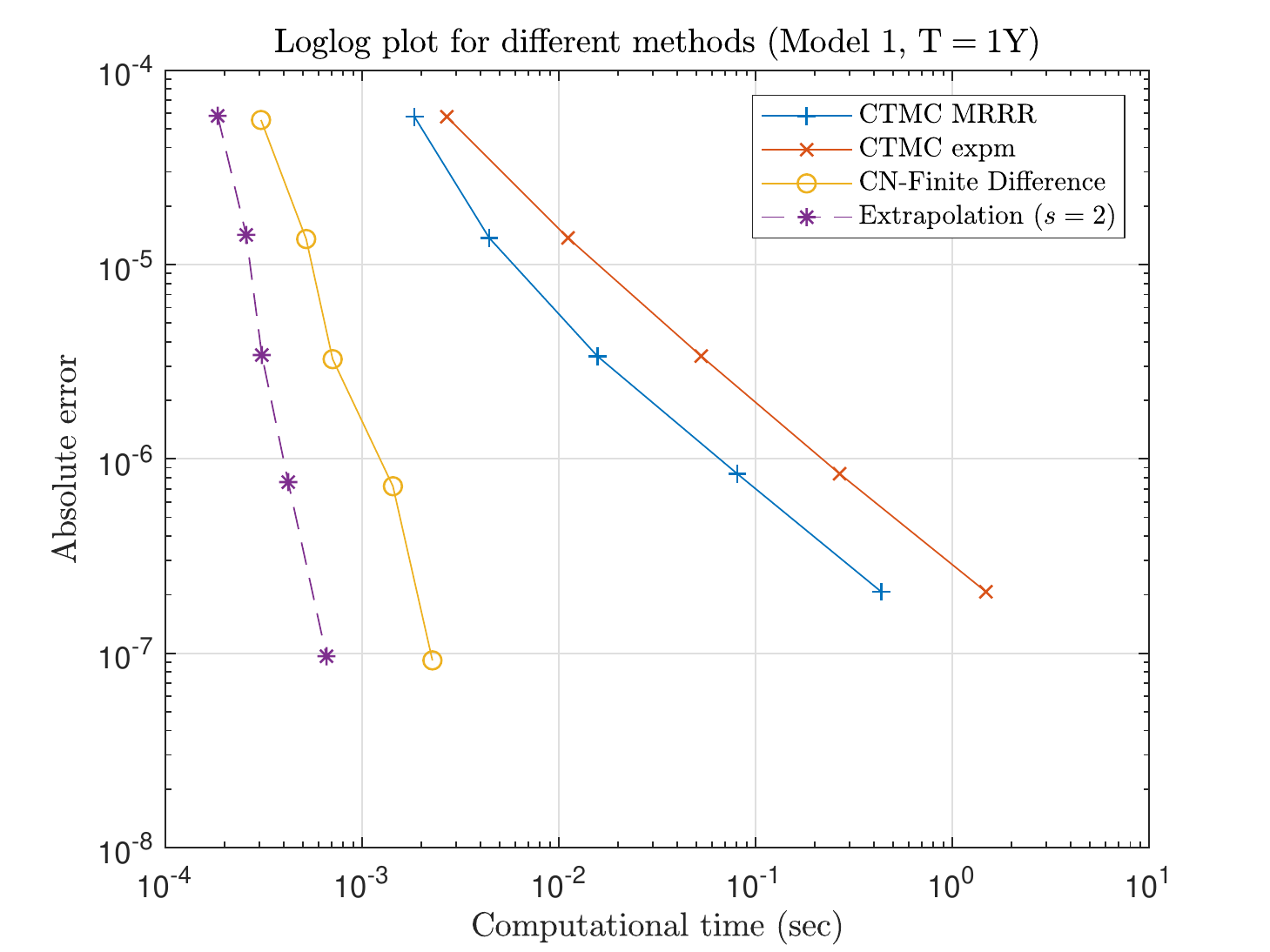}
		\end{subfigure}
		\begin{subfigure}{.496\textwidth}
			\centering
			\includegraphics[width=\linewidth]{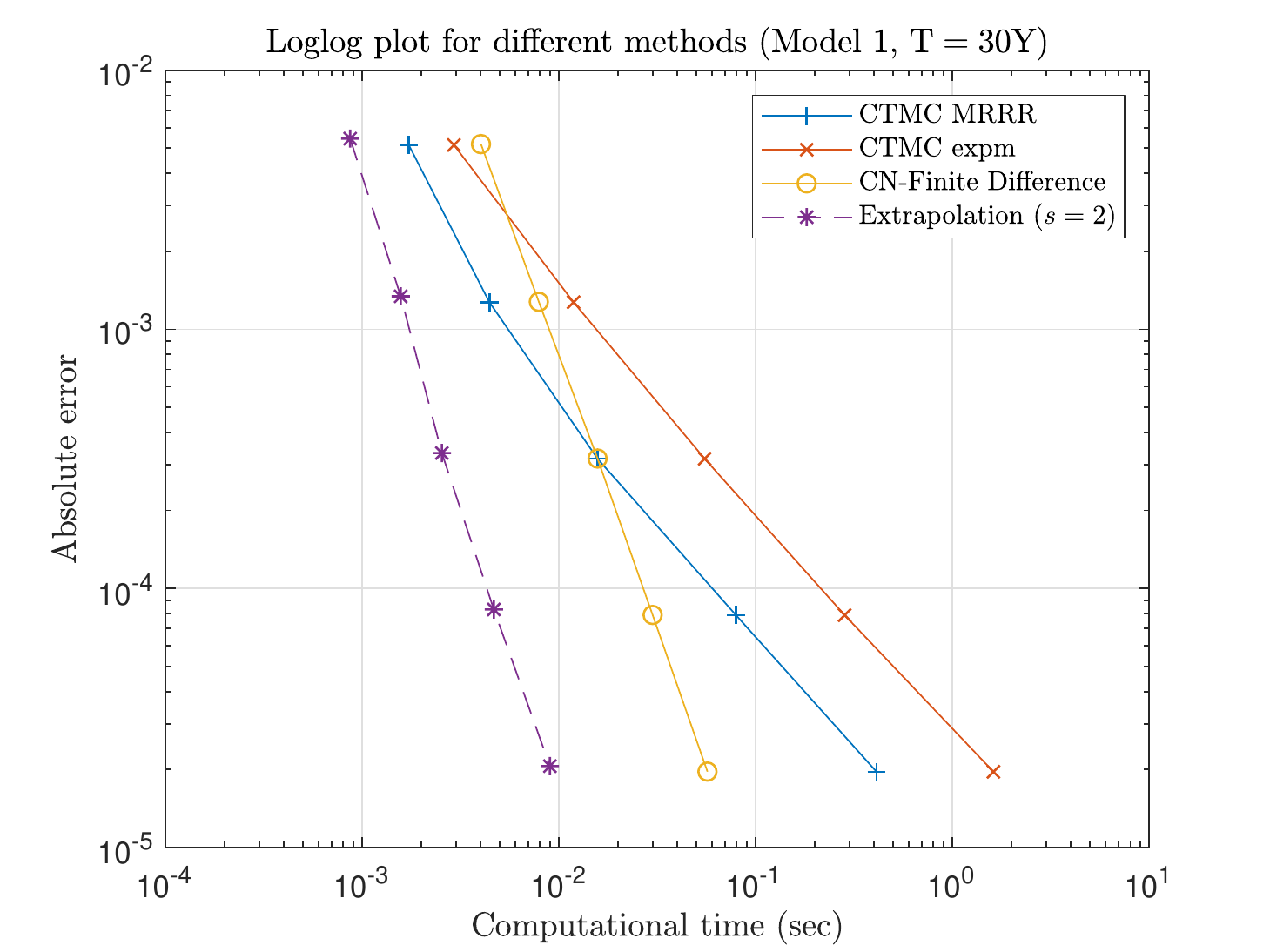}
		\end{subfigure}
		\par\bigskip
		\begin{subfigure}{.496\textwidth}
			\centering
			\includegraphics[width=\linewidth]{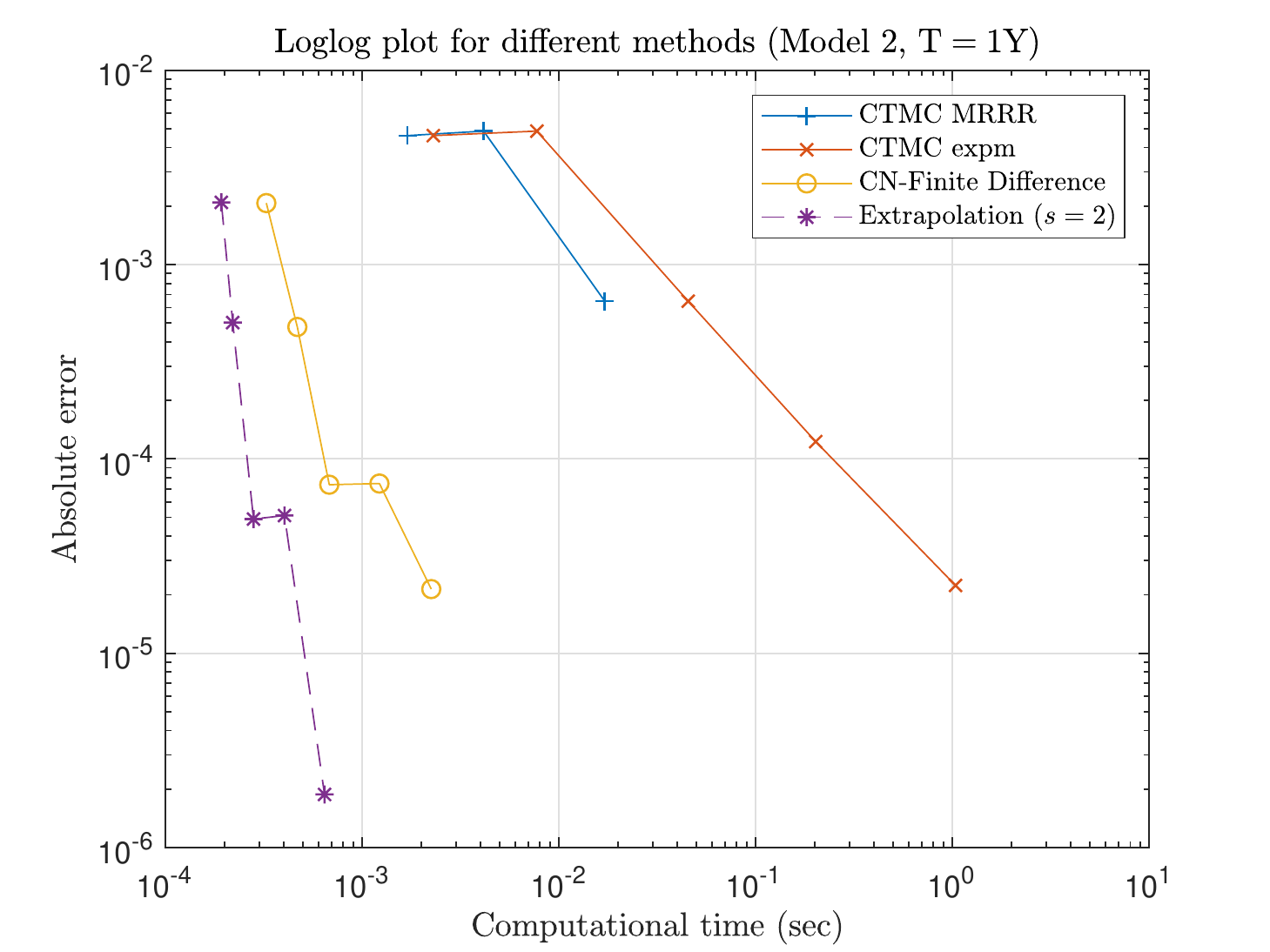}
		\end{subfigure}
		\begin{subfigure}{.496\textwidth}
			\centering
			\includegraphics[width=\linewidth]{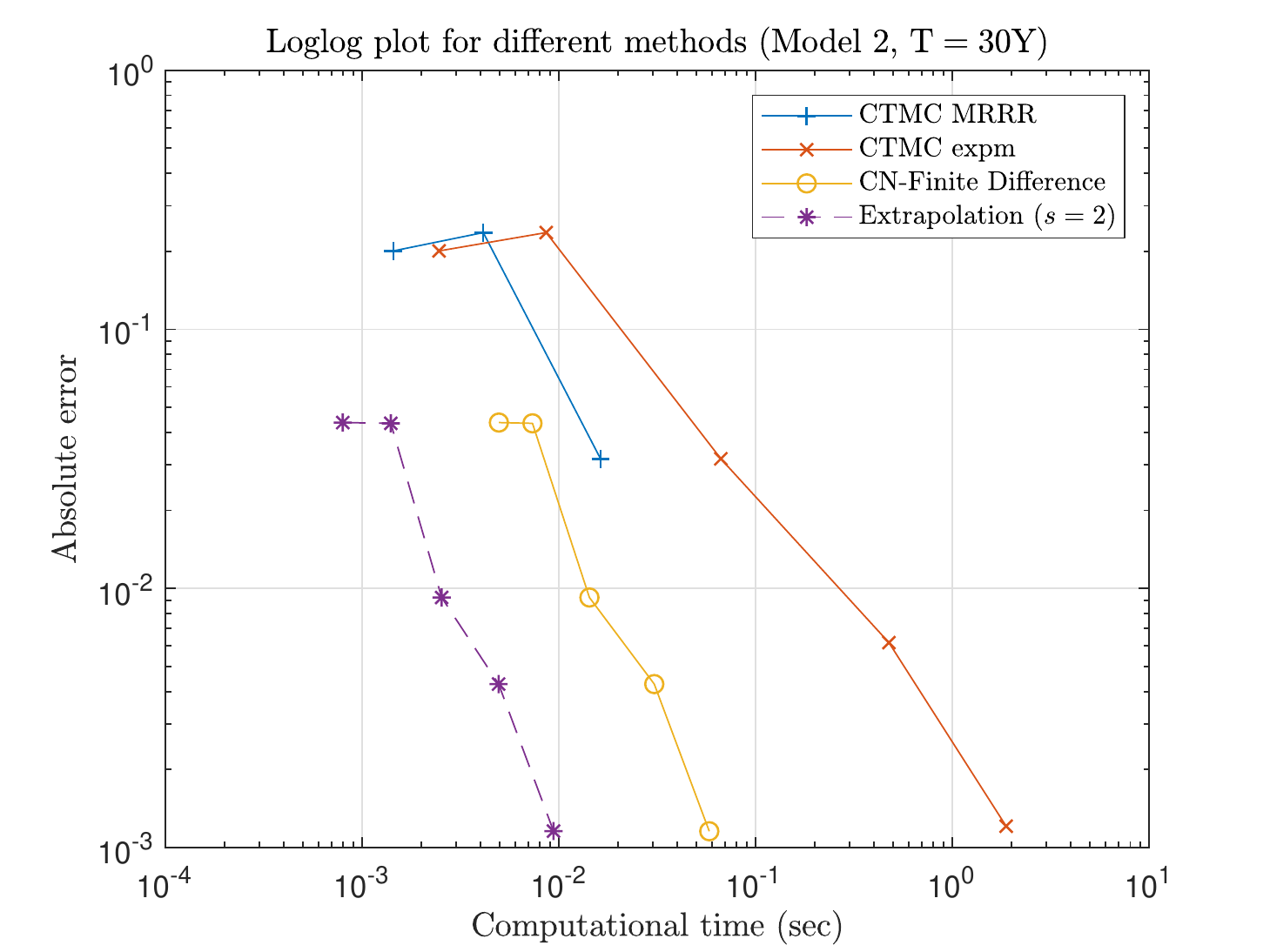}
		\end{subfigure}
		\par\bigskip
		\begin{subfigure}{.496\textwidth}
			\centering
			\includegraphics[width=\linewidth]{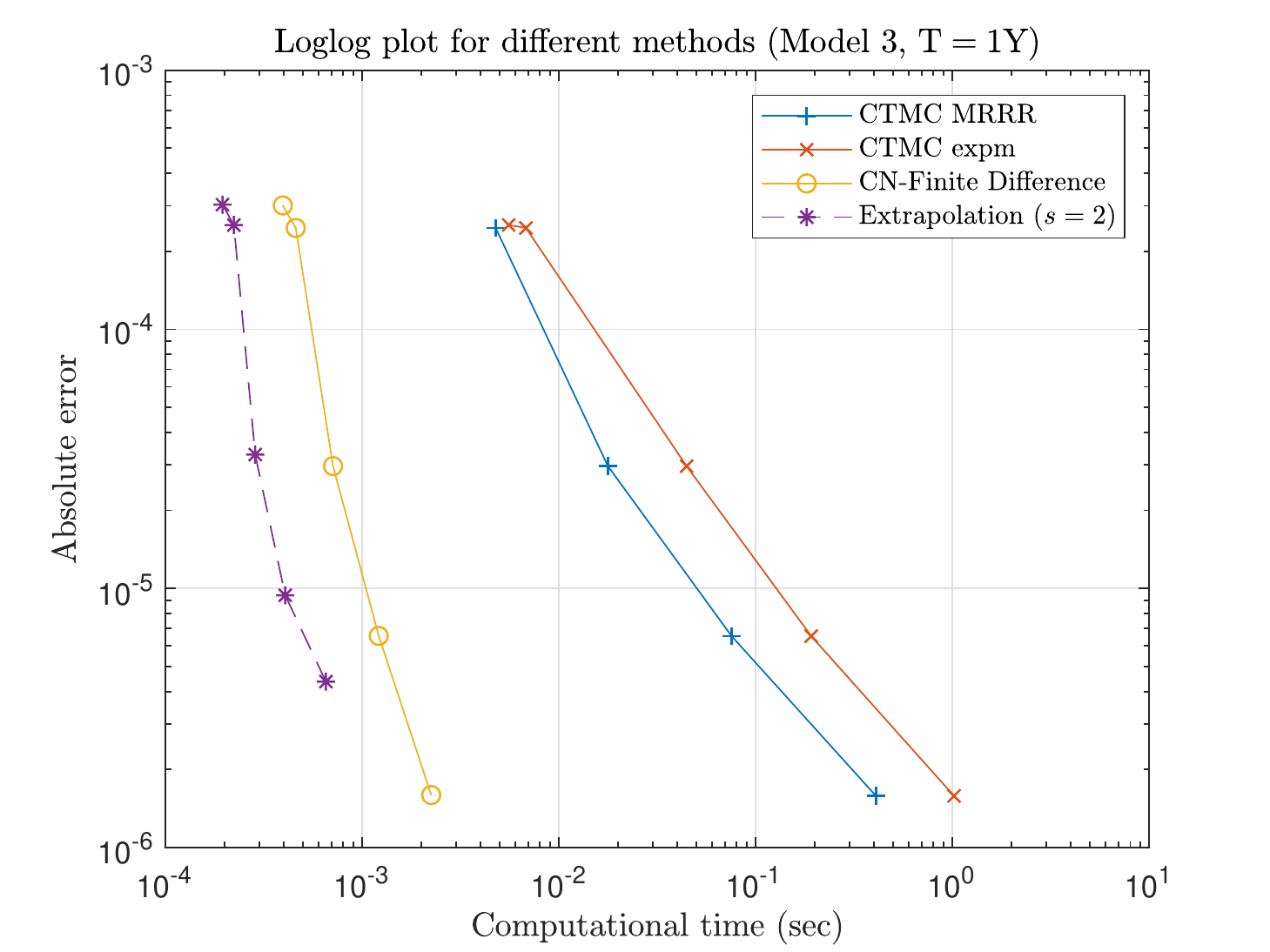}
		\end{subfigure}
		\begin{subfigure}{.496\textwidth}
			\centering
			\includegraphics[width=\linewidth]{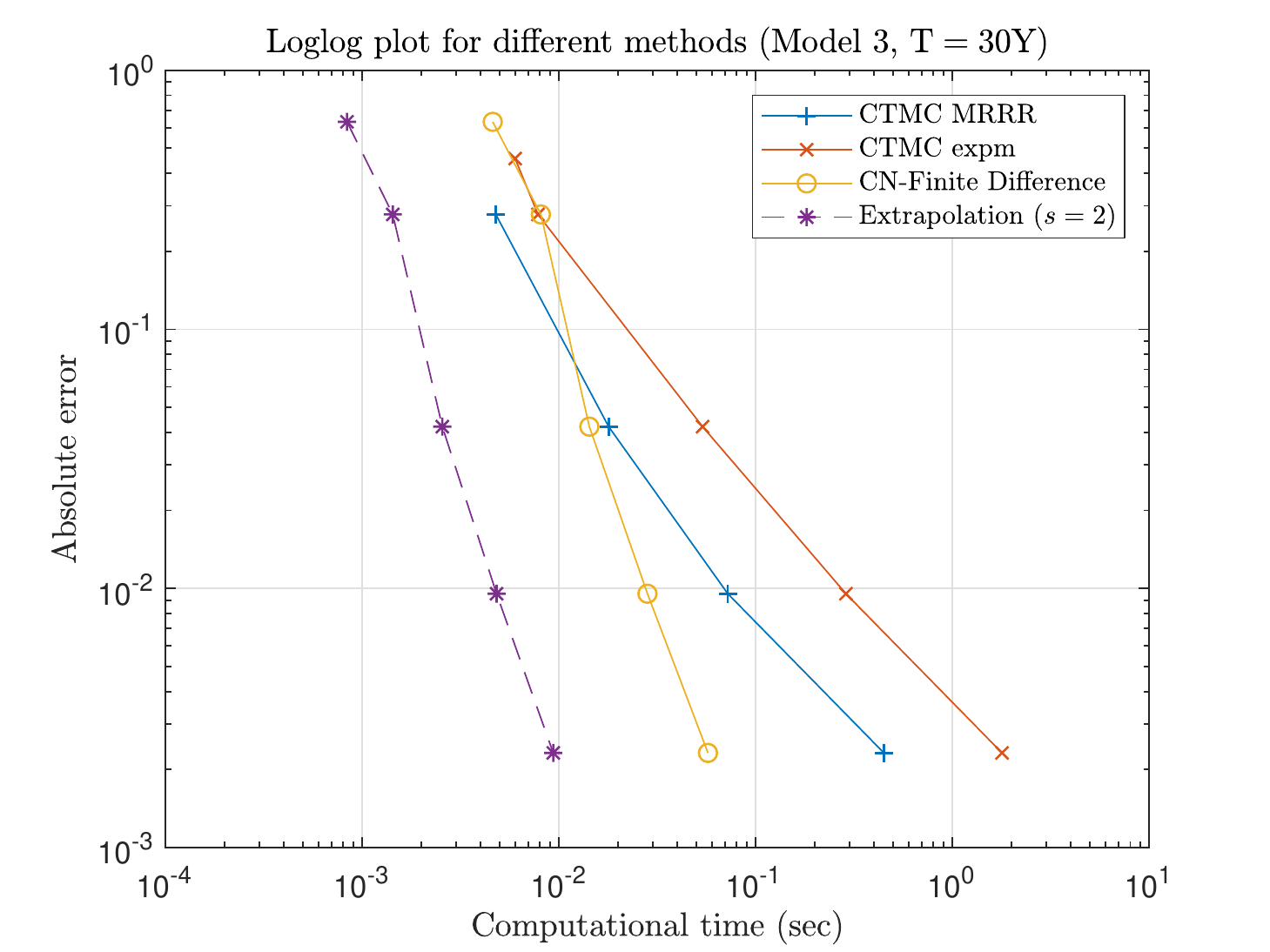}
		\end{subfigure}
		\caption{Comparison of four methods. ``CTMC expm'' stands for using the scaling and squaring algorithm for computing the matrix exponential in the CTMC method. In the extrapolation approach, extrapolation is only applied once (i.e., $s=2$).}
		\label{fig:accuracy_computational_time}
	\end{figure}
	
	The comparison clearly favors the extrapolation method which defeats the other three methods in all cases, and its leading edge becomes greater as the bond's maturity increases. It should be noted that the eigendecomposition method based on the MRRR algorithm does not always work. For Model 2, it cannot be applied when $n$ is too large and for Model 3, it fails for $n=100$ due to overflow/underflow errors in calculating a similar symmetric tridiagonal matrix required by the algorithm.

	\subsection{Simulation}
	We also consider how to do simulation for sticky diffusions. Our CTMC approximation method offers a natural alternative to the standard Euler scheme.
	The simulation of sample paths from a CTMC is straightforward and unlike the Euler scheme, no time discretization is needed.
	We use 500 grid points for the CTMC $Y$ constructed by Scheme 2, start with $Y_0=x$ and then draw an exponentially distributed random variable with intensity $\vert\mathbb{G}_{n,x,x}\vert$ to determine the amount of time spent in the initial state.
	After determining the time point when the Markov chain is transitioning, the transition rates $\mathbb{G}_{n,x,x^-}$, $\mathbb{G}_{n,x,x^+}$ and $k(x)$ are used to sample the new state.
	These steps are repeated until the maturity is reached. Using the definition of the order of weak convergence given in \cite{kloeden1999} and \cite{glasserman2004} and Theorem \ref{theorem:discretization_error_convergence}, the weak convergence order of the CTMC simulation scheme using the transition rate matrix under Scheme 2 is two (it is one if Scheme 1 is used).
	
	The Euler scheme is implemented in the following way.
	Time is discretized using 50 time steps per month. The process starts with $X_0=x$ and subsequent values of the process are computed using the discretization of the SDE \eqref{eq:general_diffusion_one_equation}.
	In particular, for $t=0,\Delta t,\ldots,T-\Delta t$,
	\begin{align*}
	Z_{t+\Delta t}&=\begin{cases} X_t+\mu\left(X_t\right)\Delta t+\sigma\left(X_t\right)\sqrt{\Delta t}\xi_{t+\Delta t}, & \qquad\textrm{if}\ X_t>0, \\ \rho\Delta t, & \qquad\textrm{if}\ X_t=0. \end{cases} \\
	e_{t+\Delta t}&\sim\textrm{Exp}\left(1\right),
	\end{align*}
	where $\xi_{t+\Delta t}$ is standard normally distributed and $e_{t+\Delta t}$ is exponentially distributed with mean 1.
	The new value $Z_{t+\Delta t}$ is accepted as given by
	\begin{equation*}
	X_{t+\Delta t}=\begin{cases} Z_{t+\Delta t} &\qquad\textrm{if}\ Z_{t+\Delta t}>0,\ e_{t+\Delta t}>k\left(X_t\right)\Delta t, \\ 0 &\qquad\textrm{if}\ Z_{t+\Delta t}\leq 0,\ e_{t+\Delta t}>k\left(X_t\right)\Delta t, \\ \partial &\qquad\textrm{if}\ e_{t+\Delta t}\leq k\left(X_t\right)\Delta t, \end{cases}
	\end{equation*}
	where $\partial$ is the cemetery state and $f(\partial)=0$.

	Figure \ref{fig:results_mc_model1} shows the Monte Carlo simulation results when sample paths are simulated by a CTMC and the Euler scheme.
	In both cases, 1000 samples paths are simulated and the price estimator together with the 99\% confidence intervals are displayed.
	\begin{figure}[htbp]
		\centering
		\begin{subfigure}{.496\textwidth}
			\centering
			\includegraphics[width=\linewidth]{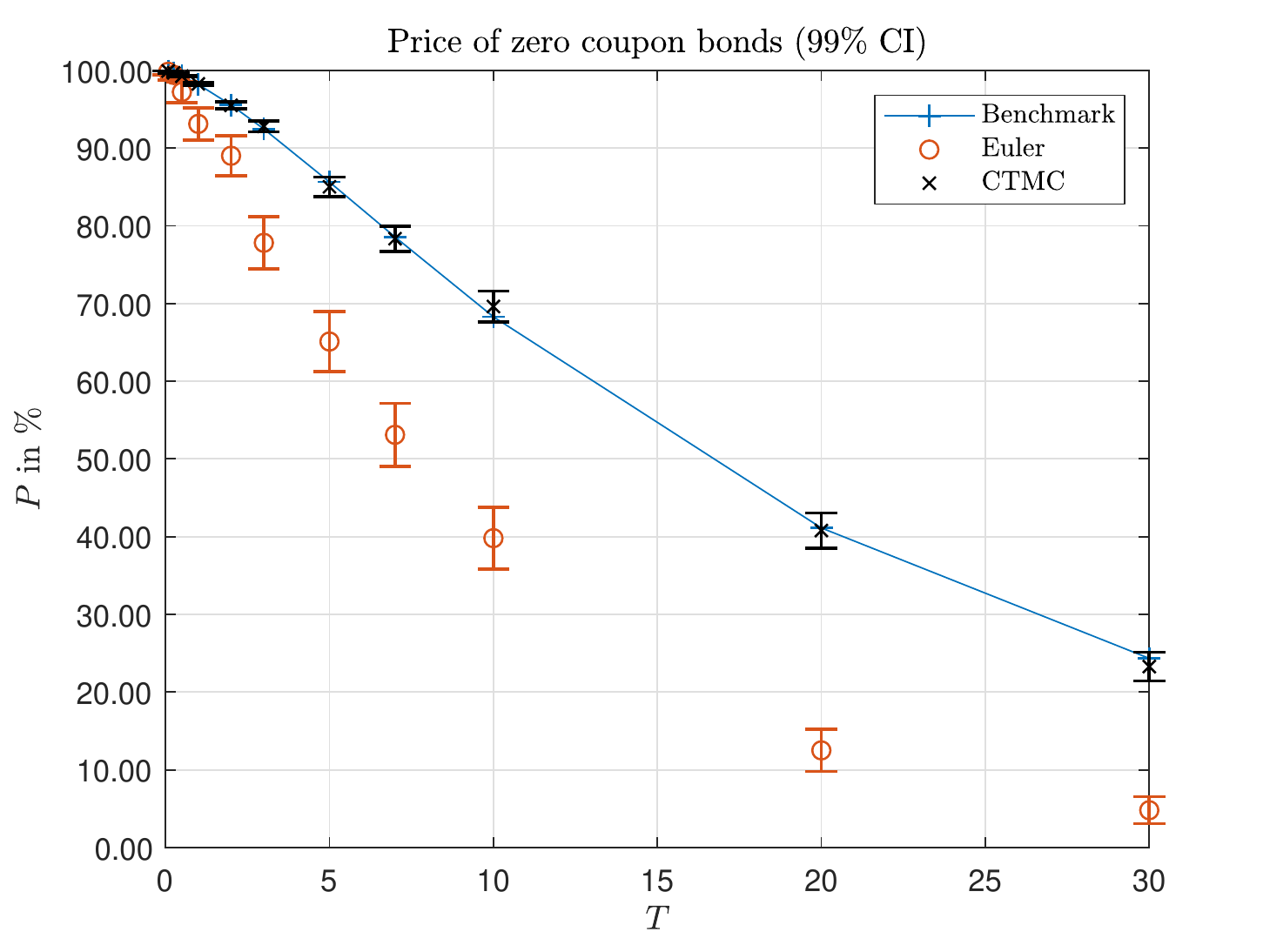}
		\end{subfigure}
		\begin{subfigure}{.496\textwidth}
			\centering
			\includegraphics[width=\linewidth]{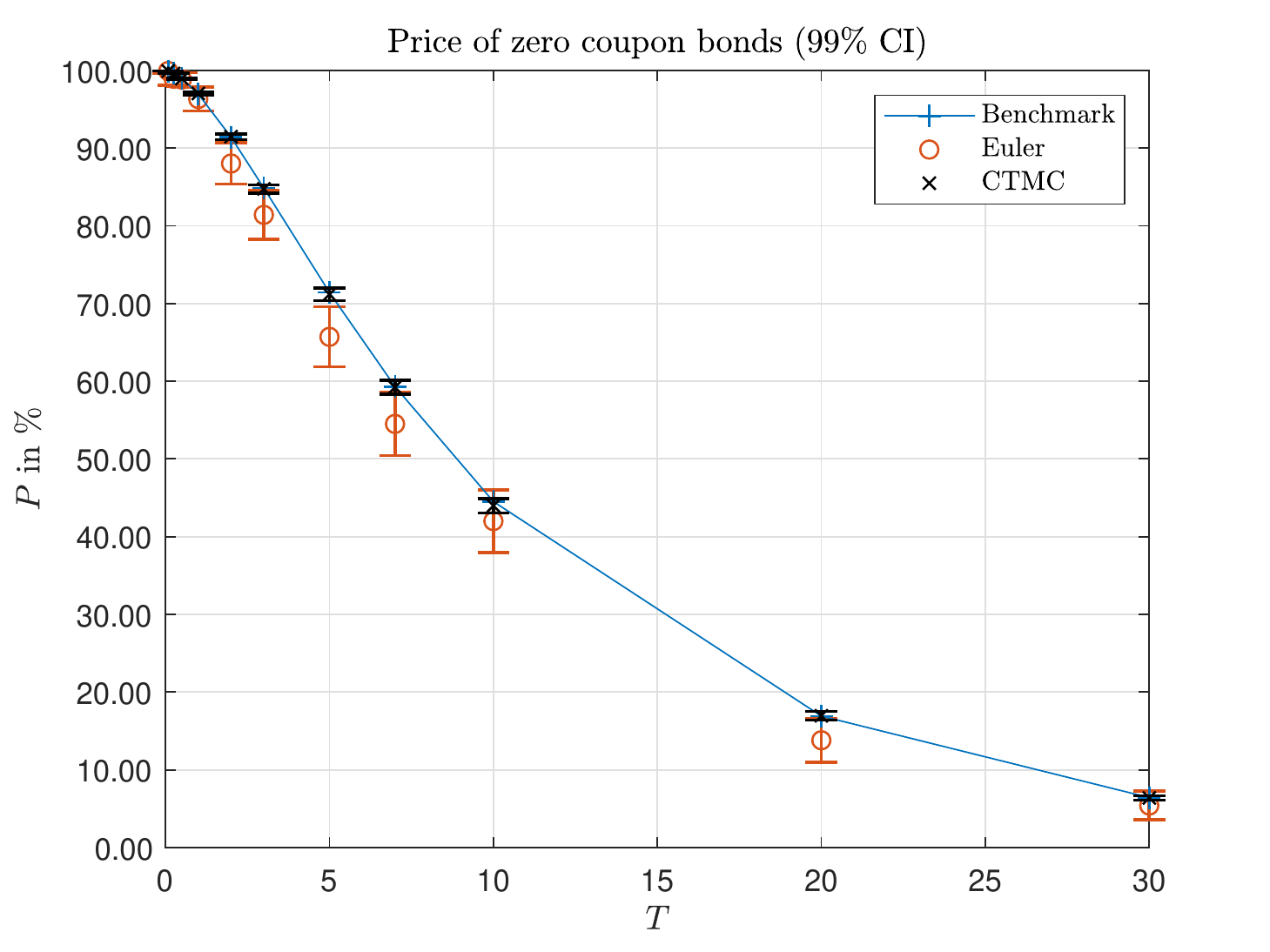}
		\end{subfigure}
		\caption{Monte Carlo simulation results for zero coupon bond pricing. The blue line indicates the benchmark prices. The left picture shows the result for Model 1 with $\rho=4.0\times 10^{-3}$ and in the right picture the results are obtained by setting $\rho=0.1$ with all other parameters remaining fixed.}
		\label{fig:results_mc_model1}
	\end{figure}
	The Euler scheme clearly fails when $\rho$ is small, i.e., the process is very sticky. However, it becomes acceptable when $\rho$ is big enough. In contrast, the CTMC simulation scheme works well regardless of the degree of stickiness.
	
	We provide some intuition about why the Euler scheme flops in very sticky cases. Note that this method only simulates a discrete time process. If at some time point on the grid, say $i\Delta t$, the process is at zero, it moves to $\rho \Delta t$ at $(i+1)\Delta t$. If $\Delta t$ is small enough, the interpolated path of the Euler scheme should resemble the path of the continuous-time process the Euler scheme converges to in the limit.
	Thus, it is intuitively clear that the limiting continuous-time process is instantaneously reflected at zero and hence not sticky there. When the original diffusion is only mildly sticky, it is not very different from the reflected case, so the Euler scheme produces acceptable results. However, if the original diffusion is very sticky, the difference from the reflected case is big and the results of the Euler scheme become useless. The following figure shows one path simulated from the CTMC scheme and the other from the Euler scheme. The CTMC simulation scheme does not discretize time and it can generate the phenomenon that the process sticks to zero whereas the Euler scheme cannot.
	\begin{figure}[htbp]
		\centering
		\begin{subfigure}{.496\textwidth}
			\centering
			\includegraphics[width=\linewidth]{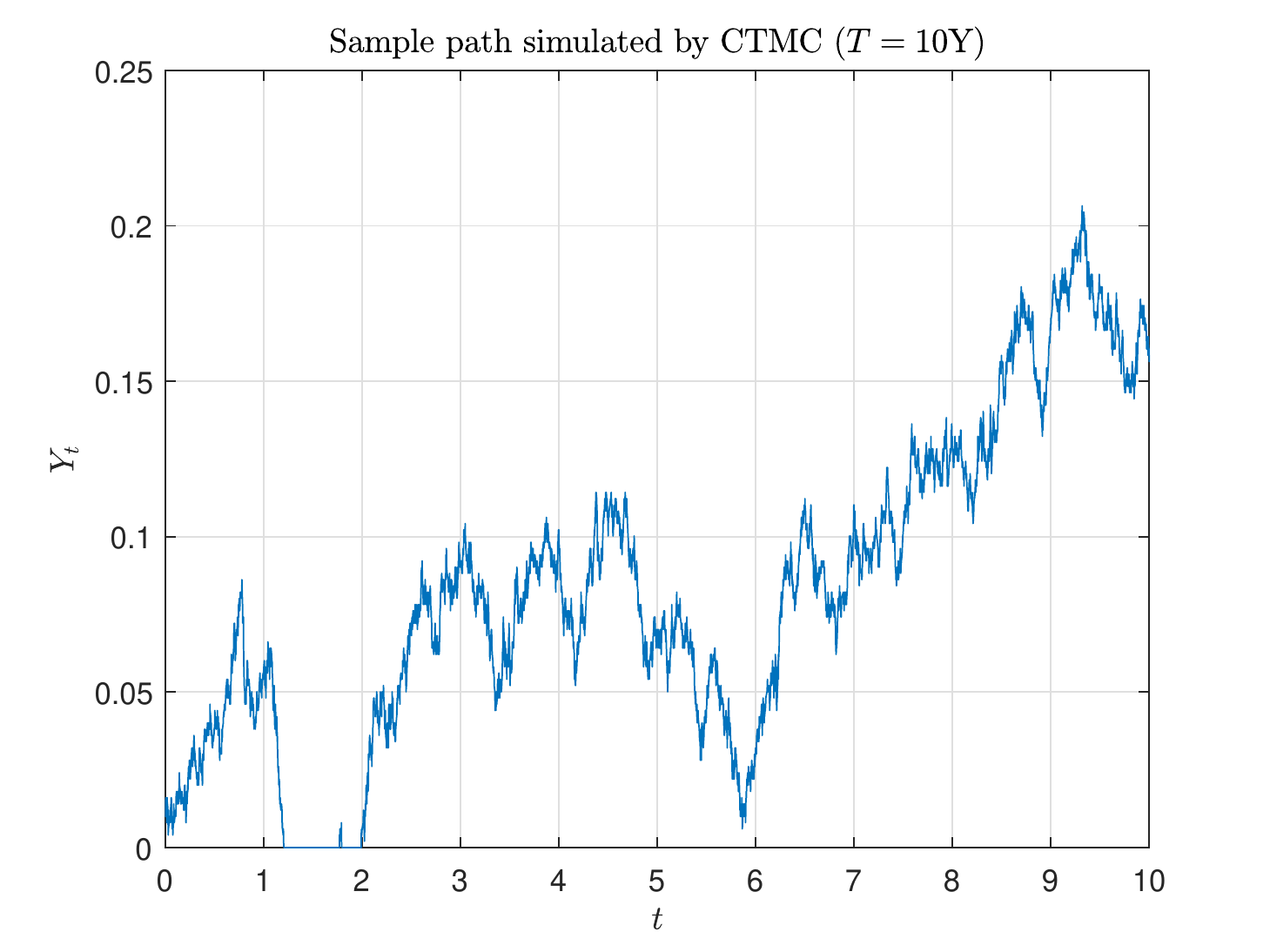}
		\end{subfigure}
		\begin{subfigure}{.496\textwidth}
			\centering
			\includegraphics[width=\linewidth]{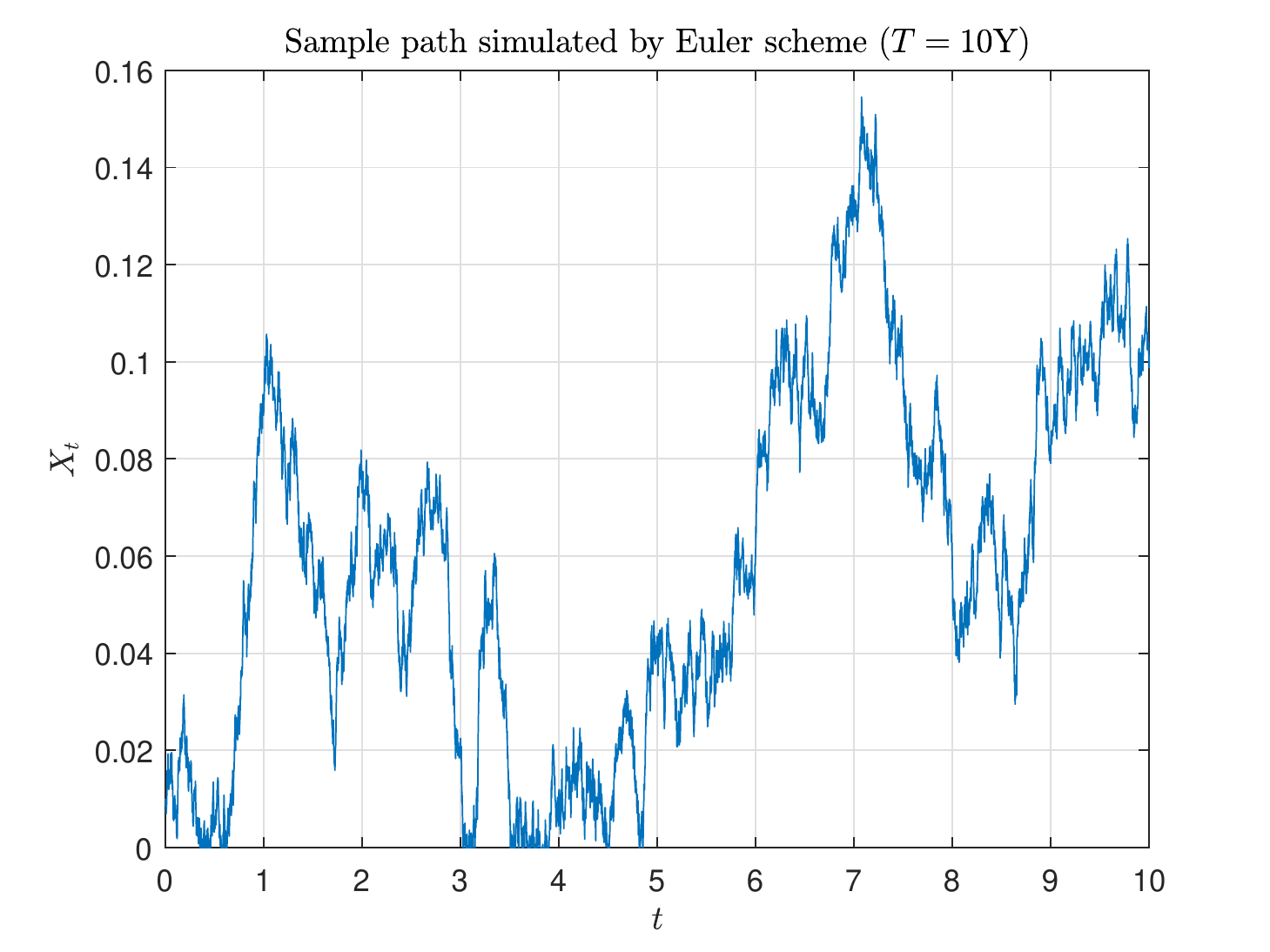}
		\end{subfigure}
		\caption{Sample paths over 10 years generated by the CTMC and Euler scheme. }
		\label{fig:sample_paths_model1}
	\end{figure}
	
	\section{Conclusion}
	\label{sec:summary}
	This paper develops CTMC approximation of one-dimensional diffusions with a sticky lower boundary. A direct finite difference approximation of the generator of the sticky diffusion on the boundary leads to the first scheme that only converges at first order. Matching the first and second moments of the infinitesimal changes at the sticky boundary, we obtain the second scheme which is second order. Under the CTMC model, calculation of the action of the Feynman-Kac operator and first passage probabilities can be obtained in closed-form using matrix exponentials. Our method has several nice features. First, it is applicable to sticky diffusions with general drift, volatility and killing rate. Second, it is computationally efficient. We show that when the extrapolation method of \cite{feng2008a} is used to solve the ODE system for the matrix exponential, our method outperforms a standard finite difference scheme that is often used for solving diffusion PDEs. Third, the CTMC can be used to simulate the sticky diffusion and it produces good results whereas the Euler scheme completely fails when the diffusion is very sticky.
	
	It is possible to construct CTMC approximations for multidimensional sticky diffusions with acceptable computational efficiency in low dimensions. However, as the state space must be discretized, the method suffers from the curse of dimensionality. For high-dimensional problems, it is important to develop efficient simulation schemes. As CTMC approximation can be used to produce sticky behavior at the boundary, we expect that combining CTMC approximation with traditional simulation schemes will lead to computationally feasible and accurate algorithms and we plan to work in this direction in future research.
	
	
	\section*{Acknowledgement}
	The research of the first two authors were supported by Hong Kong Research Grant Council General Research Fund Grant 14202117 and the research of the third author was supported by National Science Foundation of China Grant 11801423. We are grateful to Yutian Nie who kindly provided us with his code for implementing the eigenfunction expansion formula for bond pricing under the sticky OU short rate model.
	
	\appendix
	\section{Proofs}
	\label{app:proofs}
	
	\subsection{Proof of Theorem \ref{theorem:existence_and_uniqueness}}
	Suppose that the unique weak solution to \eqref{eq:general_diffusion_sde} and \eqref{eq:general_local_time_sde} for $\rho=\infty$ is given by $(X^1,B^1)$. Let
	\begin{equation*}
	\phi_t=t+\frac{1}{2\rho}L^l_t\left(X^1\right), \qquad T_t=\phi^{-1}_t, \qquad X_t=X^1_{T_t},\qquad B_t=B^1_{T_t}+\int_0^tI\left(X_s=l\right)dB^0_s,
	\end{equation*}
	where $B^0$ is a Brownian motion, which is defined on an extended probability space if needed, that is independent of $B^1$.
	The local time process is continuous and non-decreasing, hence $\phi_t$ is strictly increasing and continuous (see \cite{borodin2002}, Chapter II.13).
	This implies, that $T_t$ is also strictly increasing and continuous.
	Then $B_t$ is a continuous local martingale and
	\begin{equation*}
	\langle B\rangle_t=\langle B^1\rangle_{T_t}+\int_0^tI\left(X_s=l\right)ds=T_t+\int_0^tI\left(X_s=l\right)ds.
	\end{equation*}
	There also holds
	\begin{align*}
	T_t&=\int_0^{T_t}ds-\int_0^{T_t}I\left(X^1_s=0\right)ds=\int_0^{T_t}I\left(X^1_s>l\right)ds=\int_0^{T_t}I\left(X^1_s>l\right)\left(ds+\frac{1}{2\rho}dL^l_s\left(X^1\right)\right) \\
	&=\int_0^tI\left(X^1_{T_s}>l\right)d\phi_{T_s}=\int_0^tI\left(X_s>l\right)ds,
	\end{align*}
	where it was used that $dL_s^l(X)$ only increases for $X_s=l$ and the change of variable formula was applied.
	Therefor, $\langle B\rangle_t=t$ and by L\'evy's characterization, $B$ is a standard Brownian motion.
	Moreover,
	\begin{flalign*}
	X_t&=X_{T_t}^1=\int_0^{T_t}\mu\left(X_s^1\right)I\left(X_s^1>l\right)ds+\int_0^{T_t}\sigma\left(X_s^1\right)I\left(X_s^1>l\right)dB^1_s+\frac{1}{2}\int_0^{T_t}dL^l_s\left(X^1\right) \\
	&=\int_0^{T_t}\mu\left(X_s^1\right)I\left(X_s^1>l\right)\left(ds+\frac{1}{2\rho}dL^l_s\left(X^1\right)\right) \\
	&\qquad\qquad+\int_0^t\sigma\left(X_{T_s}^1\right)I\left(X_{T_s}^1>l\right)dB^1_{T_s}+\frac{1}{2}L^l_{T_t}\left(X^1\right)\\
	&=\int_0^{t}\mu\left(X_{T_s}^1\right)I\left(X_{T_s}^1>l\right)d\phi_{T_s}+\int_0^t\sigma\left(X_s\right)I\left(X_s>l\right)dB_s+\frac{1}{2}L^l_{T_t}\left(X^1\right) \\
	&=\int_0^{t}\mu\left(X_s\right)I\left(X_s>l\right)ds+\int_0^t\sigma\left(X_s\right)I\left(X_s>l\right)dB_s+\frac{1}{2}L^l_{t}\left(X\right),
	\end{flalign*}
	because $X_s=X^1_{T_s}$, $d\phi_{T_s}=ds+\frac{1}{2\rho}dL^l_s\left(X^1\right)=ds$ as $X^1$ is the unique weak solution of the reflecting case and $L_{T_t}^l(X^1)=L_t^l(X)$.
	This shows that $(X,B)$ solves \eqref{eq:general_diffusion_sde}.
	Furthermore,
	\begin{align*}
	\int_0^tI\left(X_s=l\right)ds&=\int_0^tI\left(X^1_{T_s}=l\right)d\phi_{T_s}=\int_0^{T_t}I\left(X^1_s=l\right)d\phi_s \\
	&=\int_0^{T_t}I\left(X^1_s=l\right)\left(ds+\frac{1}{2\rho}dL^l_s\left(X^1\right)\right)=\frac{1}{2\rho}L^l_t\left(X\right),
	\end{align*}
	where the first term vanishes because for $X^1$ there holds $I(X^1_s=l)ds=0$.
	The continuity of $X$ results from the continuity of $X^1$ and $T$.
	Hence, $(X,B)$ also solves \eqref{eq:general_local_time_sde}.
	
	The next step is to show the uniqueness in law of the solution $X$.
	We reset the notation and suppose that $(X,B)$ solves \eqref{eq:general_diffusion_sde} and \eqref{eq:general_local_time_sde}.
	Define
	\begin{equation*}
	T_t=\int_0^tI\left(X_s>l\right)ds
	\end{equation*}
	for $t\geq 0$.
	Then $T_t$ is continuous and strictly increasing almost surely.
	This can be shown by contradiction.
	Assume the $T_t$ is not strictly increasing, then there exists a set
	\begin{equation*}
	\Gamma=\left\{\omega\in\Omega:T_{t_1}=T_{t_2}\ \textrm{for some}\ 0<t_1<t_2\right\},
	\end{equation*}
	with $\mathbb{P}(\Gamma)>0$ and $t_1, t_2$ depending on $\omega$.
	Now $T_{t_1}=T_{t_2}$ implies that the process stays at the boundary for all $s\in[t_1,t_2]$ and so
	\begin{equation*}
	\Gamma\subset\left\{\omega:\int_{t_1}^{t_2}dL^l_s\left(X\right)=L^l_{t_2}\left(X\right)-L_{t_1}^l\left(X\right)>0\ \textrm{for some}\ 0<t_1<t_2\right\},
	\end{equation*}
	i.e. the local time increases between $t_1$ and $t_2$.
	On this set, there holds $I(X_s>l)=0$ for all $s\in[t_1,t_2]$ and hence
	\begin{equation*}
	\Gamma\subset\left\{\omega:X_{t_2}=X_{t_1}+L^l_{t_2}\left(X\right)-L^l_{t_1}\left(X\right)>X_{t_1}\ \textrm{for some}\ 0<t_1<t_2\right\},
	\end{equation*}
	as the drift and volatility vanish.
	This is a contradiction to $I(X_s>l)=0$ and so in summary, $T_t$ is strictly increasing almost surely.
	The inverse of $T_t$, given by
	\begin{equation*}
	\phi_t=\inf\left\{s\geq 0:T_s>t\right\},
	\end{equation*}
	is therefor also continuous and almost surely finite.
	As $X$ and $\phi$ are continuous, it follows that $X$ is constant on every interval $[\phi_{t-},\phi_t]$ and so $\phi$ is in synchronization with $X$ (see \cite{jacod1979}, Definition 10.13. Therein it is called adaptedness of $X$ to the time change $\phi$).
	
	Now set $X^1_t=X_{\phi_t}$.
	Then $X^1$ is a continuous semimartingale (see \cite{jacod1979}, Corollary 10.12 and Lemma 10.15) and using \eqref{eq:general_local_time_sde}, there holds
	\begin{equation*}
	t=T_{\phi_t}=\int_0^{\phi_t}I\left(X_s>l\right)ds=\phi_t-\int_0^{\phi_t}I\left(X_s=l\right)ds=\phi_t-\frac{1}{2\rho}L^l_{\phi_t}\left(X\right)=\phi_t-\frac{1}{2\rho}L^l_t\left(X^1\right)
	\end{equation*}
	and so
	\begin{equation*}
	\phi_t=t+\frac{1}{2\rho}L^l_t\left(X^1\right),
	\end{equation*}
	which shows that $\phi$ is also strictly increasing.
	Let $B^1_t=\int_0^{\phi_t}I(X_s>l)dB_s$.
	Then $B_t^1$ is a continuous local martingale with
	\begin{equation*}
	\langle B^1\rangle_t=\int_0^{\phi_t}I\left(X_s>l\right)ds=T_{\phi_t}=t
	\end{equation*}
	and hence $B^1$ is a Brownian motion by L\'evy's criterion.
	Furthermore, by \eqref{eq:general_diffusion_sde} there follows
	\begin{align*}
	X^1_t&=X_0+\int_0^{\phi_t}\mu\left(X_s\right)I\left(X_s>l\right)ds+\int_0^{\phi_t}\sigma\left(X_s\right)I\left(X_s>l\right)dB_s+\frac{1}{2}L^l_{\phi_t}\left(X\right) \\
	&=X_0+\int_0^{t}\mu\left(X^1_s\right)I\left(X^1_s>l\right)d\phi_s+\int_0^t\sigma\left(X^1_s\right)I\left(X^1_s>l\right)dB^1_s+\frac{1}{2}L^l_{t}\left(X^1\right),
	\end{align*}
	by the change of variables formula and further
	\begin{align*}
	dX^1_t&=\mu\left(X^1_t\right)I\left(X^1_t>l\right)d\phi_t+\sigma\left(X^1_t\right)I\left(X^1_t>l\right)dB^1_t+\frac{1}{2}dL^l_t\left(X^1\right) \\
	&=\mu\left(X^1_t\right)I\left(X^1_t>l\right)\left(dt+\frac{1}{2\rho}dL^l_t\left(X^1\right)\right)+\sigma\left(X^1_t\right)I\left(X^1_t>l\right)dB^1_t+\frac{1}{2}dL^l_t\left(X^1\right) \\
	&=\mu\left(X^1_t\right)I\left(X^1_t>l\right)dt+\sigma\left(X^1_t\right)I\left(X^1_t>l\right)dB^1_t+\frac{1}{2}dL^l_t\left(X^1\right).
	\end{align*}
	Moreover, we have
	\begin{equation*}
	\int_0^tI\left(X^1_s=l\right)ds=\int_0^tI\left(X_{\phi_s}=l\right)dT_{\phi_s}=\int_0^{\phi_t}I\left(X_s=l\right)dT_s=0.
	\end{equation*}
	The last two equations showed that $(X^1,B^1)$ is a unique weak solution to the system of SDEs \eqref{eq:general_diffusion_sde} and \eqref{eq:general_local_time_sde} for $\rho=\infty$.
	Since $X^1$ is the unique solution to the reflecting SDE and $X_t=X_{\phi_{T_t}}=X^1_{T_t}$, the law of $X$ is also unique.
	\cite{cherny2002}, Theorem 3.1 states that the uniqueness in law for $X$ implies joint uniqueness in law for $(X,B)$.\qed
	
	\subsection{Proof of Proposition \ref{prop:SLeigen}}
	Application of the Liouville transform $h_1(x)=\int_l^x1/\sigma\left(z\right)dz$ and a linear transformation $h_2(y)=-2y/B+1$, where $B=h_1(r)$, changes the problem into
	\begin{align}
	&-\frac{2}{B^2}\psi^{\prime\prime}\left(z\right)+q\left(z\right)\psi\left(z\right)=\lambda\psi\left(z\right), \qquad z\in\left(-1,1\right), \\
	&\psi\left(-1\right)=0, \qquad -\frac{2\rho}{Bh_3\left(0\right)}\psi^\prime\left(1\right)=\left(k\left(l\right)-\frac{h_4\left(0\right)}{\rho}-\lambda\right)\psi\left(1\right),
	\end{align}
	where
	\begin{align}
	\psi\left(z\right)&=\frac{\tilde{\varphi}\left(h_1^{-1}\left(h_2^{-1}\left(z\right)\right)\right)}{\sqrt{\sigma\left(h_1^{-1}\left(h_2^{-1}\left(z\right)\right)\right)s\left(h_1^{-1}\left(h_2^{-1}\left(z\right)\right)\right)}}, \label{eq:proof_lemma21_1}\\
	q\left(z\right)&=U\left(h_1^{-1}\left(h_2^{-1}\left(z\right)\right)\right)\qquad\ \textrm{and}\ U\ \textrm{is the potential function,} \\
	h_3\left(y\right)&=\frac{\left(h_1^{-1}\left(y\right)\right)^\prime}{\sqrt{\sigma\left(h_1^{-1}\left(y\right)\right)}}, \qquad\qquad h_4\left(y\right)=\frac{\left(\sigma\left(h_1^{-1}\left(y\right)\right)s\left(h_1^{-1}\left(y\right)\right)\right)^\prime}{2\sigma\left(h_1^{-1}\left(y\right)\right)s\left(h_1^{-1}\left(y\right)\right)},
	\end{align}
	and $\tilde{\varphi}$ and $\psi$ are the eigenfunctions of the original and transformed problem.
	It should be noted that $\tilde{\varphi}$ denotes the eigenfunction and $\varphi$ the normalized eigenfunction of the original problem.
	The potential function is defined in (3.31), \cite{linetsky2008}.
	
	This setting resembles the setting in \cite{altinisik2004}, with $a_1=a_2=2/B^2$ and $\gamma_i=\delta_i=1$ for $i=1,2$ allows that the transmission condition to disappear and the solution and its derivatives are continuous in $[-1,1]$.
	Following further the cited reference, let $\alpha_1=1$ and $\alpha_2=0$ and fix $\beta_1^\prime=-1$, $\beta_2^\prime=0$ and $\beta_2\neq 0$.
	It should be noted that because $\beta^\prime_2=0$, it is unimportant which value $\beta_1^\prime$ takes and normalization to 1 is a simplification of notation.
	
	The results in \cite{altinisik2004}, Section 4 then shows that the $k$-th eigenvalue of the problem in Liouville normal form satisfies the following asymptotic representation
	\begin{equation}
	\lambda_k=\frac{a_1^2k^2\pi^2}{4}-a_1\beta_2+\frac{1}{2a_1}\int_{-1}^1q\left(z\right)dz+O\left(\frac{1}{k}\right)\label{eq:SL_problem_asymptotics_lambda}
	\end{equation}
	and so $C_1k^2\leq\lambda_k\leq C_2k^2$ for constants $C_1,C_2>0$ independent of $k$. Using Theorem 3.1 in \cite{altinisik2004} with trigonometric calculations, we have for any $x\in[-1,1]$,
	\begin{equation} \psi_k\left(x\right)=-\frac{a_1}{\sqrt{\lambda_k}}\sin\frac{\sqrt{\lambda_k}\left(x+1\right)}{a_1}+\frac{1}{a_1\sqrt{\lambda_k}}\int_{-1}^x\sin\frac{\sqrt{\lambda_k}\left(x-z\right)}{a_1}q\left(z\right)\psi_k\left(z\right)dz.\label{eq:SL_eigenfunction}
	\end{equation}
	Then
	\begin{equation}
	\left\vert\psi_k\left(x\right)\right\vert\leq \frac{a_1}{\sqrt{\lambda_k}}+\frac{1}{a_1\sqrt{\lambda_k}}\int_{-1}^x\left\vert q\left(z\right)\right\vert \left\vert\psi_k\left(z\right)\right\vert dz\leq \frac{C_3}{\sqrt{\lambda_k}}+\frac{C_3}{\sqrt{\lambda_k}}\int_{-1}^x\left\vert q\left(z\right)\right\vert \left\vert\psi_k\left(z\right)\right\vert dz,
	\end{equation}
	for some constant $C_3>0$.
	The Gronwall inequality shows that
	\begin{equation}
	\left\vert\psi_k\left(x\right)\right\vert\leq \frac{C_3}{\sqrt{\lambda_k}}\exp\left(\int_0^B\frac{C_3\left\vert q\left(z\right)\right\vert}{\sqrt{\lambda_k}}dz\right)\leq\frac{C_4}{k},
	\end{equation}
	for a constant $C_4>0$ independent of $k$ and $x$ using the asymptotic representation of $\lambda_k$ in \eqref{eq:SL_problem_asymptotics_lambda}.
	
	Furthermore, there holds for the first derivative
	\begin{align}
	\left\vert\psi_k^\prime\left(x\right)\right\vert&\leq\left\vert-\cos\frac{\sqrt{\lambda_k}\left(x+1\right)}{a_1}+\frac{1}{a_1^2}\int_{-1}^x\cos\frac{\sqrt{\lambda_k}\left(x-z\right)}{a_1}q\left(z\right)\psi_k\left(z\right)dz\right\vert \\
	&\leq 1+\frac{C_5}{k}\int_{-1}^x\left\vert q\left(z\right)\right\vert dz\leq C_6,
	\end{align}
	where $C_5,C_6>0$ are constants independent of $k$ and $x$.
	
	Similar bounds can be established for further derivatives, i.e. for $j=2,3,4$, by differentiation of \eqref{eq:SL_eigenfunction}.
	Hence, one derives that $\vert\psi_k^{(j)}(x)\vert\leq C_{7}k^{j-1}$.
	Finally, by Theorem 4.2 in \cite{mukhtarov2004} there follows $\Vert\psi_k\Vert_2=\frac{C_{8}}{\pi k}+O(1/k^2)\geq C_{9}/k$.
	The normalized eigenfunctions then have the following asymptotic representation
	\begin{equation}
	\frac{\psi_k\left(x\right)}{\left\Vert\psi_k\right\Vert_2}\leq\frac{C_2/k}{C_{9}/k}\leq C_{10}.
	\end{equation}
	
	Using the relationship between $\tilde{\varphi}$ and $\psi$ in \eqref{eq:proof_lemma21_1} and the fact that $\sigma$ and $s$ are bounded on $\mathbb{S}$, one derives $\tilde{\varphi}_k(x)=O(1/k)$ and $\Vert\tilde{\varphi}_k\Vert_2\geq C_{11}/k$.
	This implies for all $x\in\mathbb{S}$, $\varphi_k(x)=\tilde{\varphi}_k(x)/\Vert\tilde{\varphi}_k\Vert=O(1)$. Using similar arguments, one can prove $\varphi_k^{(j)}(x)=O(k^j)$ for all $x\in\mathbb{S}$, $j=1,\ldots,4$.\qed
	
	\subsection{Proofs for Convergence Rate Analysis}
	We will restate some results from \cite{zhang2019} and \cite{zhang2019a} without proof but with adjustments to incorporate the sticky boundary behavior at the left boundary.
	
	\begin{lemma}
		For any $f,g:\mathbb{S}_n\to\mathbb{R}$ with $g(x_{n+1})=0$, we have
		\begin{equation}
		\sum_{x\in\mathbb{S}_n^\circ}g\left(x\right)\delta^-x\nabla^-f\left(x\right)=-\sum_{x\in\mathbb{S}_n^-}f\left(x\right)\delta^+x\nabla^+g\left(x\right)-g\left(x_0\right)f\left(x_0\right).
		\end{equation}
		\label{lemma:summation_by_parts}
	\end{lemma}
	
	Under Assumption \ref{assumption:rb}, it can be seen that
	\begin{equation}
	\underset{x,y\in\left(l,r\right)}{\sup}\ \left\vert\frac{\partial^i}{\partial x^i}\frac{\partial^j}{\partial y^j}p\left(t,x,y\right)\right\vert<\infty,\label{eq:result_lemma2_OR_paper}
	\end{equation}
	for $i,j=0,1,2$ still holds because of the results from Sturm-Liouville theory and the proof of Lemma 2 in \cite{zhang2019}.
	The use of \eqref{eq:result_lemma2_OR_paper} is to show claims of the form $\vert g(x)\vert\leq Ch_n^\beta$ for $\beta=0,1,2$ such that the constant $C>0$ is independent of $x$ and $n$.
	The application of this result will not be mentioned explicitly below.
	
	\bigskip
	\noindent\emph{Proof of Proposition \ref{prop:convergence_speed_measure_density}}:
	We will prove the claim first for $x_0$, then for $x_1$ and finally for all $x=x_2,\ldots,x_n$.
	First, note that
	\begin{align*}
	M_n\left(x_0\right)-M\left(x_0\right)&=M\left(x_0\right)\exp\left(\frac{\alpha}{\sigma^2\left(x_0\right)}\delta^+x_0\right)-M\left(x_0\right) \\
	&=M\left(x_0\right)\frac{\alpha}{\sigma^2\left(x_0\right)}\delta^+x_0+O\left(h_n^2\right).
	\end{align*}
	In a second step, applying the logarithm to \eqref{eq:approximation_speed_density_interior} for $x=x_1$ yields
	\begin{align}
	\log m_n\left(x_1\right)&=\log M_n\left(x_0\right)+\log\beta+\log\rho+\frac{\mu\left(x_1\right)}{\sigma^2\left(x_1\right)}\delta^+x_1+O\left(h_n^2\right) \\
	&\qquad+\log\frac{2}{\sigma^2\left(x_1\right)}+\left(\frac{\mu\left(x_0\right)}{\sigma^2\left(x_0\right)}+\frac{\mu\left(x_1\right)}{\sigma^2\left(x_1\right)}\right)\delta^+x_0-\left(\frac{\mu\left(x_0\right)}{\sigma^2\left(x_0\right)}+\frac{\mu\left(x_1\right)}{\sigma^2\left(x_1\right)}\right)\delta^+x_0 \\
	&=\log M\left(x_0\right)+\frac{\alpha}{\sigma^2\left(x_0\right)}\delta^+x_0+\log\beta-\log M\left(x_0\right)+O\left(h_n^2\right) \\
	&\qquad+\log m\left(x_1\right)+\frac{\mu\left(x_1\right)}{\sigma^2\left(x_1\right)}\left(\delta^+x_1-\delta^-x_1\right)-\frac{\mu\left(x_0\right)}{\sigma^2\left(x_0\right)}\delta^+x_0 \\
	&=\log m\left(x_1\right)+\frac{\mu\left(x_1\right)}{\sigma^2\left(x_1\right)}\left(\delta^+x_1-\delta^-x_1\right)+O\left(h_n^2\right),
	\end{align}
	where we used that
	\begin{equation}
	\log m\left(x_1\right)=\log\frac{2}{\sigma\left(x_1\right)}+\int_{x_0}^{x_1}\frac{2\mu\left(y\right)}{\sigma^2\left(y\right)}dy=\log\frac{2}{\sigma\left(x_1\right)}+\left(\frac{\mu\left(x_0\right)}{\sigma^2\left(x_0\right)}+\frac{\mu\left(x_1\right)}{\sigma^2\left(x_1\right)}\right)\delta^+x_0+O\left(h_n^2\right)
	\end{equation}
	and
	\begin{align}
	&\frac{\alpha}{\sigma^2\left(x_0\right)}\delta^+x_0+\log\beta-\frac{\mu\left(x_0\right)}{\sigma^2\left(x_0\right)}\delta^+x_0 \\
	&\qquad=\begin{cases} 0 &\qquad\textrm{for Scheme 1,} \\ \frac{\rho}{\sigma^2\left(x_0\right)}\delta^+x_0-\frac{\rho-\mu\left(x_0\right)}{\sigma^2\left(x_0\right)}\delta^+x_0-\frac{\mu\left(x_0\right)}{\sigma^2\left(x_0\right)}\delta^+x_0+O\left(h_n^2\right) & \qquad\textrm{for Scheme 2}. \end{cases}\label{eq:prop_41_eq_1}
	\end{align}
	Thus these terms equal $0$ for Scheme 1 and are $O(h_n^2)$ for Scheme 2.
	
	Now let $x\in\{x_2,\ldots,x_n\}$, then applying the logarithm to \eqref{eq:approximation_speed_density_interior} and using the Taylor expansion for the logarithm yields,
	\begin{flalign}
	\log m_n\left(x\right)&=\log M_n\left(x_0\right)+\log\beta+\log \rho+\frac{\mu\left(x\right)}{\sigma^2\left(x\right)}\left(\delta^+x-\delta^-x\right)+\frac{\mu\left(x_1\right)}{\sigma^2\left(x_1\right)}\delta^-x_1 \\
	&\qquad+\log\frac{2}{\sigma^2\left(x\right)}+\sum_{y=x_1}^{x^-}\left(\frac{\mu\left(y\right)}{\sigma^2\left(y\right)}+\frac{\mu\left(y^+\right)}{\sigma^2\left(y^+\right)}\right)\delta y+\left(\frac{\mu\left(x_0\right)}{\sigma^2\left(x_0\right)}+\frac{\mu\left(x_1\right)}{\sigma^2\left(x_1\right)}\right)\delta^+x_0 \\	
	&\qquad-\left(\frac{\mu\left(x_0\right)}{\sigma^2\left(x_0\right)}+\frac{\mu\left(x_1\right)}{\sigma^2\left(x_1\right)}\right)\delta^+x_0+O\left(h_n^2\right) \\
	&=\log M_n\left(x_0\right)+\log\beta+\log \rho+\frac{\mu\left(x\right)}{\sigma^2\left(x\right)}\left(\delta^+x-\delta^-x\right) \\
	&\qquad+\log m\left(x\right)-\frac{\mu\left(x_0\right)}{\sigma^2\left(x_0\right)}\delta^+x_0+O\left(h_n^2\right),
	\end{flalign}
	where we used
	\begin{equation}
	\sum_{y=x_1}^{x^-}\left(\frac{\mu\left(y\right)}{\sigma^2\left(y\right)}+\frac{\mu\left(y^+\right)}{\sigma^2\left(y^+\right)}\right)\delta y+\left(\frac{\mu\left(x_0\right)}{\sigma^2\left(x_0\right)}+\frac{\mu\left(x_1\right)}{\sigma^2\left(x_1\right)}\right)\delta^+x_0=\int_{x_0}^x\frac{2\mu\left(y\right)}{\sigma^2\left(y\right)}dy+O\left(h_n^2\right)
	\end{equation}
	and
	\begin{equation}
	\log m\left(x\right)=\log\frac{2}{\sigma^2\left(x\right)}+\int_{x_0}^x\frac{2\mu\left(y\right)}{\sigma^2\left(y\right)}dy.
	\end{equation}
	Further using $M\left(x_0\right)=\frac{1}{\rho}$ and \eqref{eq:approximation_speed_measure_x0} yields
	\begin{equation}\label{eq:proposition_41_eq1}
	\log m_n\left(x\right)=\frac{\alpha}{\sigma^2\left(x_0\right)}\delta^+x_0+\log\beta-\frac{\mu\left(x_0\right)}{\sigma^2\left(x_0\right)}\delta^+x_0+\frac{\mu\left(x\right)}{\sigma^2\left(x\right)}\left(\delta^+x-\delta^-x\right)+\log m\left(x\right)+O\left(h_n^2\right).
	\end{equation}
	Therefore, by using \eqref{eq:prop_41_eq_1}, we obtain from the previous result
	\begin{equation*}
	m_n\left(x\right)-m\left(x\right)=m\left(x\right)\left(\frac{\mu\left(x\right)}{\sigma^2\left(x\right)}\left(\delta^+x-\delta^-x\right)\right)+O\left(h_n^2\right).
	\end{equation*}
	Hence, \eqref{eq:convergence_speed_density_interior} holds for all $x\in\mathbb{S}_n^\circ$ and we proved the convergence for all $x\in\mathbb{S}_n^-$.
	The error estimate for the scale function can be obtained from similar calculations and it is omitted here.\qed
	
	\bigskip
	Proposition \ref{prop:convergence_speed_measure_density} implies the following.
	\begin{corollary}\label{cor:lower_upper_bounds_sn_mn}
		Under Assumption \ref{assumption:rb}, for $n$ sufficiently large, there exist constants $C_1,C_2>0$ independent of $n$ and $x\in\mathbb{S}_n,y\in\mathbb{S}_n^\circ$, such that
		\begin{equation}
		C_1\leq s_n\left(x\right)\leq C_2, \quad C_1\leq m_n\left(y\right)\leq C_2, \quad C_1\leq M_n\left(x_0\right)\leq C_2.
		\end{equation}
	\end{corollary}

	\begin{lemma}\label{lemma:square_k_lower_upper_bounds}
		Under Assumptions \ref{assumption:rb} and \ref{assumption:bounded_grids}, there exists a constant $C>0$, independent of $k$ and $n$, such that for $h_n\in(0,\delta)$, where $\delta$ is small enough, the following holds
		\begin{equation}
		\lambda_k^n\leq Ck^2.\label{eq:upper_lower_bound_approximate_eigenvalues}
		\end{equation}
	\end{lemma}
	
	\begin{proof}
		Let the matrix $\mathbb{M}_n$ be a diagonal matrix with entries $\mathbb{M}_{n,i,i}=M_n(x_i)$ for $i=0,\ldots,n$.
		Calculation of $\mathbb{M}_n\mathbb{G}_n$ and the choice of $M_n(x)$ as stated in Section \ref{subsec:eigenfunction_expansion_ctmc} implies that
		\begin{align*}
		\frac{\rho}{\delta^+x_0}\beta M_n\left(x_0\right)&=\frac{-\mu\left(x_1\right)\delta^+x_1+\sigma^2\left(x_1\right)}{2\delta^-x_1\delta x_1}M_n\left(x_1\right) \\
		\frac{\mu\left(x\right)\delta^-x+\sigma^2\left(x\right)}{2\delta^+x\delta x}M_n\left(x\right)&=\frac{-\mu\left(x^+\right)\delta^+x^++\sigma^2\left(x^+\right)}{2\delta^-x^+\delta x^+}M_n\left(x^+\right),\qquad\textrm{for}\  x=x_1,\ldots,x_{n-1}.
		\end{align*}
		Hence, $\mathbb{M}_n\mathbb{G}_n$ is symmetric and therefor $\mathbb{M}_n^{1/2}\mathbb{G}_n\mathbb{M}_n^{-1/2}$ is also symmetric.
		Furthermore, it can be seen that $\mathbb{M}_n^{1/2}\mathbb{G}_n\mathbb{M}_n^{-1/2}$ is similar to $\mathbb{G}_n$ and hence both matrices have the same eigenvalues.
		The min-max principle derived in \cite{zhang2019a}, Section 3, shows that
		\begin{equation}
		\lambda_k^n=\underset{U_k}{\min}\ \underset{f\in U_k}{\max}\ \frac{-f^T\mathbb{M}_n^{1/2}\mathbb{G}_n\mathbb{M}_n^{-1/2}f}{f^Tf}=\underset{U_k}{\min}\ \underset{f\in U_k}{\max}\ \frac{\left(f,-G_nf\right)_n}{\left(f,f\right)_n},\label{eq:min_max_principle_eigenvalue}
		\end{equation}
		where $U_k$ denotes a $k$-dimensional subspace of functions defined on $\mathbb{S}_n$ with boundary condition $f(x_{n+1})=0$.
		
		The upper boundary for $\lambda_k^n$ can be derived in the following way.
		Firstly, there holds
		\begin{align*}
		\left(f,-G_nf\right)_n&=-\sum\limits_{x\in\mathbb{S}_n^-}f\left(x\right)\left(G_nf\left(x\right)\right)M_n\left(x\right) \\
		&=-\sum\limits_{x\in\mathbb{S}_n^\circ}f\left(x\right)\frac{1}{m_n\left(x\right)}\frac{\delta^-x}{\delta x}\nabla^-\left(\frac{1}{s_n\left(x\right)}\nabla^+f\left(x\right)\right)m_n\left(x\right)\delta x \\
		&\qquad+\sum\limits_{x\in\mathbb{S}_n^-}k\left(x\right)f\left(x\right)^2M_n\left(x\right)-\rho\beta f\left(x_0\right)\nabla^+f\left(x_0\right)M_n\left(x_0\right) \\
		&=-\sum\limits_{x\in\mathbb{S}_n^\circ}f\left(x\right)\delta^-x\nabla^-\left(\frac{1}{s_n\left(x\right)}\nabla^+f\left(x\right)\right)+\sum\limits_{x\in\mathbb{S}_n^-}k\left(x\right)f\left(x\right)^2M_n\left(x\right) \\
		&\qquad-\rho\beta f\left(x_0\right)\nabla^+f\left(x_0\right)M_n\left(x_0\right) \\
		&=\sum\limits_{x\in\mathbb{S}_n^-}\frac{\delta^+x}{s_n\left(x\right)}\left(\nabla^+f\left(x\right)\right)^2+\sum\limits_{x\in\mathbb{S}_n^-}k\left(x\right)f\left(x\right)^2M_n\left(x\right)+\frac{1-\beta}{s_n\left(x_0\right)}f\left(x_0\right)\nabla^+f\left(x_0\right),
		\end{align*}
		where Lemma \ref{lemma:summation_by_parts} was used and $1/s_n(x_0)=\rho M_n(x_0)$.
		Furthermore, by noting that for Scheme 1, $1-\beta=0$, and for Scheme 2, $1-\beta=O(h_n)$, we derive
		\begin{align}
		&\sum_{x\in\mathbb{S}_n^-}\frac{\delta^+x}{s_n\left(x\right)}\left(\nabla^+f\left(x\right)\right)^2+\frac{1-\beta}{s_n\left(x_0\right)}f\left(x_0\right)\nabla^+f\left(x_0\right) \\
		&\qquad\leq\frac{C_1}{h_n}\sum_{x\in\mathbb{S}_n^-}\left(f\left(x^+\right)-f\left(x\right)\right)^2+\frac{C_2}{h_n}\left\vert1-\beta\right\vert\left\vert f\left(x_0\right)\right\vert\left\vert f\left(x_1\right)-f\left(x_0\right)\right\vert \\
		&\qquad\leq\frac{C_1}{h_n}\sum_{x\in\mathbb{S}_n^-}\left(f\left(x^+\right)-f\left(x\right)\right)^2+C_3\left\vert f\left(x_0\right)\right\vert\left\vert f\left(x_1\right)-f\left(x_0\right)\right\vert \\
		&\qquad\leq\frac{C_1}{h_n}\sum_{x\in\mathbb{S}_n^-}\left(f\left(x^+\right)-f\left(x\right)\right)^2+C_4\left(f\left(x_0\right)^2h_n+\frac{\left(f\left(x_0\right)-f\left(x_1\right)\right)^2}{h_n}\right) \\
		&\qquad\leq\frac{C_5}{h_n}\sum_{x\in\mathbb{S}_n^-}\left(f\left(x^+\right)-f\left(x\right)\right)^2+C_4f\left(x_0\right)^2h_n,
		\end{align}		
		and the constants $C_1,\ldots,C_5>0$ are independent of $n$ and $f$.
		Note that with $f(x_{n+1})=0$, the following holds
		\begin{align}
		\sum_{x\in\mathbb{S}_n^-}\left(f\left(x^+\right)-f\left(x\right)\right)^2&=\sum_{x\in\mathbb{S}_n^-}f\left(x^+\right)^2-2\sum_{x\in\mathbb{S}_n^-}f\left(x^+\right)f\left(x\right)+\sum_{x\in\mathbb{S}_n^-}f\left(x\right)^2 \\
		&=2\sum_{x\in\mathbb{S}_n^\circ}f\left(x\right)^2-\sum_{x\in\mathbb{S}_n^\circ}f\left(x\right)f\left(x^-\right)-\sum_{x\in\mathbb{S}_n^\circ}f\left(x\right)f\left(x^+\right) \\
		&\qquad\qquad+f\left(x_0\right)^2-f\left(x_0\right)f\left(x_1\right) \\
		&=\sum_{x\in\mathbb{S}_n^\circ}f\left(x\right)\left(-f\left(x^-\right)+2f\left(x\right)-f\left(x^+\right)\right)+f\left(x_0\right)^2-f\left(x_0\right)f\left(x_1\right).
		\end{align}
		Then
		\begin{align}
		&\sum\limits_{x\in\mathbb{S}_n^-}\frac{\delta^+x}{s_n\left(x\right)}\left(\nabla^+f\left(x\right)\right)^2+\frac{1-\beta}{s_n\left(x_0\right)}f\left(x_0\right)\nabla^+f\left(x_0\right) \\
		&\qquad\leq\frac{C_5}{h_n}\sum_{x\in\mathbb{S}_n^\circ}f\left(x\right)\left(-f\left(x^-\right)+2f\left(x\right)-f\left(x^+\right)\right)+\frac{C_5}{h_n}f\left(x_0\right)\left(f\left(x_0\right)-f\left(x_1\right)\right)+C_4f\left(x_0\right)^2h_n \\
		&\qquad=C_5\frac{1}{h_n}f_n^TAf_n+C_4f\left(x_0\right)^2h_n,
		\end{align}
		where $f_n=(f(x_0),f(x_1),\ldots,f(x_n))^T\in\mathbb{R}^{n+1}$ and $A$ is a $n+1$ by $n+1$ tridiagonal matrix with diagonal elements 2 (the first diagonal entry is equal to 1) and off diagonal elements -1.
		It can easily be seen that
		\begin{equation*}
		0\leq\frac{\sum_{x\in\mathbb{S}_n^-}k\left(x\right)f\left(x\right)^2M_n\left(x\right)}{\left(f,f\right)_n}\leq C_6\frac{\sum_{x\in\mathbb{S}_n^-}f\left(x\right)^2M_n\left(x\right)}{\sum_{x\in\mathbb{S}_n^-}f\left(x\right)^2M_n\left(x\right)}\leq C_6,
		\end{equation*}
		for some constant $C_6>0$ independent of $n$ and $f$ as all terms are positive and $k(x)$ is bounded.
		Lastly,
		\begin{equation*}
		\left(f,f\right)_n=\sum_{x\in\mathbb{S}_n^-}f\left(x\right)^2M_n\left(x\right) \geq\sum_{x\in\mathbb{S}_n^\circ}f\left(x\right)^2m_n\left(x\right)\delta x+f\left(x_0\right)^2M_n\left(x_0\right)h_n \geq C_7h_nf_n^Tf_n,
		\end{equation*}
		for a constant $C_7>0$ independent of $n$ and $f$ as $h_n\in(0,\delta)$ with $\delta$ small enough.
		Hence,
		\begin{equation*}
		\frac{\left(f,-G_nf\right)_n}{\left(f,f\right)_n}\leq\frac{C_5\frac{1}{h_n}f_n^TAf_n}{C_7h_nf_n^Tf_n}+\frac{C_4h_nf\left(x_0\right)^2}{C_7h_nf_n^Tf_n}+C_6.
		\end{equation*}
		Putting these results into \eqref{eq:min_max_principle_eigenvalue}, one obtains
		\begin{equation}
		\lambda_k^n\leq\frac{C_5}{C_7h_n^2}\underset{U_k}{\min}\ \underset{f\in U_k}{\max}\ \frac{f_n^TAf_n}{f_n^Tf_n}+C_8.
		\end{equation}
		As $\min_{U_k}\max_{f\in U_k}\frac{f_n^TAf_n}{f_n^Tf_n}$ is the $k$-th eigenvalue of the matrix $A$, one can use Theorem 2 in \cite{yueh2005} and obtain that
		\begin{equation}
		\lambda_k(A)=4\sin^2\frac{(2k-1)\pi}{4n+6}\leq\frac{4k^2\pi^2}{\left(n+1\right)^2},\quad k=1,2,\ldots,n+1.
		\end{equation}
		This now shows that
		\begin{equation*}
		\lambda_k^n\leq\frac{C_5}{C_7h_n^2}\frac{4k^2\pi^2}{\left(n+1\right)^2}+C_8\leq\frac{4\pi^2C_5}{C_7C^2_{9}}k^2+C_8\leq C_{10}k^2,
		\end{equation*}
		as by Assumption \ref{assumption:bounded_grids}, we have $C_{9}\leq h_n(n+1)$.
	\end{proof}
	
	\begin{lemma}\label{lemma:lemma4_OR_lower_bound_eigenvalue_approximation}
		Consider a grid such that $h_n\in(0,\delta)$ with $\delta$ small enough, then there exists a constant $C>0$ such that for any $1\leq k\leq h_n^{-1/4}$,
		\begin{equation}
		\lambda^n_k-\lambda_k\geq -Ch_n^{1/4},
		\end{equation}
		where $C$ is independent of $k$ and $n$.
	\end{lemma}
	
	\begin{proof}
		It should first be noted that as $\mathcal{G}$ is a self-adjoint operator, the min-max principle holds (see \cite{eschwe2004}, Theorem 2.1).
		In particular,
		\begin{equation}
		\lambda_k=\underset{L\subset\mathcal{D},\dim L=k}{\min}\ \underset{\psi\in L,\psi\neq 0}{\max}\ \frac{\left(\psi,-\mathcal{G}\psi\right)}{\left(\psi,\psi\right)},
		\end{equation}
		where $L$ is a linear subspace of the domain of $\mathcal{G}$.
		
		For $i=1,\ldots,k$ define $\psi_i:\mathbb{S}\to\mathbb{R}$ as a linear interpolation of the approximate eigenfunction $\varphi_i^n$ over the interval $\mathbb{S}$, which is given by
		\begin{equation}
		\psi_i\left(x\right)=\varphi_i^n\left(y^-\right)+\nabla^-\varphi_i^n\left(y\right)\left(x-y^-\right),
		\end{equation}
		for $x\in[y^-,y]$ and $y\in\mathbb{S}^+_n$.
		Then $\{\psi_1,\ldots,\psi_k\}$ form a $k$-dimensional linear space.
		Furthermore, set $\psi_a(x)=\sum_{i=1}^ka_i\psi_i(x)$ such that the $a_i$ are normalized, i.e. $\sum_{i=1}^ka_i^2=1$.
		Using the min-max principle and integration by parts, we obtain
		\begin{align}
		\lambda_k&\leq\underset{a_1,\ldots,a_k:\sum_{i=1}^ka_i^2=1}{\max}\frac{\left(\psi_a,-\mathcal{G}\psi_a\right)}{\left(\psi_a,\psi_a\right)} \nonumber \\ &\quad=\underset{\sum_{i=1}^ka_i^2=1}{\max}\frac{\int_l^r\frac{\psi_a^\prime\left(x\right)^2}{s\left(x\right)}dx+\int_l^rk\left(x\right)\psi_a\left(x\right)^2M\left(dx\right)}{\int_l^r\psi_a\left(x\right)^2M\left(dx\right)}. \label{eq:lemma4_li2017a_eq1}
		\end{align}
		
		We will now estimate the different terms appearing in this equation.
		First, note that as $1-\beta=0$ for Scheme 1 and $\vert 1-\beta\vert\leq C\delta^+x_0$ for Scheme 2, there holds
		\begin{align}
		\left\vert 1-\beta\right\vert\left\vert\psi_a\left(x_0\right)\nabla^+\psi_a\left(x_0\right)\right\vert&\leq C_1\psi_a\left(x_0\right)^2\sqrt{\delta^+x_0}+C_1\left(\nabla^+\psi_a\left(x_0\right)\right)^2\left(\delta^+x_0\right)^{3/2} \nonumber \\
		&\leq C_2\sqrt{h_n}+C_2\sqrt{\delta}\sum_{x\in\mathbb{S}_n^-}\left(\nabla^+\psi_a\left(x\right)\right)^2\delta^+x,\label{eq:proof_lemma_a3_eq2}
		\end{align}
		with constants $C_1,C_2>0$ independent of $a$ and $n$, because
		\begin{align}
		\psi_a\left(x_0\right)^2\leq\sum_{x\in\mathbb{S}_n^-}\psi_a\left(x_0\right)\frac{M_n\left(x\right)}{M_n\left(x\right)}\leq C_3\sum_{i_1=1}^k\sum_{i_2=1}^ka_{i_1}a_{i_2}\sum_{x\in\mathbb{S}_n^-}\varphi_{i_1}^n\left(x\right)\varphi_{i_2}^n\left(x\right)M_n\left(x\right)=C_3,
		\end{align}
		as $\sum_{x\in\mathbb{S}_n^-}\varphi_{i_1}^n(x)\varphi_{i_2}^n(x)M_n(x)=(\varphi_{i_1}^n,\varphi_{i_2}^n)_n=\delta_{i_1,i_2}$ and $C_3>0$ is independent of $a$ and $n$.
		Using this result and the fact that $\psi_a$ is a piecewise linear function, we obtain
		\begin{align}
		&\int_l^r\frac{1}{s\left(x\right)}\psi^\prime_a\left(x\right)^2dx-\sum_{x\in\mathbb{S}_n^-}\frac{1}{s_n\left(x\right)}\left(\nabla^+\psi_a\left(x\right)\right)^2\delta^+x-\frac{1-\beta}{s_n\left(x_0\right)}\psi_a\left(x_0\right)\nabla^+\psi_a\left(x_0\right) \\
		&\qquad\leq\sum_{x\in\mathbb{S}_n^-}\left(\nabla^+\psi_a\left(x\right)\right)^2\int_x^{x^+}\left\vert\frac{1}{s\left(y\right)}-\frac{1}{s_n\left(x\right)}\right\vert dy-\frac{1-\beta}{s_n\left(x_0\right)}\psi_a\left(x_0\right)\nabla^+\psi_a\left(x_0\right) \\
		&\qquad\leq C_4h_n\sum_{x\in\mathbb{S}_n^-}\left(\nabla^+\psi_a\left(x\right)\right)^2\delta^+x+C_5\left(\sqrt{h_n}+\sqrt{\delta}\sum_{x\in\mathbb{S}_n^-}\left(\nabla^+\psi_a\left(x\right)\right)^2\delta^+x\right) \\
		&\qquad\leq C_6h_n\sum_{x\in\mathbb{S}_n^-}\left(\nabla^+\psi_a\left(x\right)\right)^2\delta^+x+C_5\sqrt{h_n},\label{eq:proof_lemma_a3_eq1}
		\end{align}
		where $C_4,C_5,C_6>0$ are independent of $a$ and $n$.
		The term appearing in \eqref{eq:proof_lemma_a3_eq1} can be handled in the following way.
		\begin{align}
		&\sum_{x\in\mathbb{S}_n^-}\left(\nabla^+\psi_a\left(x\right)\right)^2\delta^+x \nonumber \\
		&\quad\leq -C_7\frac{\beta}{s_n\left(x_0\right)}\psi_a\left(x_0\right)\nabla^+\psi_a\left(x_0\right)+C_7\frac{\beta}{s_n\left(x_0\right)}\psi_a\left(x_0\right)\nabla^+\psi_a\left(x_0\right) \\
		&\quad\qquad+C_7\sum_{x\in\mathbb{S}_n^-}\frac{1}{s_n\left(x\right)}\left(\nabla^+\psi_a\left(x\right)\right)^2\delta^+x+C_7\sum_{x\in\mathbb{S}_n^-}k\left(x\right)M_n\left(x\right)\psi_a\left(x\right)^2 \\
		&\quad=-C_7\beta\rho M_n\left(x_0\right)\psi_a\left(x_0\right)\nabla^+\psi_a\left(x_0\right)+\frac{C_7\beta}{s_n\left(x_0\right)}\psi_a\left(x_0\right)\nabla^+\psi_a\left(x_0\right) \\
		&\quad\qquad-C_7\sum_{x\in\mathbb{S}_n^\circ}\psi_a\left(x\right)\delta^-x\nabla^-\left(\frac{1}{s_n\left(x\right)}\nabla^+\psi_a\left(x\right)\right)-\frac{C_7}{s_n\left(x_0\right)}\psi_a\left(x_0\right)\nabla^+\psi_a\left(x_0\right) \\
		&\quad\qquad+C_7\sum_{x\in\mathbb{S}_n^-}k\left(x\right)M_n\left(x\right)\psi_a\left(x\right)^2 \\
		&\quad=-C_7\sum_{x\in\mathbb{S}_n^-}\psi_a\left(x\right)M_n\left(x\right)G_n\psi_a\left(x\right)+\frac{C_7\left(\beta-1\right)}{s_n\left(x_0\right)}\psi_a\left(x_0\right)\nabla^+\psi_a\left(x_0\right) \\
		&\quad\leq C_7\sum_{i_1=1}^k\sum_{i_2=1}^ka_{i_1}a_{i_2}\lambda_{i_2}^n\sum_{x\in\mathbb{S}_n^-}\varphi_{i_1}^n\left(x\right)\varphi_{i_2}^n\left(x\right)M_n\left(x\right)+C_8\left(\sqrt{h_n}+\sqrt{\delta}\sum_{x\in\mathbb{S}_n^-}\left(\nabla^+\psi_a\left(x\right)\right)^2\delta^+x\right) \\
		&\quad\leq C_9\lambda_k^n+C_7\sqrt{\delta}\sum_{x\in\mathbb{S}_n^-}\left(\nabla^+\psi_a\left(x\right)\right)^2\delta^+x, \label{eq:proof_lemma4_OR_eq1}
		\end{align}
		where the last inequality follows from the fact that $0\leq \lambda_1^n<\lambda_2^n<\dots<\lambda_k^n$.
		The constants $C_7,C_8,C_9>0$ are independent of $a$ and $n$.
		We can now choose $\delta$ small enough, such that $1-C_7\sqrt{\delta}>0$, then
		\begin{align}
		\sum_{x\in\mathbb{S}_n^-}\left(\nabla^+\psi_a\left(x\right)\right)^2\delta^+x\leq\frac{C_9}{1-C_7\sqrt{\delta}}\lambda_k^n\leq C_{10}\lambda_k^n.
		\end{align}
		Combining \eqref{eq:proof_lemma_a3_eq1} and the previous results together, yields
		\begin{align}
		\int_l^r\frac{1}{s\left(x\right)}\psi^\prime_a\left(x\right)dx&\leq\frac{1-\beta}{s_n\left(x_0\right)}\psi_a\left(x_0\right)\nabla^+\psi_a\left(x_0\right)+\sum_{x\in\mathbb{S}_n^-}\frac{1}{s_n\left(x\right)}\left(\nabla^+\psi_a\left(x\right)\right)^2\delta^+x \\
		&\qquad+C_{11}\lambda_k^nh_n+C_5h_n^{1/2}.
		\end{align}
		The second term in \eqref{eq:lemma4_li2017a_eq1} can be handled in the following way.
		\begin{align*}
		&\left\vert\int_l^rk\left(x\right)\psi_a\left(x\right)^2M\left(dx\right)-\sum_{x\in\mathbb{S}_n^-}k\left(x\right)\psi_a\left(x\right)^2M_n\left(x\right)\right\vert \\
		&\qquad\leq\left\vert k\left(x_0\right)\psi_a\left(x_0\right)^2M\left(x_0\right)-k\left(x_0\right)\psi_a\left(x_0\right)^2M_n\left(x_0\right)\right\vert \\
		&\qquad\qquad+\left\vert\int_l^rk\left(x\right)\psi_a\left(x\right)^2m\left(x\right)dx-\sum_{x\in\mathbb{S}_n^\circ}k\left(x\right)\psi_a\left(x\right)^2m_n\left(x\right)\delta x\right\vert \\
		&\qquad\leq C_{12}h_n\psi_a\left(x_0\right)^2+\frac{1}{2}\sum_{x\in\mathbb{S}_n^-}\int_x^{x^+}\left\vert k\left(y\right)\psi_a\left(y\right)^2m\left(y\right)-k\left(x\right)\psi_a\left(x\right)^2m_n\left(x\right)\right\vert dy \\
		&\qquad\qquad+\frac{1}{2}\sum_{x\in\mathbb{S}_n^-}\int_x^{x^+}\left\vert k\left(y\right)\psi_a\left(y\right)^2m\left(y\right)-k\left(x^+\right)\psi_a\left(x^+\right)^2m_n\left(x^+\right)\right\vert dy \\
		&\qquad\leq C_{13}\left(\sqrt{\lambda_k^n}h_n+\lambda_k^nh_n^2\right),
		\end{align*}
		where the last inequality follows in the same way as in the proof of Lemma 4 in \cite{zhang2019a}.
		Using this result to bound the numerator, one obtains
		\begin{align*}
		\qquad&\int_l^r\frac{1}{s\left(x\right)}\psi_a^\prime\left(x\right)^2dx+\int_l^rk\left(x\right)\psi_a\left(x\right)^2M\left(dx\right) \\
		&\qquad\leq -\sum_{x\in\mathbb{S}_n^-}\psi_a\left(x\right)M_n\left(x\right)G_n\psi_a\left(x\right)+C_{14}\left(\left(\sqrt{\lambda_k^n}+\lambda_k^n\right)h_n+\lambda_k^nh_n^2+h_n^{1/2}\right) \\
		&\qquad\leq\lambda_k^n+C_{14}\left(\left(\sqrt{\lambda_k^n}+\lambda_k^n\right)h_n+\lambda_k^nh_n^2+h_n^{1/2}\right),
		\end{align*}
		for some constant $C_{14}>0$ independent of $a$, $k$ and $n$.
		The denominator can be estimated similarly as before by setting $k(x)=1$.
		\begin{equation*}
		\left\vert\int_l^r\psi_a\left(x\right)^2M\left(dx\right)-1\right\vert\leq C_{15}\left(\sqrt{\lambda_k^n}h_n+\lambda_k^nh_n^2\right),
		\end{equation*}
		where $C_{15}>0$ is a constant independent of $a$, $k$ and $n$.
		As all of the constants are independent of $a$, there follows
		\begin{equation}
		\lambda_k\leq \frac{\lambda_k^n+C_{14}\left(\left(\sqrt{\lambda_k^n}+\lambda_k^n\right)h_n+\lambda_k^nh_n^2+h_n^{1/2}\right)}{1-C_{15}\left(\sqrt{\lambda_k^n}h_n+\lambda_k^nh_n^2\right)}.
		\end{equation}
		Using Lemma \ref{lemma:square_k_lower_upper_bounds}, i.e. $\lambda_k^n\leq C_{16}k^2\leq C_{16}h_n^{-1/2}$ for some $C_{16}>0$ independent of $k$ and $n$, there holds
		\begin{align*}
		\lambda_k-\lambda_k^n&\leq\frac{C_{14}\left(\left(\sqrt{\lambda_k^n}+\lambda_k^n\right)h_n+\lambda_k^nh_n^2+h_n^{1/2}\right)+C_{15}\lambda_k^n\left(\sqrt{\lambda_k^n}h_n+\lambda_k^nh_n^2\right)}{1-C_{15}\left(\sqrt{\lambda_k^n}h_n+\lambda_k^nh_n^2\right)} \\
		&\leq\frac{C_{17}\left(h_n^{3/4}+h_n^{1/2}+h_n^{3/2}+h_n^{1/4}+h_n^{1/2}\right)}{1-C_{15}\left(\delta^{3/4}+\delta^{3/2}\right)}\leq C_{18}h_n^{1/4},
		\end{align*}
		for constants $C_{17},C_{18}>0$ independent of $k$ and $n$, as long as $\delta$ is small enough, such that $1-C_{15}(\delta^{3/4}+\delta^{3/2})>0$.
	\end{proof}
	
	\begin{lemma}
		If $h_n\in(0,\delta)$ for $\delta$ small enough, there exists a constant $C>0$, such that for any $1\leq k\leq n$,
		\begin{equation}
		\left\Vert\varphi_k^n\right\Vert_{n,\infty}\leq Ck,
		\end{equation}
		where $C$ is independent of $k$ and $n$.
	\end{lemma}
	
	\begin{proof}
		Note that for every $y\in\mathbb{S}^-_n$
		\begin{equation}
		\varphi_k^n\left(y\right)=\sum_{y\leq x<x_{n+1}}\varphi_k^n\left(x\right)-\varphi_k^n\left(x^+\right)=-\sum_{y\leq x<x_{n+1}}\nabla^+\varphi_k^n\left(x\right)\delta^+x,
		\end{equation}
		as $\varphi_k^n(x_{n+1})=0$.
		Then
		\begin{align}
		\left\vert\varphi_k^n\left(y\right)\right\vert&=\left\vert-\sum_{y\leq x<x_{n+1}}\nabla^+\varphi_k^n\left(x\right)\delta^+x\right\vert\leq\sum_{x\in\mathbb{S}_n^-}\left\vert\nabla^+\varphi_k^n\left(x\right)\right\vert\delta^+x \\
		&\leq\sqrt{\sum_{x\in\mathbb{S}_n^-}\left(\nabla^+\varphi_k^n\left(x\right)\right)^2\delta^+x\sum_{x\in\mathbb{S}_n^-}\delta^+x}\leq C_1\sqrt{\lambda_k^n},
		\end{align}
		for a constant $C_1>0$ independent of $k$, $n$ and $y$ because of the same steps as shown in \eqref{eq:proof_lemma4_OR_eq1} when $\psi_a$ is replaced by $\varphi_k^n$.
		Furthermore, by Lemma \ref{lemma:square_k_lower_upper_bounds}, i.e. $\lambda_k^n\leq C_2k^2$, there follows
		\begin{equation}
		\left\Vert\varphi_k^n\right\Vert_{n,\infty}\leq C_1\sqrt{\lambda_k^n}\leq C_3k,\label{eq:proof_lemma5_OR_eq4}
		\end{equation}
		for constants $C_2,C_3>0$ independent of $k$ and $n$.
	\end{proof}
	
	\bigskip
	\noindent\emph{Proof of Proposition \ref{prop:li_or_paper_prop2}}:
	For $x\in\mathbb{S}_n^\circ$, the proof of Proposition 2 in \cite{zhang2019} shows that
	\begin{equation*}
	\left\vert G_n\varphi_k\left(x\right)-\mathcal{G}\varphi_k\left(x\right)\right\vert\leq C_1k^4h_n^2.
	\end{equation*}
	For $\varphi_k$ the Taylor expansion at $x_1$ is given by
	\begin{equation}
	\varphi_k\left(x_1\right)=\varphi_k\left(x_0\right)+\varphi_k^\prime\left(x_0\right)\delta^+x_0+\frac{1}{2}\varphi_k^{\prime\prime}\left(x_0\right)\left(\delta^+x_0\right)^2+\frac{1}{6}\varphi_k^{\prime\prime\prime}\left(x_0\right)\left(\delta^+x_0\right)^3+\frac{1}{24}\varphi_k^{(4)}\left(\eta\right)\left(\delta^+x_0\right)^4,
	\end{equation}
	for some $\eta\in[x_0,x_1]$.
	Subtracting $\varphi_k(x_0)$ and dividing by $\delta^+x_0$ yields
	\begin{equation}
	\nabla^+\varphi_k\left(x_0\right)=\varphi_k^\prime\left(x_0\right)+\frac{1}{2}\varphi_k^{\prime\prime}\left(x_0\right)\delta^+x_0+\frac{1}{6}\varphi_k^{\prime\prime\prime}\left(x_0\right)\left(\delta^+x_0\right)^2+\frac{1}{24}\varphi_k^{(4)}\left(\eta\right)\left(\delta^+x_0\right)^3.
	\end{equation}
	Further, note that as $\varphi_k$ is in the domain of the generator, it also satisfies
	\begin{equation}
	\frac{1}{2}\varphi_k^{\prime\prime}\left(x_0\right)=\frac{\rho-\mu\left(x_0\right)}{\sigma^2\left(x_0\right)}\varphi_k^\prime\left(x_0\right).
	\end{equation}
	Using these results, one obtains
	\begin{align*}
	&\left\vert G_n\varphi_k\left(x_0\right)-\mathcal{G}\varphi_k\left(x_0\right)\right\vert=\left\vert\rho\beta\nabla^+\varphi_k\left(x_0\right)-\rho\varphi_k^\prime\left(x_0\right)\right\vert \\
	&\qquad=\rho\left\vert\left(\beta-1\right)\varphi^\prime_k\left(x_0\right)+\frac{\beta}{2}g^{\prime\prime}\left(x_0\right)\delta^+x_0+\frac{\beta}{6}\varphi_k^{\prime\prime\prime}\left(x_0\right)\left(\delta^+x_0\right)^2+\frac{\beta}{24}\varphi_k^{(4)}\left(\eta\right)\left(\delta^+x_0\right)^3\right\vert \\
	&\qquad=\rho\left\vert\left(\beta-1\right)\varphi^\prime_k\left(x_0\right)+\beta\frac{\rho-\mu\left(x_0\right)}{\sigma^2\left(x_0\right)}\varphi_k^\prime\left(x_0\right)\delta^+x_0+\frac{\beta}{6}\varphi_k^{\prime\prime\prime}\left(x_0\right)\left(\delta^+x_0\right)^2+\frac{\beta}{24}\varphi_k^{(4)}\left(\eta\right)\left(\delta^+x_0\right)^3\right\vert \\ \\
	&\qquad\leq\begin{cases} C_2\left\vert\varphi_k^\prime\left(x_0\right)\right\vert\delta^+x_0+O\left(h_n^2\right)\leq C_3\left\Vert\varphi_k^\prime\right\Vert_\infty h_n+O\left(h_n^2\right) &\qquad\textrm{for Scheme 1,} \\ \\ C_4\frac{\left\vert\beta\right\vert}{6}\left\vert\varphi^{\prime\prime\prime}\left(x_0\right)\right\vert\left(\delta^+x_0\right)^2+O\left(h_n^3\right)\leq C_5\left\Vert\varphi^{\prime\prime\prime}_k\right\Vert_\infty h_n^2+O\left(h_n^3\right) &\qquad\textrm{for Scheme 2},\end{cases}
	\end{align*}
	where the boundedness of $\beta$ was used and the particular form of $\beta$ in Scheme 2 allows the terms involving $\varphi_k^\prime(x_0)$ to cancel.
	Here the constants $C_2,C_3,C_4,C_5>0$ are independent of $n$ and $k$.
	
	Applying the result from Proposition \ref{prop:SLeigen} shows that $\Vert\varphi_k^\prime\Vert_\infty\leq C_6k\leq C_6k^3$ and $\Vert\varphi_k^{\prime\prime\prime}\Vert_\infty\leq C_7k^3$.
	Thus summarizing the two cases for $x_0$ and combining it with the result for $x\in\mathbb{S}_n^\circ$ implies
	\begin{equation}
	\left\Vert G_n\varphi_k-\mathcal{G}\varphi_k\right\Vert_{n,\infty}\leq\max\{C_1k^4h_n^2,C_8k^3h_n^\gamma\}\leq C_9k^4h_n^\gamma,\label{eq:proposition_2_li_eq1}
	\end{equation}
	where $C_8,C_9>0$ are independent of $n$ and $k$ and as always $\gamma=1$ for Scheme 1 and $\gamma=2$ for Scheme 2.
	Using the fact that $\mathcal{G}\varphi_k=-\lambda_k\varphi_k$ implies
	\begin{equation}
	\left\vert\mu_k^n-\lambda_k\right\vert\leq C_{10}k^4h_n^\gamma,
	\end{equation}
	for some constant $C_{10}>0$ independent of $k$ and $n$ and $\mu_k^n=\arg\min_{\mu\in\Lambda(G_n)}\vert\mu-\lambda_k\vert$ where $\Lambda(G_n)$ is the set of eigenvalues of $-G_n$.
	The arguments stated in the proof of Proposition 2 in \cite{zhang2019} and Proposition 3.6 in \cite{li2018} still remain valid, hence it can be shown that $\mu_k^n=\lambda_k^n$ and so the claim follows for any $k\leq h_n^{-1/4}$ with sufficiently small $h_n$.\qed
	
	\bigskip
	\begin{lemma}\label{lemma:or_2018_lemma7}
		There holds
		\begin{align}
		&\left\vert\sum_{x\in\mathbb{S}_n^-}f\left(x\right)M_n\left(x\right)-\int_{x_0}^{x_{n+1}}f\left(x\right)M\left(dx\right)\right\vert\leq C\max\{\left\Vert f\right\Vert_\infty,\left\Vert f^{\prime\prime}\right\Vert_\infty\}h_n^\gamma,
		\end{align}
		for some constant $C>0$ independent of $n$ and $f$.
	\end{lemma}
	\begin{proof}
		We can show the following by using Proposition \ref{prop:convergence_speed_measure_density} and the trapezoidal rule
		\begin{align*}
		&\sum_{x\in\mathbb{S}_n^-}f\left(x\right)M_n\left(x\right)-\int_{x_0}^{x_{n+1}}f\left(x\right)M\left(dx\right) \\
		&\qquad=f\left(x_0\right)M_n\left(x_0\right)-f\left(x_0\right)M\left(x_0\right)+\sum_{x\in\mathbb{S}_n^\circ}f\left(x\right)m_n\left(x\right)\delta x-\int_{x_0}^{x_{n+1}}f\left(x\right)m\left(x\right)dx \\
		&\qquad=f\left(x_0\right)\frac{\delta^+x_0}{\sigma^2\left(x_0\right)}\left(M\left(x_0\right)\alpha-1\right)+O\left(h_n^2\right)+\sum_{x\in\mathbb{S}_n^\circ}f\left(x\right)m_n\left(x\right)\delta x-\int_{x_0}^{x_{n+1}}f\left(x\right)m\left(x\right)dx \\
		&\qquad\leq C_1\left\vert M\left(x_0\right)\alpha-1\right\vert\left\Vert f\right\Vert_\infty h_n+C_2\left\Vert f^{\prime\prime}\right\Vert_\infty h_n^2 \\
		&\qquad\leq\begin{cases} C_3\left\Vert f\right\Vert_\infty h_n & \qquad\textrm{for}\ \alpha=\mu\left(x_0\right), \\ C_2\left\Vert f^{\prime\prime}\right\Vert_\infty h_n^2 &\qquad\textrm{for}\ \alpha=\rho\ (\textrm{as}\ M\left(x_0\right)\alpha=\frac{1}{\rho}\times\rho=1), \end{cases} \\
		&\qquad\leq C_4\max\{\left\Vert f\right\Vert_\infty,\left\Vert f^{\prime\prime}\right\Vert_\infty\}h_n^\gamma,
		\end{align*}
		where $C_1,\ldots,C_4>0$ are independent of $n$ and $f$.
	\end{proof}
	
	\begin{corollary}
		For $h_n\in(0,\delta)$ there holds for every $1\leq k\leq h_n^{-1/4}$, the following lower bound
		\begin{equation}
		\lambda_k^n\geq Ck^2,
		\end{equation}
		if $\delta$ is sufficiently small and $C>0$ is a constant independent of $k$ and $n$.
	\end{corollary}
	
	\begin{proof}
		The proof is the same as the proof of Corollary 3.7 in \cite{li2018} by using Proposition \ref{prop:li_or_paper_prop2} and Lemma \ref{lemma:square_k_lower_upper_bounds}.
	\end{proof}
	
	\bigskip
	\noindent\emph{Proof of Proposition \ref{proposition:li_2017_proposition3}}:
	The main arguments in \cite{zhang2019} for the error of eigenfunctions cannot be applied to the problem here, so we use different ideas.
	Define $\psi_k^n(x)=c\varphi_k^n(x)$ with a constant $c$ such that $\nabla^+\psi_k^n(x_n)=\nabla^+\varphi_k(x_n)$.
	Furthermore, let $e_k^n(x)=\psi_k^n(x)-\varphi_k(x)$.
	Then for every $x\in\mathbb{S}_n^\circ$
	\begin{align*}
	G_ne_k^n\left(x\right)&=G_n\left(\psi_k^n\left(x\right)-\varphi_k\left(x\right)\right) \\
	&=G_n\psi_k^n\left(x\right)-\mathcal{G}\varphi_k\left(x\right)+\left(\mathcal{G}-G_n\right)\varphi_k\left(x\right) \\
	&=\lambda_k^n\psi_k^n\left(x\right)-\lambda_k\varphi_k\left(x\right)+\left(\mathcal{G}-G_n\right)\varphi_k\left(x\right) \\
	&=\lambda_k^ne_k^n\left(x\right)+\left(\lambda_k^n-\lambda_k\right)\varphi_k\left(x\right)+\left(\mathcal{G}-G_n\right)\varphi_k\left(x\right).
	\end{align*}
	We multiply both sides with $m_n(x)\delta x$ and sum the terms over $x$ from $y\in\mathbb{S}_n^\circ$ to $x_n$
	\begin{align*}
	&\sum_{y\leq x<x_{n+1}}m_n\left(x\right)\delta x G_ne_k^n\left(x\right) \\
	&\qquad=\sum_{y\leq x<x_{n+1}}\delta^-x\nabla^-\left(\frac{1}{s_n\left(x\right)}\nabla^+e_k^n\left(x\right)\right)+\sum_{y\leq x<x_{n+1}}k\left(x\right)e_k^n\left(x\right)m_n\left(x\right)\delta x \\
	&\qquad=\frac{1}{s_n\left(x_{n+1}\right)}\nabla^+e_k^n\left(x_{n+1}\right)-\frac{1}{s_n\left(y^-\right)}\nabla^+e_k^n\left(y^-\right) \\
	&\qquad\qquad+\sum_{y\leq x<x_{n+1}}k\left(x\right)e_k^n\left(x\right)m_n\left(x\right)\delta x \\
	&\qquad=-\frac{1}{s_n\left(y^-\right)}\nabla^+e_k^n\left(y^-\right)+\sum_{y\leq x<x_{n+1}}k\left(x\right)e_k^n\left(x\right)m_n\left(x\right)\delta x \\
	&\qquad=\lambda_k^n\sum_{y\leq x<x_{n+1}}e_k^n\left(x\right)m_n\left(x\right)\delta x+\left(\lambda_k^n-\lambda_k\right)\sum_{y\leq x<x_{n+1}}\varphi_k\left(x\right)m_n\left(x\right)\delta x \\
	&\qquad\qquad+\sum_{y\leq x<x_{n+1}}m_n\left(x\right)\delta x\left(\mathcal{G}-G_n\right)\varphi_k\left(x\right),
	\end{align*}
	because $\psi_k^n(x_n)=\varphi_k(x_n)$, $\psi_k^n(x_{n+1})=c\varphi_k^n(x_{n+1})=\varphi_k(x_{n+1})=0$ and $\nabla^+e_k^n(x_n)=0$.
	Then
	\begin{align}
	\frac{1}{s_n\left(y^-\right)\delta^+y^-}e_k^n\left(y^-\right)&=\frac{1}{s_n\left(y^-\right)\delta^+y^-}e_k^n\left(y\right)+\sum_{y\leq x<x_{n+1}}\left(\lambda_k^n-k\left(x\right)\right)e_k^n\left(x\right)m_n\left(x\right)\delta x \\
	&\qquad+\left(\lambda_k^n-\lambda_k\right)\sum_{y\leq x<x_{n+1}}\varphi_k\left(x\right)m_n\left(x\right)\delta x \\
	&\qquad+\sum_{y\leq x<x_{n+1}}m_n\left(x\right)\delta x\left(\mathcal{G}-G_n\right)\varphi_k\left(x\right).\label{eq:prop_43_eq_1}
	\end{align}
	Multiplying both sides with $s_n(y^-)\delta^+y^-$ and taking the absolute value, results in
	\begin{flalign*}
	\qquad\left\vert e_k^n\left(y^-\right)\right\vert&\leq\left\vert e_k^n\left(y\right)\right\vert+\left\vert s_n\left(y^-\right)\delta^+y^-\sum_{y\leq x<x_{n+1}}\left(\lambda_k^n-k\left(x\right)\right)e_k^n\left(x\right)m_n\left(x\right)\delta x\right\vert \\
	&\qquad+\left\vert\left(\lambda_k^n-\lambda_k\right)s_n\left(y^-\right)\delta^+y^-\sum_{y\leq x<x_{n+1}}\varphi_k\left(x\right)m_n\left(x\right)\delta x\right\vert \\
	&\qquad+\left\vert s_n\left(y^-\right)\delta^+y^-\sum_{y\leq x<x_{n+1}}m_n\left(x\right)\delta x\left(\mathcal{G}-G_n\right)\varphi_k\left(x\right)\right\vert \\
	&\leq\left\vert e_k^n\left(y\right)\right\vert+C_1k^2h_n\sum_{y\leq x<x_{n+1}}\left\vert e_k^n\left(x\right)\right\vert\delta x
	&& \mathmakebox[1pt][r]{\left(\textrm{See (a)}\right)} \\
	&\qquad+C_2k^4h_n^{\gamma+1}
	&& \mathmakebox[1pt][r]{\left(\textrm{See (b)}\right)} \\
	&\qquad+C_3k^4h_n^{\gamma+1}
	&& \mathmakebox[1pt][r]{\left(\textrm{See (c)}\right)} \\
	&\leq\left\vert e_k^n\left(y\right)\right\vert+C_1k^2h_n\sum_{y\leq x<x_{n+1}}\left\vert e_k^n\left(x\right)\right\vert\delta x+C_4k^4h_n^{\gamma+1} \\
	&\leq C_5k^2h_n\sum_{y\leq x<x_{n+1}}\left\vert e_k^n\left(x\right)\right\vert\delta x+C_4k^4h_n^{\gamma+1},
	&& \mathmakebox[1pt][r]{\left(\textrm{See (d)}\right)} \\
	\end{flalign*}
	where
	\begin{enumerate}[(a)]
		\item holds because $\lambda_k^n\leq Ck^2$ by \eqref{eq:upper_lower_bound_approximate_eigenvalues}, $m_n\left(x\right)\leq C$, $s_n\left(y^-\right)\leq C$, $\delta^+y^-\leq h_n$.
		\item holds because $\left\vert\lambda_k^n-\lambda_k\right\vert\leq Ck^4h_n^\gamma$ by \eqref{eq:or_paper_2018_prop2} and
		\begin{equation*}
		\left\vert\sum_{y\leq x<x_{n+1}}\varphi_k\left(x\right)m_n\left(x\right)\delta x\right\vert\leq Ch_n\sum_{y\leq x<x_{n+1}}\left\vert\varphi_k\left(x\right)\right\vert\leq Ch_nn\left\Vert\varphi_k\right\Vert_\infty\leq C.
		\end{equation*}
		\item holds because because
		\begin{align*}
		&\left\vert\sum_{y\leq x<x_{n+1}}m_n\left(x\right)\delta x\left(\mathcal{G}-G_n\right)\varphi_k\left(x\right)\right\vert\leq\left\Vert\mathcal{G}\varphi_k-G_n\varphi_k\right\Vert_{n,\infty}\sum_{y\leq x<x_{n+1}}m_n\left(x\right)\delta x \\
		&\qquad\leq C\left\Vert\mathcal{G}\varphi_k-G_n\varphi_k\right\Vert_{n,\infty}nh_n\leq Ck^4h_n^\gamma,
		\end{align*}
		by \eqref{eq:proposition_2_li_eq1}.
		\item holds because one can choose $C_1$ large enough, such that $C_1k^2h_n\delta x\geq 1$ and the first term can be put into the sum.
	\end{enumerate}
	The constants $C_1\ldots,C_5>0$ are independent of $k$, $n$ and $x$, $y$.	
	Using the discrete version of Gronwall's inequality and noting that $k\leq h_n^{-1/4}$, there holds
	\begin{align*}
	\left\vert e_k^n\left(y^-\right)\right\vert&\leq\sum_{y\leq x<x_{n+1}}C_4k^4h_n^{\gamma+1}\exp\left(C_5k^2h_n\sum_{y\leq x<x_{n+1}}\delta x\right) \\
	&\leq C_4k^4h_n^{\gamma+1}n\exp\left(C_5\left(r-l\right)\delta^{1/2}\right)\leq C_6k^4h_n^\gamma.
	\end{align*}
	Note that this inequality also holds for $y=x_{n+1}$ because of the choice of $c$ and $\psi_k^n$, such that $e_k^n(x_{n+1})=e_k^n(x_{n})$.
	Furthermore, $C_6>0$ is independent of $k$, $n$ and $x$, $y$.
	
	Then there holds
	\begin{align*}
	\left\Vert\varphi_k^n-\varphi_k\right\Vert_{n,\infty}=\left\Vert\frac{\psi_k^n}{\left\Vert\psi_k^n\right\Vert_{n,2}}-\varphi_k\right\Vert_{n,\infty}&\leq\frac{1}{\left\Vert\psi_k^n\right\Vert_{n,2}}\left\Vert\psi_k^n-\varphi_k\right\Vert_{n,\infty}+\left\vert\frac{1}{\left\Vert\psi_k^n\right\Vert_{n,2}}-1\right\vert\left\Vert\varphi_k\right\Vert_\infty \\
	&\leq\frac{C_6k^4h_n^\gamma}{\left\Vert\psi_k^n\right\Vert_{n,2}}+C_7\left\vert\frac{1}{\left\Vert\psi_k^n\right\Vert_{n,2}}-1\right\vert
	\end{align*}
	and furthermore,
	\begin{flalign*}
	\qquad\left\vert 1-\left\Vert\psi_k^n\right\Vert_{n,2}^2\right\vert&=\left\vert\int_{x_0}^{x_{n+1}}\varphi_k^2\left(x\right)M\left(dx\right)-\sum_{x\in\mathbb{S}_n^-}\psi_k^n\left(x\right)^2M_n\left(x\right)\right\vert \\
	&\leq C_8\left\Vert\varphi_k^{\prime\prime}\right\Vert_\infty h_n^\gamma+\sum_{x\in\mathbb{S}_n^-}\left\vert\varphi_k\left(x\right)^2-\psi_k^n\left(x\right)^2\right\vert M_n\left(x\right)\leq C_{9}k^4h_n^\gamma,
	\end{flalign*}
	by making use of Lemma \ref{lemma:or_2018_lemma7} and the fact that
	\begin{align*}
	&\sum_{x\in\mathbb{S}_n^-}\left\vert\varphi_k\left(x\right)^2-\psi_k^n\left(x\right)^2\right\vert M_n\left(x\right)=\sum_{x\in\mathbb{S}_n^-}\left\vert e_k^n\left(x\right)\right\vert\left\vert\varphi_k\left(x\right)^2-\psi_k^n\left(x\right)^2\right\vert M_n\left(x\right) \\
	&\qquad\leq\sum_{x\in\mathbb{S}_n^-}\left\vert e_k^n\left(x\right)\right\vert^2M_n\left(x\right)+2\sum_{x\in\mathbb{S}_n^-}\left\vert e_k^n\left(x\right)\right\vert\left\vert\varphi_k\left(x\right)\right\vert M_n\left(x\right)\leq C_{10}k^4h_n^\gamma.
	\end{align*}
	Putting this result back into the above equation yields
	\begin{equation*}
	\left\Vert\psi_k^n\right\Vert_{n,2}\geq\sqrt{1-C_{9}k^4h_n^\gamma}\geq\sqrt{1-C_{9}\delta^{\gamma-4/5}},
	\end{equation*}
	as $k\leq h_n^{-1/5}$.
	Collecting all results shows
	\begin{equation}
	\left\Vert\varphi_k^n-\varphi_k\right\Vert_{n,\infty}\leq\frac{C_6k^4h_n^\gamma}{\sqrt{1-C_{9}\delta^{\gamma-4/5}}}+C_7\left\vert\frac{1}{\sqrt{1-C_{9}\delta^{\gamma-4/5}}}-1\right\vert\leq C_{10}k^4h_n^\gamma \label{eq:li_2018_proposition3_conv_eigenfunctions}
	\end{equation}
	and thus the claim is shown.\qed
	
	\bigskip
	\noindent\emph{Proof of Theorem \ref{theorem:convergence_transition_probability}}:
	We first take a more detailed look at the approximation of the transition density $p_n$ in the interior of the state space.
	
	Comparing the eigenfunction expansions of $p_n$ and $p$ shows that for $y\in\mathbb{S}_n^\circ$,
	\begin{align*}
	&\left\vert\frac{1}{m_n\left(y\right)}p_n\left(t,x,y\right)-\frac{1}{m\left(y\right)}p\left(t,x,y\right)\right\vert \\
	&\quad=\left\vert\sum_{k=1}^n\exp\left(-\lambda_k^nt\right)\varphi_k^n\left(x\right)\varphi_k^n\left(y\right)-\sum_{k=1}^\infty\exp\left(-\lambda_kt\right)\varphi_k\left(x\right)\varphi_k\left(y\right)\right\vert \\
	&\quad\leq\sum_{1\leq k\leq h_n^{-1/5}}\exp\left(-\lambda_k^nt\right)\left\Vert\varphi_k^n-\varphi_k\right\Vert_{n,\infty}\left\Vert\varphi_k^n\right\Vert_{n,\infty} +\sum_{1\leq k\leq h_n^{-1/5}}\exp\left(-\lambda_k^nt\right)\left\Vert\varphi_k\right\Vert_\infty\left\Vert\varphi_k^n-\varphi_k\right\Vert_{n,\infty} \\
	&\quad\qquad+\sum_{1\leq k\leq h_n^{-1/5}}\left\vert\exp\left(-\lambda_k^nt\right)-\exp\left(-\lambda_kt\right)\right\vert\left\Vert\varphi_k\right\Vert_\infty\left\Vert\varphi\right\Vert_\infty \\
	&\quad\qquad+\sum_{h_n^{-1/5}<k\leq n}\exp\left(-\lambda_k^nt\right)\left\Vert\varphi_k^n\right\Vert_{n,\infty}\left\Vert\varphi_k^n\right\Vert_{n,\infty}+\sum_{k>h_n^{-1/5}}\exp\left(-\lambda_kt\right)\left\Vert\varphi_k\right\Vert_\infty\left\Vert\varphi_k\right\Vert_\infty \\
	&\quad\leq  C_1h_n^\gamma\sum_{1\leq k\leq h_n^{-1/5}}\exp\left(-C_2k^2t\right)\left(k^5+k^4\right)+C_3\sum_{1\leq k\leq h_n^{-1/5}}\left\vert\exp\left(-\lambda_k^nt\right)-\exp\left(-\lambda_kt\right)\right\vert \\
	&\quad\qquad+C_4\sum_{h_n^{-1/5}<k\leq n}\exp\left(-\lambda_k^n t\right)k^2+C_5\sum_{k>h_n^{-1/5}}\exp\left(-C_6k^2t\right),
	\end{align*}
	for positive constants independent of $k$, $n$ and also $x$, $y$.
	The last three terms will be studied further.
	\begin{align*}
	C_3\sum_{1\leq k\leq h_n^{-1/5}}\left\vert\exp\left(-\lambda_k^nt\right)-\exp\left(-\lambda_kt\right)\right\vert&\leq C_3C_7\sum_{1\leq k\leq h_n^{-1/5}}\exp\left(-C_8k^2t\right)tk^4h_n^\gamma \\
	&\leq C_9\left(t\right)h_n^\gamma,
	\end{align*}
	where $C_9(t)=C_3C_7\sum_{k=1}^\infty \exp\left(-C_8k^2t\right)tk^4<\infty$.
	Furthermore,
	\begin{flalign*}
	C_4\sum_{h_n^{-1/5}<k\leq n}\exp\left(-\lambda_k^n t\right)k^2&\leq C_4n\exp\left(-C_{2}h_n^{-2/5}t\right)n^2 \\
	&\leq C_{10}h_n^2\exp\left(-C_{2}h_n^{-2/5}t\right)n^5 \\
	&\leq C_{11}\left(t\right)h_n^2,
	\end{flalign*}
	where $C_{11}(t)=C_{10}\exp(-C_{2}h_n^{-2/5}t)n^5<\infty$.
	Lastly,
	\begin{equation*}
	C_5\sum_{k>h_n^{-1/5}}\exp\left(-C_6k^2t\right)\leq C_{12}h_n^2\sum_{k>h_n^{-1/5}}\exp\left(-C_7k^2t\right)k^{12}\leq C_{13}\left(t\right)h_n^2.
	\end{equation*}
	Summarizing the above results yields
	\begin{equation*}
	\left\vert\frac{1}{m_n\left(y\right)}p_n\left(t,x,y\right)-\frac{1}{m\left(y\right)}p\left(t,x,y\right)\right\vert\leq C_{14}\left(t\right)h_n^\gamma+C_{15}\left(t\right)h_n^2\leq C_th_n^\gamma,
	\end{equation*}
	for some constant $C_t>0$ depending only on $t$.
	Finally, by using the difference between $m_n(x)$ and $m(x)$, derived in \eqref{eq:convergence_speed_density_interior}, there follows
	\begin{align*}
	p_n\left(t,x,y\right)-p\left(t,x,y\right)&=p\left(t,x,y\right)\left(\frac{m_n\left(y\right)}{m\left(y\right)}-1\right)+C_th_n^\gamma \\
	&=p\left(t,x,y\right)\frac{\mu\left(y\right)}{\sigma^2\left(y\right)}\left(\delta^+y-\delta^-y\right)+C_th_n^\gamma.
	\end{align*}
	It can also be shown using the same steps as above that
	\begin{equation*}
	\left\vert\frac{1}{M_n\left(y\right)}P_n\left(t,x,y\right)-\frac{1}{M\left(y\right)}P\left(t,x,y\right)\right\vert\leq C_th_n^\gamma
	\end{equation*}
	and by using \eqref{eq:convergence_speed_measure_x0} there holds
	\begin{align*}
	P_n\left(t,x,x_0\right)-P\left(t,x,x_0\right)&=P\left(t,x,x_0\right)\left(\frac{M_n\left(x_0\right)}{M\left(x_0\right)}-1\right)+C_th_n^\gamma \\
	&=P\left(t,x,x_0\right)\frac{\alpha}{\sigma^2\left(x_0\right)}\delta x_0+C_th_n^\gamma.
	\end{align*}
	This concludes the proof.\qed
	
	\bigskip
	\noindent\emph{Proof of Theorem \ref{theorem:discretization_error_convergence}}:
	For $x\in\mathbb{S}_n^-$ we have the following decomposition of the value function
	\begin{flalign}
	&u_n\left(t,x\right)-u\left(t,x\right)=\sum_{y\in\mathbb{S}_n^-}P_n\left(t,x,y\right)f\left(y\right)-\int_{x_0}^{x_{n+1}}P\left(t,x,dy\right)f\left(y\right) \nonumber\\
	&\qquad=P\left(t,x,x_0\right)f\left(x_0\right)\frac{\alpha}{\sigma^2\left(x_0\right)}\delta^+ x_0+C_th_n^\gamma f\left(x_0\right)\nonumber \\
	&\qquad\qquad+\sum_{y\in\mathbb{S}_n^\circ}\left(p_n\left(t,x,y\right)-p\left(t,x,y\right)\right)f\left(y\right)\delta y+\sum_{y\in\mathbb{S}_n^\circ}p\left(t,x,y\right)f\left(y\right)\delta y \nonumber \\
	&\qquad\qquad+\frac{1}{2}p\left(t,x,x_0\right)f\left(x_0\right)\delta^+x_0-\frac{1}{2}p\left(t,x,x_0\right)f\left(x_0\right)\delta^+x_0-\int_{x_0}^{x_{n+1}}p\left(t,x,y\right)f\left(y\right)dy \nonumber\\
	&\qquad=f\left(x_0\right)\frac{\alpha}{\sigma^2\left(x_0\right)}\delta^+x_0M\left(x_0\right)\sum_{k=1}^\infty\exp\left(-\lambda_kt\right)\varphi_k\left(x\right)\varphi_k\left(x_0\right) \nonumber \\
	&\qquad\qquad-f\left(x_0\right)\frac{1}{\sigma^2\left(x_0\right)}\delta^+x_0\sum_{k=1}^\infty\exp\left(-\lambda_kt\right)\varphi_k\left(x\right)\varphi_k\left(x_0\right)+O\left(h_n^\gamma\right)
	&&\mathmakebox[1pt][r]{\left(\textrm{See (a)}\right)} \nonumber \\
	&\qquad\qquad+\sum_{y\in\mathbb{S}_n^\circ}p\left(t,x,y\right)\frac{\mu\left(y\right)}{\sigma^2\left(y\right)}\left(\delta^+y-\delta^-y\right)f\left(y\right)\delta y+O\left(h_n^2\right)
	&&\mathmakebox[1pt][r]{\left(\textrm{See (b)}\right)} \nonumber \\
	&\qquad\qquad+\frac{1}{2}\sum_{y\in\mathbb{S}_n^-}\left(p\left(t,x,y\right)f\left(y\right)+p\left(t,x,y^+\right)f\left(y^+\right)\right)\delta ^+y-\int_{x_0}^{x_{n+1}}p\left(t,x,y\right)f\left(y\right)dy \nonumber \\
	&\qquad=\left(\alpha M\left(x_0\right)-1\right)O\left(h_n\right)+O\left(h_n^\gamma\right)+O\left(h_n^2\right) \label{eq:proof_theorem6_eq0} \\
	&\qquad\qquad+\sum_{y\in\mathbb{S}_n^\circ}p\left(t,x,y\right)\frac{\mu\left(y\right)}{\sigma^2\left(y\right)}\left(\delta^+y-\delta^-y\right)f\left(y\right)\delta y \label{eq:proof_theorem6_eq1} \\
	&\qquad\qquad+\frac{1}{2}\sum_{y\in\mathbb{S}_n^-}\left(p\left(t,x,y\right)f\left(y\right)+p\left(t,x,y^+\right)f\left(y^+\right)\right)\delta^+y-\int_{x_0}^{x_{n+1}}p\left(t,x,y\right)f\left(y\right)dy,\label{eq:proof_theorem6_eq2}
	\end{flalign}
	where
	\begin{enumerate}[(a)]
		\item holds because $p\left(t,x,x_0\right)=\lim_{y\searrow x_0}\ p\left(t,x,y\right)=\frac{2}{\sigma^2\left(x_0\right)}\sum_{k=1}^\infty\exp\left(-\lambda_kt\right)\varphi_k\left(x\right)\varphi_k\left(x_0\right)$.
		\item holds because of Theorem \ref{theorem:convergence_transition_probability}.
		\item Eq. \eqref{eq:proof_theorem6_eq0} holds because
		\begin{equation}
		\frac{1}{\sigma^2\left(x_0\right)}\sum_{k=1}^\infty\exp\left(-\lambda_kt\right)\varphi_k\left(x\right)\varphi_k\left(x_0\right)\leq C_1\sum_{k=1}^\infty\exp\left(-C_2k^2\right)\leq C_3.
		\end{equation}
	\end{enumerate}
	It can be seen that the approximation error for the value function consists of 3 parts.
	The first error is due to the boundary behavior and depends on the scheme.
	The second error results from the discretization error of the transition kernel and the last error is the discretization error of the integral.
	Here we denoted by $O(h_n^\gamma)$ a term that is bounded by $C_th_n^\gamma$ where the constant $C_t$ is independent of $n$, $x$ and $y$ but might depend on $t$ and $f$.
	
	The first term in \eqref{eq:proof_theorem6_eq0} is of order $O(h_n)$ for Scheme 1 and vanished for Scheme 2 as $\alpha M(x_0)=\rho\times 1/\rho=1$.
	The term in \eqref{eq:proof_theorem6_eq1} is of order $O(h_n^2)$ as can be seen in Eq. (23) in \cite{zhang2019}.
	The term in \eqref{eq:proof_theorem6_eq2} is generally of order $O(h_n)$ as seen in Eq. (25) in this reference.
	For call/put type payoffs, the term is of order $O(h_n^2)$.
	
	One can also summarize the above cases by using Eq. (25) and the proof of Theorem 1 in \cite{zhang2019}, such that
	\begin{align}
	&\left\vert u_n\left(t,x\right)-u\left(t,x\right)\right\vert\leq p\left(t,x,\xi\right)\left\vert f\left(\xi-\right)-f\left(\xi+\right)\right\vert\ \left\vert\frac{\xi^-+\xi^+}{2}-\xi\right\vert+C_th_n^\gamma, \\
	&\left\vert u_n\left(t,x\right)-u\left(t,x\right)\right\vert\geq p\left(t,x,\xi\right)\left\vert f\left(\xi-\right)-f\left(\xi+\right)\right\vert\ \left\vert\frac{\xi^-+\xi^+}{2}-\xi\right\vert-D_th_n^\gamma, \nonumber
	\end{align}
	with positive constants $C_t,D_t>0$ independent of $n$ and $x$.
	Taking the maximum on both sides then shows the claim of the theorem.\qed
	
	\bibliographystyle{apalike}
	\bibliography{references}
	
\end{document}